\documentclass[12pt]{amsart}
\usepackage{amssymb}
\usepackage[shortlabels]{enumitem}
\usepackage[all]{xy}
\usepackage{xcolor}
\usepackage[margin=1in]{geometry} 
\usepackage{hyperref}
\usepackage{mathrsfs}
\usepackage{eucal}
\usepackage{contour}
\usepackage{graphicx}
\usepackage{dsfont}
\usepackage[normalem]{ulem}
\hypersetup{bookmarksdepth=2}
\hypersetup{colorlinks=true}
\hypersetup{linkcolor=blue}
\hypersetup{citecolor=blue}
\hypersetup{urlcolor=blue}
\usepackage{tikz}

\numberwithin{equation}{section}
\newtheorem{theorem}[equation]{Theorem}
\newtheorem{proposition}[equation]{Proposition}
\newtheorem{lemma}[equation]{Lemma}
\newtheorem{corollary}[equation]{Corollary}
\newtheorem{maintheorem}{Theorem}

\theoremstyle{definition}
\newtheorem{remark}[equation]{Remark}
\newtheorem{example}[equation]{Example}
\newtheorem{definition}[equation]{Definition}

\newcommand{\cA}{\mathcal{A}}
\newcommand{\sA}{\mathscr{A}}
\newcommand{\cB}{\mathcal{B}}
\newcommand{\fB}{\mathfrak{B}}
\newcommand{\cC}{\mathcal{C}}
\newcommand{\sC}{\mathscr{C}}
\newcommand{\bD}{\mathbf{D}}
\newcommand{\cD}{\mathcal{D}}
\newcommand{\fD}{\mathfrak{D}}
\newcommand{\cE}{\mathcal{E}}
\newcommand{\bF}{\mathbf{F}}
\newcommand{\cF}{\mathcal{F}}
\newcommand{\rH}{\mathrm{H}}
\newcommand{\bI}{\mathbf{I}}
\newcommand{\bK}{\mathbf{K}}
\newcommand{\rK}{\mathrm{K}}
\newcommand{\rL}{\mathrm{L}}
\newcommand{\bN}{\mathbf{N}}
\newcommand{\bP}{\mathbf{P}}
\newcommand{\bQ}{\mathbf{Q}}
\newcommand{\bR}{\mathbf{R}}
\newcommand{\rR}{\mathrm{R}}
\newcommand{\bS}{\mathbf{S}}
\newcommand{\fS}{\mathfrak{S}}
\newcommand{\sS}{\mathscr{S}}
\newcommand{\bT}{\mathbf{T}}
\newcommand{\cT}{\mathcal{T}}
\newcommand{\cX}{\mathcal{X}}
\newcommand{\bZ}{\mathbf{Z}}
\newcommand{\bbG}{\mathbb{G}}
\newcommand{\bbI}{\mathbb{I}}
\newcommand{\bbP}{\mathbb{P}}
\newcommand{\bbQ}{\mathbb{Q}}
\newcommand{\bbS}{\mathbb{S}}
\newcommand{\bbT}{\mathbb{T}}
\newcommand{\fa}{\mathfrak{a}}
\newcommand{\fn}{\mathfrak{n}}

\newcommand{\arxiv}[1]{\href{http://arxiv.org/abs/#1}{{\tiny\tt arXiv:#1}}}

\newcommand{\DOI}[1]{\href{http://doi.org/#1}{\color{purple}{\tiny\tt DOI:#1}}}

\newcommand{\defn}[1]{\emph{#1}}

\contourlength{1pt}

\contourlength{0.8pt}
\newcommand{\myuline}[1]{%
  \uline{\phantom{#1}}%
  \llap{\contour{white}{#1}}%
}
\DeclareMathOperator{\uRep}{\text{\myuline{\rm Rep}}}
\DeclareMathOperator{\uPerm}{\ul{Perm}}

\makeatletter
\newcommand{\bbDelta@scaling}{0.775}\newcommand{\bbDelta@kern}{0.3}
\newcommand{\bbDelta}{\mathord{\mathpalette\bbDelta@\relax\Delta}}
\newcommand{\bbDelta@}[2]{%
  \rlap{\scalebox{\bbDelta@scaling}{$\m@th#1\mkern \bbDelta@kern mu\Delta$}}%
}
\newcommand{\vflip}[1]{\mathord{\mathpalette\vflip@aux{#1}}}
\newcommand{\vflip@aux}[2]{
  \sbox0{$#1#2$}
  \dimen0=\ht0 \advance\dimen0 by -\dp0
  \raisebox{\dimen0}{\scalebox{1}[-1]{\usebox{0}}}
}

\makeatother
\let\stan\bbDelta
\let\cost\bbNabla

\let\ul\underline
\let\ol\overline
\renewcommand{\phi}{\varphi}
\DeclareMathOperator{\tr}{tr}
\DeclareMathOperator{\im}{im}

\DeclareMathOperator{\Rep}{Rep}
\DeclareMathOperator{\End}{End}
\DeclareMathOperator{\Aut}{Aut}
\DeclareMathOperator{\Del}{Del}
\DeclareMathOperator{\Tilt}{Tilt}
\DeclareMathOperator{\Inj}{Inj}
\DeclareMathOperator{\Hom}{Hom}
\DeclareMathOperator{\Ind}{Ind}

\newcommand{\id}{\mathrm{id}}
\renewcommand{\Vec}{\mathrm{Vec}}
\newcommand{\GL}{\mathbf{GL}}
\newcommand{\bbone}{\mathds{1}}

\newcommand{\bone}{\mathbf{1}}

\newcommand{\bb}{{\bullet}}
\newcommand{\ww}{{\circ}}
\newcommand{\Ext}{\mathrm{Ext}}
\newcommand{\Rad}{\mathrm{Rad}}
\DeclareMathOperator{\PSh}{PSh}
\DeclareMathOperator{\Spec}{Spec}
\DeclareMathOperator{\Proj}{Proj}
\DeclareMathOperator{\Tor}{Tor}
\newcommand{\alt}{\mathrm{alt}}
\renewcommand{\inf}{\mathrm{inf}}
\newcommand{\pf}{\mathrm{pf}}
\newcommand{\bDelta}{\boldsymbol{\Delta}}
\newcommand{\bnabla}{\boldsymbol{\nabla}}
\newcommand{\Tens}{\mathrm{Tens}}
\newcommand{\op}{\mathrm{op}}
\newcommand{\he}{\heartsuit}
\newcommand{\RHom}{\rR \hspace{-.6mm}\Hom}
\newcommand{\heart}{\bK(\cA')^\heartsuit}

\title{The second Delannoy category}
\date{January 21, 2026}

\author{Kevin Coulembier}
\address{School of Mathematics and Statistics, University of Sydney, NSW 2006, Australia}
\email{\href{mailto:kevin.coulembier@sydney.edu.au}{kevin.coulembier@sydney.edu.au}}
\urladdr{\url{https://www.maths.usyd.edu.au/u/kevinc/}}
\thanks{KC was supported by ARC grant FT220100125}

\author{Andrew Snowden}
\address{Department of Mathematics, University of Michigan, Ann Arbor, MI}
\email{\href{mailto:asnowden@umich.edu}{asnowden@umich.edu}}
\urladdr{\url{http://www-personal.umich.edu/~asnowden/}}
\thanks{AS was supported by NSF DMS-2301871}

\begin{document}

\begin{abstract}
In recent work, Harman and Snowden constructed a symmetric tensor category associated to an oligomorphic group equipped with a measure. The oligomorphic group $\bbG$ of order preserving automorphisms of the real line admits exactly four measures. The category $\cC$ associated to the first measure is called the (first) Delannoy category; it is semi-simple and pre-Tannakian, with numerous special properties.

In this paper, we study the (non-abelian) category $\cA$ associated to the second measure, which we call the second Delannoy category. We construct a new pre-Tannakian category $\cD$ together with a fully faithful tensor functor $\Psi \colon \cA \to \cD$. The category $\cD$ is the correct ``abelian version'' of the second Delannoy category. Like $\cC$, it has remarkable properties: for instance, it is non-semi-simple, but behaves uniformly in the coefficient field (e.g., it has the same Grothendieck ring and $\Ext^1$ quiver over any field). 

Additionally, we completely solve the problem of understanding how $\cA$ relates to general pre-Tannakian categories. We show that $\cA$ admits exactly two local abelian envelopes: the functor $\Psi$, and a previously constructed functor $\Phi \colon \cA\to\cC$. This is the first case where the local envelopes of a category have been completely determined, outside of cases where there is at most one envelope. This work opens the door to constructing abelian versions of other oligomorphic tensor categories that do not admit a unique envelope.
\end{abstract}

\maketitle
\tableofcontents

\section{Introduction}

Pre-Tannakian categories are an important class of tensor categories generalizing representation categories of algebraic groups (see \S \ref{ss:tencat} for our conventions regarding tensor categories). This class of categories is still mysterious, in large part due to the  difficultly in constructing examples.  In this paper, we construct an interesting new pre-Tannakian $\cD$ called the abelian second Delannoy category. It is a close cousin to the first Delannoy category $\cC$ studied in \cite{line}, which has numerous fascinating properties. The category $\cD$ is constructed from a preliminary additive tensor category $\cA$ introduced by Harman and the second author in \cite{repst}. A novel feature of the present situation is that $\cD$ is not the (unique) abelian envelope of $\cA$, but one of two local abelian envelopes in the sense of the first author \cite{HomKer}.

In this brief introduction, we give a high level summary of the results of this paper, with an emphasis on context and motivation. In the subsequent section, we describe the specifics of our constructions and theorems in more detail.

\subsection{Oligomorphic tensor categories}

In recent work \cite{repst}, Harman and the second author constructed an additive tensor category $\uPerm(G, \mu)$ associated to an oligomorphic group $G$ and a measure $\mu$ valued in a field $k$. When the measure $\mu$ is quasi-regular, and nilpotent endomorphisms have vanishing trace, this category has a (unique) abelian envelope $\uRep(G, \mu)$ in a precise sense discussed below (see \S \ref{ss:intro-env}); this result is \cite[Theorem~13.2]{repst}.

This approach recovers many previously known pre-Tannakian categories, such as the interpolation categories $\uRep(\fS_t)$ and $\uRep(\GL_t(\bF_q))$ studied by Deligne \cite{Deligne} and Knop \cite{Knop}, but has also led to fundamentally new examples, such as the (first) Delannoy category \cite{line}, the circular Delannoy category \cite{circle}, and the arboreal tensor categories \cite{arboreal}. In all of these cases, the relevant measure is quasi-regular. 

When the measure $\mu$ is not quasi-regular, it is not known if there is a suitable abelian version of $\uPerm(G, \mu)$. This is a major outstanding problem in the theory. This paper contains the first piece of progress in this direction.

\subsection{The Delannoy categories}

Let $\bbG=\Aut(\bR, <)$ be the oligomorphic group of order-preserving bijections of the real line. As shown in \cite[\S 16]{repst}, $\bbG$ admits exactly four $k$-valued measures $\mu_1$, $\mu_2$, $\mu_3$, and $\mu_4$. We call $\uPerm(\bbG, \mu_i)^{\rm kar}$ the \defn{$i$th Delannoy category}. We write $\cC$ for the first Delannoy category and $\cA$ for the second. The other two will not play a prominent role in this paper, but the third is equivalent to the second.

The meausure $\mu_1$ is regular, which is a stronger condition than quasi-regular. This, combined with \cite[Theorem~13.2]{repst} and some additional work, implies that $\cC$ is a semi-simple pre-Tannakian category; see \cite[\S 16]{repst}. It is therefore its own abelian envelope. As mentioned above, this category has numerous remarkable properties: for example, it was the first known semi-simple pre-Tannakian category of super-exponential growth in positive characteristic. It was studied in depth in \cite{line}, and has since received additional attention \cite{delmap, delalg, Kriz, KhovanovSnyder, delchar}.

The measure $\mu_2$ is not quasi-regular. Prior to this paper, it was not known if $\cA$ leads to a new pre-Tannakian category. One of our main contributions is a positive resolution of this problem: we construct a new pre-Tannakian category $\cD$ and a fully faithful tensor functor
\begin{displaymath}
\Psi \colon \cA \to \cD.
\end{displaymath}
The category $\cD$ appears to be quite interesting. Of course, it is closely related to $\cC$, which is known to be very special. But it also has some novel features: for example, to the best of our knowledge, it is the first known non-semi-simple pre-Tannakian category that is ``characteristic independent'' (see Remark~\ref{rem:Z}). We note that $\cD$ is a lower finite highest weight category and $\Psi$ identifies $\cA$ with the category of tilting objects.

This example suggests that other non-quasi-regular measures might lead to new pre-Tannakian categories, and sheds some light on how to produce them. There are many known examples of such measures, such as those in \cite{colored, homoperm}, so there are potentially many categories on the cusp of discovery.

\subsection{Local abelian envelopes} \label{ss:intro-env}

Pre-Tannakian categories are often constructed in two steps: first one constructs an additive tensor category $\cE$, and then one constructs a pre-Tannakian category $\cT$ from $\cE$. The theory of local abelian envelopes, due to the second author \cite{HomKer}, organizes the possible $\cT$'s arising from an $\cE$.

Fix a rigid $k$-linear tensor category $\cE$ with $\End(\bbone)=k$. In \cite{HomKer} it is shown that there is a family of faithful tensor functors $\{ \Gamma_i \colon \cE \to \cT_i \}_{i \in I}$, where each $\cT_i$ is pre-Tannakian, such that any faithful functor from $\cE$ to a pre-Tannakian category factors uniquely through a unique $\Gamma_i$. (See \S \ref{ss:local} for a more complete statement.) These $\Gamma_i$'s are called the \defn{local (abelian) envelopes} of $\cE$. If there is a unique local envelope, we call it the \emph{abelian envelope}.

Most recent constructions of pre-Tannakian categories are obtained as abelian envelopes. This includes the examples coming from quasi-regular measures discussed above, as well as some non-oligomorphic examples \cite{BEO, CO, AbEnv, EHS}. If an abelian envelope exists, one might hope it is not too difficult to construct, since it satisfies a simple universal property. Indeed, there is a uniform description in terms of sheaves with respect to a Grothendieck topology that works for most abelian envelopes, see \cite[\S 3.2]{CEOP}. Thus it is not surprising that most known examples are abelian envelopes.

In \cite{delmap}, we showed that the second Delannoy category $\cA$ has at least two local envelopes, by exhibiting two tensor functors $\cA \to \cC$ that must belong to different local envelopes. In particular, $\cA$ does not have an abelian envelope. This perhaps explains why $\cD$ was harder to find than other recent examples of pre-Tannakian categories.

A second main result of this paper is that $\cA$ has exactly two local envelopes, namely, the functor $\Psi \colon \cA \to \cD$ above and one of the functors $\Phi \colon \cA \to \cC$ constructed in \cite{delmap}. This is the first time all local envelopes of a category have been determined, outside of cases where there is at most one local envelope. It is also the first known example where there are multiple, but finitely many, local envelopes. The proof of this result relies on an interesting semi-orthogonal decomposition of the derived category of the pre-sheaf category of $\cA$. This is a new method in the study of local envelopes.

\subsection*{Acknowledgments}

The second author thanks Nate Harman, Catharina Stroppel, and Milen Yakimov for helpful conversations.

\section{Overview of results} \label{s:over}

In the introduction, we gave a high level summary of the paper, focusing on the context and significance of our results. In this section, we get into the details, and provide precise formulations of our main constructions and theorems.

\subsection{Set-up}

Recall that $\cC$ is the first Delannoy category and $\cA$ is the second. The basic objects of these categories are the Schwartz spaces $\sC(\bR^{(n)})$ and $\sA(\bR^{(n)})$ for $n \ge 0$; see \S \ref{sec:permmod}. As mentioned above, in \cite{delmap} we constructed a faithful tensor functor
\begin{displaymath}
\Phi \colon \cA \to \cC, \qquad \sA(\bR) \mapsto \sC(\bR) \oplus \bbone.
\end{displaymath}
This functor will play a crucial role in this paper.

A \defn{weight} is a word in the two letter alphabet $\{\bb, \ww\}$. We let $\Lambda$ denote the set of weights. Weights are the central combinatorial objects in this paper, and will index many different families of modules. In \cite{line}, we showed that the simple objects of $\cC$ are naturally indexed by weights. We write $L_{\lambda}$ for the simple object of $\cC$ indexed by $\lambda \in \Lambda$.

\subsection{The structure of $\cA$}

The main goal of this paper is to construct and classify the local abelian envelopes of $\cA$. Before we can make any substantial progress in this direction, we must first understand the internal structure of $\cA$. We explain our results in this direction.

If $\lambda \in \Lambda$ has length $n$, which we denote by $\ell(\lambda)=n$, then $L_{\lambda}$ occurs with multiplicity one in the Schwartz space $\sC(\bR^{(n)})$; moreover, in \cite{line} we gave an explicit idempotent $E_{\lambda}$ such that $L_{\lambda} = E_{\lambda} \sC(\bR^{(n)})$. Now, the endomorphism algebras of $\sA(\bR^{(n)})$ and $\sC(\bR^{(n)})$ are the same vector space, but with different products. We observe that $E_{\lambda}$ is still idempotent as an endomorphism of $\sA(\bR^{(n)})$, and we define $M_{\lambda} = E_{\lambda} \sA(\bR^{(n)})$. The following theorem tells us the structure of $\cA$ as an additive category.

\begin{maintheorem} \label{mainthm1}
The $M_{\lambda}$'s are the indecomposable objects of $\cA$. There are non-zero maps
\begin{displaymath}
\id \colon M_{\lambda} \to M_{\lambda}, \qquad
d_\lambda \colon M_{\lambda \ww} \to M_{\lambda}, \qquad
u_\lambda \colon M_{\lambda} \to M_{\lambda \bb}, \qquad
u_\lambda\circ d_\lambda \colon M_{\lambda \ww} \to M_{\lambda \bb}.
\end{displaymath}
These are the only non-zero maps (up to scalar multiples) between indecomposable objects.
\end{maintheorem}

A key step in the proof is establishing the decomposition
\begin{displaymath}
\Phi(M_{\lambda}) = L_{\lambda} \oplus L_{\lambda^{\flat}},
\end{displaymath}
where $\lambda^{\flat}$ is $\lambda$ with its final letter deleted; if $\lambda=\varnothing$ then $L_{\lambda^{\flat}}=0$ by convention. The fact that $\Phi(M_{\lambda})$ is so small (length at most two) makes the analysis of $\cA$ easier than it otherwise might have been.

We also give an explicit combinatorial rule for the decomposition of $M_{\lambda} \otimes M_{\mu}$ into indecomposable objects. This is quite similar to the fusion rule in the Delannoy category found in \cite[Corollary~7.3]{line}, though not exactly the same.

\subsection{The pre-sheaf category $\cB$} \label{ss:overview-B}

We are interested in studying maps from $\cA$ to abelian categories. It therefore makes sense to consider the pre-sheaf category $\PSh(\cA)$, i.e., the category of $k$-linear functors from $\cA^{\op}$ to vector spaces. Indeed, this is an abelian category, $\cA$ maps to it via the Yoneda embedding, and this functor has a universal property for maps to abelian categories.

Thanks to Theorem~\ref{mainthm1}, we can give a more concrete description of this category. Define a $k$-linear combinatorial category $\fB$, as follows. The set of objects is the set $\Lambda$ of weights. The morphism spaces
\begin{displaymath}
\Hom_{\fB}(\lambda, \lambda), \quad
\Hom_{\fB}(\lambda, \lambda \ww), \quad
\Hom_{\fB}(\lambda \bb, \lambda), \quad
\Hom_{\fB}(\lambda \bb, \lambda \ww)
\end{displaymath}
are one-dimensional, with distinguished basis vectors, and all other morphisms spaces vanish. There are essentially three compositions of distinguished morphisms to consider:
\begin{displaymath}
\lambda\bb \to \lambda \to \lambda\ww, \qquad
\lambda \to \lambda \ww \to \lambda \ww \ww, \qquad
\lambda\bb\bb \to \lambda\bb \to \lambda.
\end{displaymath}
The first of these is the distinguished morphism, while the other two vanish. The category $\fB$ is anti-equivalent to the category of indecomposable objects in $\cA$ by Theorem~\ref{mainthm1}, and so $\PSh(\cA)$ is equivalent to the category of $\fB$-modules, i.e., $k$-linear functors from $\fB$ to vector spaces. We let $\cB$ denote the category of finite length $\fB$-modules. We tend to refer to this as the ``pre-sheaf category.''

The simple $\fB$-modules are indexed by weights; we let $\bbS_{\lambda}$ denote the simple at $\lambda \in \Lambda$. We let $\bbP_{\lambda}$ denote the projective cover of $\bbS_{\lambda}$. Using the structure of $\fB$, it is easy to see that $\bbP_{\lambda}$ has length two or four; see \S \ref{ss:B-basic}(d). We have an equivalence
\begin{displaymath}
\cA \to \Proj(\cB), \qquad M_{\lambda} \mapsto \bbP_{\lambda},
\end{displaymath}
where $\Proj(-)$ denotes the category of projective objects. From the pre-sheaf perspective, this is simply the Yoneda embedding. Via the above equivalence, we can transfer the tensor product on $\cA$ to $\Proj(\cB)$, and then extend it by cocontinuity to all of $\cB$. This gives $\cB$ the structure of a tensor category, though the tensor product is only right exact and not all objects are rigid. Nonetheless, we at least now have a fully faithful embedding of $\cA$ into an abelian tensor category.

Morphisms in $\fB$ have an obvious ``downwards--upwards'' factorization, which endows $\fB$ with a triangular structure in the sense of \cite{brauercat}. It follows that $\cB$ is a semi-infinite highest weight category in the sense of \cite{BS}. In particular, for each weight $\lambda$ there is a standard module $\stan_{\lambda}$ and a costandard module $\cost_{\lambda}$, both of which have length two (see \S \ref{ss:B-basic}). There is also an indecomposable tilting module $\bbT_{\lambda}$. This has infinite length, though it is multiplicity free and admits a simple explicit description (\S \ref{ss:B-tilt}). We have verified some instances of compatibility between the tensor product and highest weight structure, but it remains unclear whether the ``typical'' compatibility (as in \cite[Proposition~3.10]{FlakeGruber}) holds.

\begin{remark}
We are only able to see the triangular structure on $\cA$ due to Theorem~\ref{mainthm1}. We do not know when the $\uPerm(G, \mu)$ categories should have such a structure, for a general oligomorphic group $G$. This is an interesting problem (and potentially quite important).
\end{remark}

\begin{remark}
One can picture $\fB$ as a quiver with relations, as follows. Start with the Cayley graph for the free monoid on two generators labeled $\ww$ and $\bb$, drawn in the first quadrant of the plane in the usual manner, with $\ww$ edges being horizontal and $\bb$ edges vertical. Direct the edges so that horizontal edges point right and vertical edges point down. This directed graph is the quiver in question. We impose relations so that each row is a complex (i.e., applying two adjacent horizontal maps gives zero), as is each column. We thus see that a $\fB$-module can be viewed as a collection of interrelated complexes on this Cayley graph. It would be interesting if this geometric picture has a deeper connection to our work.
\end{remark}

\subsection{The pre-Tannakian category $\cD$} \label{ss:overview-D}

Since $\cB$ is a highest weight category, it has a Ringel dual. By definition, the Ringel dual is the pre-sheaf category for the category $\Tilt(\cB)$ of tilting modules in $\cB$ (with some finiteness conditions imposed). However, as with the pre-sheaf category of $\cA$, we can give a much more concrete description of this category.

Define a $k$-linear combinatorial category $\fD$ as follows. The set of objects is again the set of weights $\Lambda$. The morphism space $\Hom_{\fD}(\lambda, \mu)$ is one-dimensional (with a distinguished basis vector) if $\lambda=\mu$, or in the following two cases:
\begin{itemize}
\item $\lambda=\mu \nu$, where $\nu$ is an alternating word ending in $\ww$.
\item $\mu=\lambda \nu$, where $\nu$ is an alternating word ending in $\bb$.
\end{itemize}
By an alternating word, we mean one that does not contain $\ww\ww$ or $\bb\bb$ as a substring. In all other cases, the $\Hom$ space vanishes. The composition of two distinguished morphisms is the distinguished morphism, if it exists, and zero otherwise. We show that $\fD$ is anti-equivalent to the category of indecomposable tilting modules in $\cB$.

We define $\cD$ to be the category of finite length $\fD$-modules. From the above discussion, we see that $\cD$ is (equivalent to) the Ringel dual of $\cB$ and, as such, has a highest weight structure. We let $\bT_{\lambda}$ be the indecomposable tilting module in $\cD$ indexed by the weight $\lambda$. We show that there is an equivalence of categories
\begin{displaymath}
\Psi \colon \cA \to \Tilt(\cD), \qquad M_{\lambda} \mapsto \bT_{\lambda}.
\end{displaymath}
Via this equivalence, we can transfer the tensor product on $\cA$ to $\Tilt(\cD)$. The following important theorem establishes the tensor structure on $\cD$.

\begin{maintheorem} \label{mainthm2}
The tensor product on $\Tilt(\cD)$ extends uniquely to an exact tensor product on $\cD$, which makes $\cD$ into a pre-Tannakian category.
\end{maintheorem}

Thus $\Psi \colon \cA \to \cD$ is a fully faithful tensor functor from $\cA$ to a pre-Tannakian category.

\begin{remark}
It took substantial effort to reach the category $\cD$: we first needed the oligomorphic category $\cA$, then we needed to prove Theorem~\ref{mainthm1} to obtain a highest weight structure on the pre-sheaf category $\cB$, and finally we needed to analyze tilting modules in $\cB$ to obtain the combinatorial category $\fD$. A posteriori, one can define $\fD$ directly, as we did above, and obtain the abelian category $\cD$ with very little effort. However, we only know how to construct the tensor structure on $\cD$ by leveraging the equivalence $\cA=\Tilt(\cD)$. Thus much of this work is still needed to obtain $\cD$ as a tensor category. In the body of the paper, we actually treat $\cD$ before $\cB$, but the work on $\cA$ is still prerequisite.
\end{remark}

\begin{remark}\label{rem:RingelEnv}
In many of the currently understood abelian envelopes $\cE\to\cT$, the pre-Tannakian category $\cT$ has a highest weight structure for which $\cE$ is the category of tilting modules; see, e.g., \cite{CO, EHS} or \cite[Propositions~7.3.1 and~7.4.1]{CEOP}. It then follows that $\cT$ can be obtained as the Ringel dual of (an appropriate finite version of) $\PSh(\cE)$; see \cite{BS}.

A systematic treatment of this recurring principle is provided in \cite{FlakeGruber}, which appeared in the finishing stages of the current paper. Our proof that $\cD$ has a rigid monoidal structure does not quite fit the general approach of \cite[Theorem~D]{FlakeGruber}, since we work more directly with $\cD$ rather than its Ringel dual. We only bring the latter into the picture after $\cD$ is established, in order to deal with the local abelian envelope question.

Nonetheless, this general theme is the reason that we were led to consider the Ringel dual $\cD$ of $\cB$. Crucially, the present case differs from previous ones since $\cD$ is only one of multiple (two) local abelian envelopes. This accounts for some of the difficulty we encounter.
\end{remark}

\subsection{Local abelian envelopes}

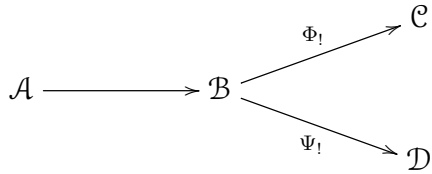
\begin{figure}
\begin{displaymath}
\xymatrix@R=1em@C=5em{
&& \cC \\
\cA \ar[r] & \cB \ar[ru]^-{\Phi_!} \ar[rd]_-{\Psi_!} \\
&& \cD }
\end{displaymath}
\caption{The four main categories and their relationships: $\cA$ is the second Delannoy category, $\cB$ is the pre-sheaf category of $\cA$, $\cC$ is the first Delannoy category, and $\cD$ is the new tensor category constructed in this paper, which we call the abelian second Delannoy category. The categories $\cC$ and $\cD$ are pre-Tannakian, while $\cA$ and $\cB$ are not.}
\end{figure}

We now have two faithful tensor functors
\begin{displaymath}
\Phi \colon \cA \to \cC, \qquad
\Psi \colon \cA \to \cD
\end{displaymath}
from $\cA$ to pre-Tannakian categories. We prove:

\begin{maintheorem} \label{mainthm3}
The category $\cA$ has exactly two local envelopes, namely $\Phi$ and $\Psi$.
\end{maintheorem}

Theorem~\ref{mainthm3} is a consequence of structural results we establish for the bounded derived category $\bD(\cB)$, which we now explain. Let $\Theta \colon \cA \to \Vec$ be the semi-simplification functor (\S \ref{ss:ss}); it is quite degenerate, e.g., each $\sA(\bR^{(n)})$ with $n \ge 1$ is sent to zero. Identifying $\cA$ with $\Proj(\cB)$, the functors $\Phi$, $\Psi$, and $\Theta$ admit unique right-exact extensions
\begin{displaymath}
\Phi_! \colon \cB \to \cC, \qquad \Psi_! \colon \cB \to \cD, \qquad \Theta_! \colon \cB \to \Vec.
\end{displaymath}
Each of these functors admits a left-derived functor. A priori, the derived functors exist on the bounded above derived categories, but we show they preserve the bounded derived categories. The following is the most basic structural result for $\bD(\cB)$.

\begin{maintheorem} \label{mainthm4}
The category $\bD(\cB)$ admits a semi-orthogonal decomposition $\langle \cB_1, \cB_2, \cB_3 \rangle$. The pieces $\cB_1$, $\cB_2$, and $\cB_3$ are equivalent to $\bD(\Vec)$, $\bD(\cC)$, and $\bD(\cD)$ via $\rL \Theta_!$, $\rL \Phi_!$, and $\rL \Psi_!$.
\end{maintheorem}

The proof of Theorem~\ref{mainthm4} is fairly straightforward: we compute the above derived functors and $\Ext$ groups for certain modules using explicit resolutions, and construct some explicit triangles relating simple objects of $\cB$ to certain objects in the $\cB_i$'s. Discovering the correct statements was more difficult than finding their proofs.

There is much more to say about this situation, which we spell out in \S \ref{s:B2}. For now, we mention one other important aspect. Recall that $\cB$ carries a right-exact tensor product. The derived tensor product endows the bounded above derived category $\bD^+(\cB)$ with a tensor structure. We show that $\bD(\cB)$, as well as each $\cB_i$, is closed under tensor product, and that $\cB_i$ and $\cB_j$ (for $i \ne j$) are orthogonal under the tensor product. Moreover, the functors $\Phi$, $\Psi$, and $\Theta$, as well as their derived functors, are tensor functors, and the equivalences in the above theorem hold as tensor categories.

We explain the basic idea in the proof of Theorem~\ref{mainthm3}. Let $F \colon \cA \to \cT$ be a tensor functor, with $\cT$ pre-Tannakian. Then $F$ admits a unique right-exact extension $F_! \colon \cB \to \cT$, and we can consider its left-derived functor $\rL F_! \colon \bD(\cB) \to \bD^+(\cT)$. Since non-zero objects in a pre-Tannakian category have non-zero tensor product, it follows that $\rL F_!$ must kill two of the $\cB_i$'s, and so it factors through one of $\rL \Phi_!$, $\rL \Psi_!$, or $\rL \Theta_!$. More work is still needed to complete the proof of Theorem~\ref{mainthm3}, but this is the key idea.

\begin{remark}
If, in the definition of local envelope, one dropped the faithful assumption then $\cA$ would have three local envelopes, namely, $\Phi$, $\Psi$, and $\Theta$. It is very interesting that $\bD(\cB)$ is, in a sense, composed of the three local envelopes. It is not clear to us if this should be a more general phenomenon.
\end{remark}

\subsection{Problems}

We mention a few problems related to our work.

\begin{itemize}
\item An interesting feature of the abelian envelope $\uRep(G, \mu)$ constructed in \cite{repst} in the quasi-regular case is that it is concrete: the objects are certain modules over a ring naturally associated to $G$ (the completed group algebra). Is there a similar concrete interpretation of our category $\cD$?
\item A natural problem is to extend our work to the fourth Delannoy category. (The third Delannoy category is equivalent to the second one.) We expect our general approach should apply, but we also expect the fourth category to be more complicated than the second. In particular, we expect four local envelopes: one equivalent to $\cC$, two to $\cD$, and one that is new.
\item There has been much recent progress on the Delannoy category $\cC$. In \cite{delalg}, its simple commutative algebras are classified. In \cite{delchar}, it is characterized by its Grothendieck group and by the triviality of Adams operations. And in \cite{KhovanovSnyder}, a diagrammatic description is given. It would be interesting to extend results like these to $\cD$.
\item Kriz \cite{Kriz} defined a $q$-variant of the Delannoy category $\cC$. Is there a $q$-variant of our pre-Tannakian category $\cD$?
\item The field of tensor triangulated categories has been very active in recent years. It would be interesting to study $\bD(\cB)$ in the context of this general theory. For instance, what is its Balmer spectrum? 
\item Suppose $G$ is an oligomorphic group with a measure $\mu$. Under what conditions does $\uPerm(G, \mu)$ admit a triangular structure? For instance, does this hold for the measures found in \cite{colored, homoperm}?
\end{itemize}

\subsection{Tensor category terminology} \label{ss:tencat}

We fix a field $k$ for the duration of the paper. A \defn{tensor category} is an additive $k$-linear category equipped with a $k$-bilinear symmetric monoidal structure. A \defn{tensor functor} is a $k$-linear symmetric monoidal functor. We write $\bbone$ for the unit object of a tensor category. A \defn{pre-Tannakian category} is an abelian tensor category such that all objects have finite length, all $\Hom$ spaces are finite dimensional, every object is rigid (i.e., admits a dual), and $\End(\bbone)=k$.

\subsection{Notation} \label{ss:not}

We list some of the important notation here. Basics:
\begin{description}[align=right,labelwidth=2.25cm,leftmargin=!]
\item[ $k$ ] the coefficient field
\item[$\bN$ ] the natural numbers, including $0$
\item[ $\bbone$ ] unit object of a tensor category
\item[ $\Vec$ ] the category of finite dimensional $k$-vector spaces
\end{description}
The main categories:
\begin{description}[align=right,labelwidth=2.25cm,leftmargin=!]
\item[ $\cA$ ] the (rigid non-abelian) second Delannoy category (\S \ref{ss:delcat})
\item[ $\cB$ ] the (abelian non-rigid) category of $\fB$-modules, or pre-sheaves on $\cA$ (\S \ref{ss:B-mod})
\item[ $\cC$ ] the (semi-simple pre-Tannakian) first Delannoy category (\S \ref{ss:delcat})
\item[ $\cD$ ] the (pre-Tannakian) category of $\fD$-modules, or the Ringel dual of $\cB$ (\S \ref{def:Dmod})
\end{description}
Functors:
\begin{description}[align=right,labelwidth=2.25cm,leftmargin=!]
\item[ $\Phi$ ] the functor $\cA \to \cC$ (\S \ref{ss:functor})
\item[ $\Psi$ ] the functor $\cA \to \cD$ (\S \ref{ss:Dtilt})
\item[ $\Theta$ ] the semi-simplification functor $\cA \to \Vec$ (\S \ref{s:classindec})
\item[ $\Phi_!$ ] the right-exact extension $\cB \to \cC$ of $\Phi$ (\S \ref{ss:Phi-shriek})
\item[ $\Psi_!$ ] the right-exact extension $\cB \to \cD$ of $\Psi$ (\S \ref{ss:Psi-shriek})
\item[ $\Theta_!$ ] the right-exact extension $\cB \to \Vec$ of $\Theta$ (\S \ref{ss:Theta-shriek})
\end{description}
Oligomorphic theory:
\begin{description}[align=right,labelwidth=2.25cm,leftmargin=!]
\item[ $\bbG$ ] the oligomorphic group $\Aut(\bR, <)$ (\S \ref{ss:delgp})
\item[ $\bR^{(n)}$ ] the set of increasing tuples $(x_1, \ldots, x_n)$ in $\bR^n$
\item[ $p_{n,i}$ ] the projection $\bR^{(n)} \to \bR^{(n-1)}$ that omits the $i$th coordinate
\item[ $\sC(\bR^{(n)})$ ] the Schwartz space of $\bR^{(n)}$ in $\cC$ (\S \ref{ss:delcat})
\item[ $\sA(\bR^{(n)})$ ] the Schwartz space of $\bR^{(n)}$ in $\cA$ (\S \ref{ss:delcat})
\end{description}
Weights:
\begin{description}[align=right,labelwidth=2.25cm,leftmargin=!]
\item[ $\Lambda$ ] the set of weights (words in $\bb$ and $\ww$)
\item[ $\Lambda_{\ww}$ ] the set of non-empty weights with final letter $\ww$
\item[ $\Lambda_{\bb}$ ] the set of non-empty weights with final letter $\bb$
\item[ $\Lambda^{\alt}$ ] the set of alternating weights (no $\ww\ww$ or $\bb\bb$ substring)
\item[ $\lambda$ ] a weight
\item[ $\lambda^{\flat}$ ] the weight obtained from $\lambda$ by deleting the final letter
\item[ $\lambda^{\vee}$ ] the dual weight to $\lambda$ (change $\ww$ to $\bb$ and vice versa)
\end{description}
Modules parametrized by weights:
\begin{description}[align=right,labelwidth=2.25cm,leftmargin=!]
\item[ $L_{\lambda}$ ] the simple of $\cC$ indexed by $\lambda$ (\S \ref{ss:delcat})
\item[ $M_{\lambda}$ ] the indecomposable of $\cA$ indexed by $\lambda$ (\S \ref{ss:indecomp-constr})
\item[ $\bS_{\lambda}$ ] the simple of $\cD$ indexed by $\lambda$ (\S \ref{def:Dmod})
\item[ $\bDelta_{\lambda}$ ] the standard of $\cD$ indexed by $\lambda$ (\S \ref{ss:Dstd})
\item[ $\bnabla_{\lambda}$ ] the costandard of $\cD$ indexed by $\lambda$ (\S \ref{ss:Dstd})
\item[ $\bT_{\lambda}$ ] the tilting module of $\cD$ indexed by $\lambda$ (\S \ref{ss:Dtilt})
\item[ $\bbS_{\lambda}$ ] the simple of $\cB$ indexed by $\lambda$ (\S \ref{ss:B-basic})
\item[ $\bbP_{\lambda}$ ] the projective of $\cB$ indexed by $\lambda$ (\S \ref{ss:B-basic})
\item[ $\bbI_{\lambda}$ ] the injective of $\cB$ indexed by $\lambda$ (\S \ref{ss:B-basic})
\item[ $\stan_{\lambda}$ ] the standard of $\cB$ indexed by $\lambda$ (\S \ref{ss:B-basic})
\item[ $\cost_{\lambda}$ ] the costandard of $\cB$ indexed by $\lambda$ (\S \ref{ss:B-basic})
\item[ $\bbQ_{\lambda}$ ] the module $\bbP_{\lambda\bb}/\bbS_{\lambda \bb\ww}$ (\S \ref{ss:B-basic})
\item[ $\bbT_{\lambda}$ ] the tilting $\fB$-module indexed by $\lambda$ (\S \ref{ss:B-tilt})
\end{description}
By convention, if $\lambda$ is the empty weight then $L_{\lambda^{\flat}}$ is~0; similarly for other objects parametrized by weights. We note that we have tensor equivalences
\begin{displaymath}
\Proj(\cB) \cong \cA \cong \Tilt(\cD), \qquad
\bbP_{\lambda} \leftrightarrow M_{\lambda} \leftrightarrow \bT_{\lambda}.
\end{displaymath}
Some additional notation:
\begin{description}[align=right,labelwidth=2.25cm,leftmargin=!]
\item[ $\bD(-)$ ] the bounded derived category
\item[ $\bD^+(-)$ ] the bounded above derived category
\item[ $\bK(-)$ ] the bounded homotopy category
\item[ $\rK(-)$ ] the Grothendieck group
\item[ $(-)^*$ ] the dual in a monoidal category
\item[ $(-)^{\vee}$ ] the pointwise duality on $\cD$ (\S \ref{def:Dmod}) and $\cB$ (\S \ref{ss:B-mod})
\end{description}

\section{Oligomorphic tensor categories} \label{s:oligo}

In this section, we briefly review the general theory of oligomorphic tensor categories from \cite{repst}. A more detailed overview can be found in \cite[\S 2]{line}.

\subsection{Oligomorphic groups}

An \defn{oligomorphic group} is a permutation group $(G, \Omega)$ such that $G$ has finitely many orbits on $\Omega^n$ for all $n \ge 0$. Fix such a group. For a finite subset $A$ of $\Omega$, let $G(A)$ be the subgroup of $G$ fixing each element of $A$. These subgroups form a neighborhood basis for a topology on $G$. This topology has the following properties \cite[\S 2.2]{repst}: it is Hausdorff; it is non-archimedean, i.e., open subgroups form a neighborhood basis of the identity; and it is Roelcke precompact, i.e., if $U$ and $V$ are open subgroups then $U \backslash G/V$ is a finite set. A topological group with these three properties is called \defn{pro-oligomorphic}. While most pro-oligomorphic groups of interest are in fact oligomorphic, working in the pro-oligomorphic setting can be clearer since many concepts depend only on the topology and not $\Omega$.

Fix a pro-oligomorphic group $G$. An action of $G$ on a set $X$ is \defn{smooth} if every point in $X$ has open stabilizer in $G$, and \defn{finitary} if $G$ has finitely many orbits on $X$. We use the term ``$G$-set'' to mean ``set equipped with a finitary smooth $G$-action.'' Let $\bS(G)$ be the category of $G$-sets. This category is closed under finite products \cite[\S 2.3]{repst}.

Let $X$ be a $G$-set. A \defn{$\hat{G}$-subset} of $X$ is a subset that is stable under some open subgroup of $G$. More generally, one can define a $\hat{G}$-set without an ambient $G$-set \cite[\S 2.5]{repst}. The symbol $\hat{G}$ has no meaning on its own, but intuitively one can think of it as an infinitesimal neighborhood of the identity in $G$.

\subsection{Measures}

Fix a pro-oligomorphic group $G$ and a field $k$. We require the notion of measure introduced in \cite{repst}. There are two equivalent definitions, which are both useful:

\begin{definition}
A $k$-valued \defn{measure} for $G$ is a rule assigning to each finitary $\hat{G}$-set $X$ a quantity $\mu(X)$ in $k$ such that:
\begin{enumerate}
\item If $X$ and $Y$ are isomorphic then $\mu(X)=\mu(Y)$.
\item If $X$ is a singleton set then $\mu(X)=1$.
\item If $X$ is a $\hat{G}$-set and $g \in U$ then $\mu(X^g)=\mu(X)$; here $X^g$ is the set $X$ with the action of $\hat{G}$ twisted by conjugation by $g$.
\item We have $\mu(X \amalg Y)=\mu(X)+\mu(Y)$.
\item Suppose $Y \to X$ is a map of transitive $U$-sets, for some open subgroup $U$, and let $F$ be the fiber above some point. Then $\mu(Y)=\mu(F) \cdot \mu(X)$.
\end{enumerate}
\end{definition}

\begin{definition}
A $k$-valued \defn{measure} for $G$ is a rule assigning to each morphism $f \colon Y \to X$ of transitive $G$-sets a quantity $\mu(f)$ in $k$ such that:
\begin{enumerate}
\item If $f$ is an isomorphism then $\mu(f)=1$.
\item We have $\mu(g \circ f)=\mu(g) \circ \mu(f)$ when defined.
\item Let $f$ be as above, let $X' \to X$ be another morphism of transitive $G$-sets, and let $f' \colon Y' \to X'$ be the base change of $f$. Let $Y'=\bigsqcup_{i=1}^n Y'_i$ be the orbit decomposition of $Y'$, and let $f'_i$ be the restriction of $f'$ to $Y'_i$. Then $\mu(f)=\sum_{i=1}^n \mu(f'_i)$.
\end{enumerate}
\end{definition}

We briefly explain the equivalence between the two definitions. Given a measure $\mu$ in the first sense, we obtain one in the second send by defining $\mu(f)$ to be $\mu(f^{-1}(x))$ for any point $x \in X$. Now suppose we have a measure $\mu$ in the second sense. Given a transitive $U$-set $X$, we define $\mu(X)=\mu(f)$, where $f \colon G \times_U X \to G/U$ is the natural map (which is a map of transitive $G$-sets). We then extend this definition additively to general $\hat{G}$-sets. We refer to \cite[\S 3]{repst} for more details on measures.

\subsection{Integration}

Maintain the above setting, and fix a measure $\mu$. Let $X$ be a $G$-set. We say that a function $\phi \colon X \to k$ is a \defn{Schwartz function} if it is smooth, i.e., invariant under translation by some open subgroup. We let $\sS(X)$ be the space of Schwartz functions, which is a $k$-vector space.

Let $\phi \in \sS(X)$ be given. Then $\phi$ assumes finitely many non-zero values, say $a_1, \ldots, a_n$. Let $A_i=\phi^{-1}(a_i)$, which is a finitary $\hat{G}$-set. We define the integral of $\phi$ by
\begin{displaymath}
\int_X \phi(x) dx = \sum_{i=1}^n a_i \mu(A_i).
\end{displaymath}
This is easily seen to be well-defined. Integration defines a $k$-linear map $\sS(X) \to k$. See \cite[\S 4.3]{repst} for details.

More generally, if $f \colon Y \to X$ is a map of $G$-sets then we define $f_* \colon \sS(Y) \to \sS(X)$ by integrating over the fibers. This has all the expected properties of a push-forward operation; see \cite[\S 4.4]{repst} for details. There is also a pull-back $f^\ast \colon \sS(X)\to\sS(Y)$, given by $\varphi\mapsto \varphi\circ f$. Contrary to $f_\ast$, the pull-back does not depend on the measure. 

\subsection{Matrices}\label{SSMatrices}

Let $X$ and $Y$ be finitary $G$-sets. A \defn{$Y \times X$ matrix} is a Schwartz function $A \colon Y \times X \to k$. If $B$ is a $Z \times Y$ matrix then we define the product $BA$ to be the $Z \times X$ matrix given by
\begin{displaymath}
(BA)(z,x) = \int_Y B(z,y) A(y,x) dy.
\end{displaymath}
This has the expected properties of matrix multiplication, e.g., it is associative and the $X \times X$ identity matrix (defined by $I(x,x')=\delta_{x,x'}$) is the identity. See \cite[\S 7]{repst} for details.

Let $A$ be a $Y \times X$ matrix. A Schwartz function $\phi \in \sS(X)$ can be regarded as an $X \times \bone$ matrix (with $\bone$ a singleton), or a ``column vector'' of height $X$. Thus the product $A \phi \in \sS(Y)$ is defined. We thus see that multiplication by $A$ defines a linear map
\begin{displaymath}
A \colon \sS(X) \to \sS(Y).
\end{displaymath}
The matrix $A$ is determined by this map \cite[Proposition~7.5]{repst}. If $B$ is a $Z \times Y$ matrix then the composition of the maps associated to $B$ and $A$ is the map associated to the product matrix $BA$.

Let $f \colon Y \to X$ be a map of finitary $G$-sets.
As observed in \cite[\S 7.7]{repst}, there are matrices such that multiplication realizes the maps $f^\ast$ and $f_\ast$. By slight abuse of notation, we denote these matrices again by $f^\ast$ and $f_\ast$. The former is the $Y\times X$-matrix that sends $(y,x)$ to $\delta_{x,f(y)}$ and the latter is the $X\times Y$-matrix that sends $(x,y)$ to $\delta_{x,f(y)}$. We can then write
\begin{displaymath}
f^\ast v=f^\ast(v)\quad\mbox{and}\quad f_\ast w=f_\ast(w),
\end{displaymath}
for $v\in \sS(X)$ and $w\in\sS(Y)$, where the left-hand side of the equations denotes matrix multiplication.

\subsection{Permutation modules}\label{sec:permmod}
We now define a category $\uPerm(G, \mu)$. The objects are labeled by the finitary $G$-sets $X$, and it is convenient to give them the same notation $\sS(X)$ as their Schwartz space. The morphisms $\sS(X) \to \sS(Y)$ are the $G$-invariant $Y \times X$ matrices (or the linear transformations they define). Composition is given by matrix multiplication (or composition of functions). This category has a direct sum $\oplus$ and a tensor product $\otimes$ defined on objects by
\begin{displaymath}
\sS(X) \oplus \sS(Y) = \sS(X \amalg Y), \qquad
\sS(X) \otimes \sS(Y) = \sS(X \times Y).
\end{displaymath}
On morphisms these operations are given by natural analogs of the usual constructions on matrices (block matrices and Kronecker products).

The category $\uPerm(G, \mu)$ is a tensor category with the above structure. Every object of $\uPerm(G, \mu)$ is self-dual \cite[Proposition~8.6]{repst}, and so $\uPerm(G, \mu)$ is a rigid tensor category. The (categorical) dimension of the object $\sS(X)$ is $\mu(X)$. We note that a $G$-invariant $Y \times X$ matrix is essentially a function on the finite set $G \backslash (Y \times X)$, and so the $\Hom$ spaces in $\uPerm(G, \mu)$ are finite dimensional.

When we denote a specific instance of $\uPerm(G,\mu)$ by $\cA$ or $\cC$, we will replace the symbol $\sS(X)$ accordingly with $\sA(X)$ or $\sC(X)$ (see \S \ref{ss:delcat}).

\section{Delannoy categories} \label{s:del}

In this section we introduce the Delannoy categories and recall some known results about them. Essentially everything in this section can be found in \cite{repst}, \cite{line}, or \cite{delmap}.

\subsection{The group} \label{ss:delgp}

Let $\bbG=\Aut(\bR, <)$ be the group of all order preserving self-bijections of the real line $\bR$. This group is oligomorphic via its action on $\bR$.

We now describe the structure of the category $\bS(\bbG)$. For $n\in \bN$, let $\bR^{(n)}$ be the subset of $\bR^n$ consisting of increasing tuples $(x_1, \ldots, x_n)$, i.e., $x_1<x_2<\cdots<x_n$. By convention $\bR^{(0)}$ is a singleton. The group $\bbG$ acts on $\bR^{(n)}$, and one sees that the action is smooth and transitive. In fact, every transitive $\bbG$-set is isomorphic to $\bR^{(n)}$ for some $n$ \cite[Corollary~16.2]{repst}. If $ [m] \to [n]$ is an order-preserving injection of finite sets then the induced projection $ \bR^{(n)} \to \bR^{(m)}$ is a map of $\bbG$-sets, and one easily sees that all maps are of this form. Thus the category of transitive objects in $\bS(\bbG)$ is anti-equivalent to the category of finite totally ordered sets.

A $\hat{\bbG}$-subset of $\bR^{(n)}$ is one that can be defined by a first-order formula in the language of total orders using some number of constants. For example, the $\hat{\bbG}$-subsets of $\bR$ are finite unions of points and open intervals.

\subsection{Measures} \label{ss:four-meas}

The group $\bbG$ has exactly four $k$-valued measures $\mu_1, \mu_2, \mu_3, \mu_4$, as shown in \cite[\S 16]{repst}. We first describe $\mu_1$. For a map $f \colon \bR^{(n)} \to \bR^{(m)}$ of transitive $\bbG$-sets, we have $\mu_1(f)=(-1)^{n+m}$. For a $\hat{\bbG}$-subset $X$ of $\bR^{(n)}$, the quantity $\mu_1(X)$ is equal to $\chi_c(X)$, the Euler characteristic of compactly supported cohomology, where $X$ is endowed with the subspace topology. For instance, if $X \subset \bR$ is an open interval then $\mu_1(X)=-1$, if $X$ is a half-open interval then $\mu_1(X)=0$, and if $X$ is a closed interval then $\mu_1(X)=1$.

We now describe the measure $\mu_2$. Consider a map $f \colon \bR^{(n)} \to \bR^{(m)}$ of $\bbG$-sets. As discussed, $f$ is a projection onto some subset of coordinates, corresponding to an order preserving injection $i \colon [m] \to [n]$. If $f$ remembers the final coordinate (i.e., $n$ is in the image of $i$) then $\mu(f)=(-1)^{n+m}$; if $f$ discards the final coordinate then $\mu(f)=0$. Now suppose that $X$ is a $\hat{\bbG}$-subset of $\bR^{(n)}$. Let $\ol{\bR}=\bR \cup \{\infty\}$, equipped with the topology homeomorphic to $(0,1]$. Then $\mu_2(X)$ is equal to $\chi_c(\ol{X})$, where $\ol{X}$ is the closure of $X$ in $\ol{\bR}{}^{(n)}$ (defined in the obvious manner). For instance, if $X \subset \bR$ is a bounded open interval then $\mu_2(X)=-1$, while if $X$ is an open interval that is unbounded above then $\mu_2(X)=0$.

The measure $\mu_3$ is similar to $\mu_2$, but adds $-\infty$ instead of $+\infty$. The measure $\mu_4$ adds both $-\infty$ and $+\infty$. So, for example, we have $\mu_4(\bR)=1$.

We can characterize all the measures succinctly as follows. Let $p_{2,i} \colon \bR^{(2)} \to \bR$ be the projection omitting the $i$th coordinate for $i=1,2$. The values of the measures on these morphisms are indicated in the following table:
\begin{center}
\begin{tabular}{l|cccc}
& $\mu_1$ & $\mu_2$ & $\mu_3$ & $\mu_4$ \\
\hline
$p_{2,1}$ & $-1$ & $-1$ &  $0$ & 0 \\
$p_{2,2}$ & $-1$ &  $0$ & $-1$ & 0
\end{tabular}
\end{center}
The measures are also completely determined by their values on sets $I^{(n)}$, where the $I$ are open intervals in $\bR$. For example, we have
\begin{displaymath}
\mu_1(I^{(n)})=(-1)^n\quad\mbox{and}\quad \mu_2(I^{(n)})=\begin{cases}(-1)^n&\mbox{if $I$ is bounded from above}\\
0&\mbox{otherwise}.
\end{cases}
\end{displaymath}

\subsection{The Delannoy categories} \label{ss:delcat}

We refer to the categories $\uPerm(\bbG, \mu_i)^{\rm kar}$ for $1 \le i \le 4$ as the \defn{Delannoy categories}. We are interested in the first two, and introduce names for them:
\begin{displaymath}
\cC = \uPerm(\bbG, \mu_1)^{\rm kar}, \qquad
\cA = \uPerm(\bbG, \mu_2)^{\rm kar}.
\end{displaymath}
This notation will be in effect for the remainder of the paper. For a $\bbG$-set $X$, we write $\sC(X)$ for the Schwartz space in $\cC$, and $\sA(X)$ for the Schwartz space in $\cA$.

The first Delannoy category $\cC$ is a semi-simple pre-Tannakian category, for any choice of coefficient field $k$ \cite[Theorem~16.10]{repst}. We recall the classification of simple objects. A \defn{weight} is a word in the two letter alphabet $\{\bb,\ww\}$. We write $\Lambda$ for the set of all weights, and $\ell(\lambda)$ for the length of a weight $\lambda$. Given a weight $\lambda$ of length $n=\ell(\lambda)$, there is a direct summand $L_{\lambda}$ of $\sC(\bR^{(n)})$, as defined in \cite[\S 4.1]{repst}. These objects are simple \cite[Theorem~4.3]{line} and account for all simple objects \cite[Corollary~4.12]{line}. The dimension of $L_{\lambda}$ is $(-1)^n$ \cite[Corollary~5.7]{line}. The dual $L_{\lambda}^*$ of $L_{\lambda}$ is the simple $L_{\lambda^{\vee}}$, where $\lambda^\vee$ is obtained from $\lambda$ by interchanging $\bb$ and $\ww$. The simple decomposition of Schwartz space is given by
\begin{displaymath}
\sC(\bR^{(n)}) = \bigoplus_{\ell(\lambda) \le n} L_{\lambda}^{\oplus m(\lambda)}, \qquad m(\lambda)=\binom{n}{\ell(\lambda)},
\end{displaymath}
see \cite[Theorem~4.7]{line}. In particular, $\sC(\bR^{(n)})$ has length $3^n$.

\subsection{Examples of matrix multiplication} \label{ss:matrix-ex}

The two vector spaces
\begin{displaymath}
\Hom_{\cC}(\sC(\bR^{(m)}), \sC(\bR^{(n)})), \qquad
\Hom_{\cA}(\sA(\bR^{(m)}), \sA(\bR^{(n)}))
\end{displaymath}
are equal: they are both the space of $\bbG$-invariant $k$-valued functions on $\bR^{(n)} \times \bR^{(m)}$. The only difference between the categories $\uPerm(\bbG, \mu_1)$ and $\uPerm(\bbG, \mu_2)$ is that products of matrices are computed differently. It will be useful to spell out some simple examples in detail.

Consider the case $n=m=1$. We define two $\bR \times \bR$ matrices:
\begin{displaymath}
A(y,x) = \begin{cases}
1 & \text{if $y \le x$} \\ 
0 & \text{otherwise} \end{cases}
\qquad
B(y,x) = \begin{cases}
1 & \text{if $x \le y$} \\ 
0 & \text{otherwise.} \end{cases}
\end{displaymath}
We compute products with respect to the measure $\mu_i$, for $1 \le i \le 4$. To begin, we have
\begin{displaymath}
(A^2)(z,x) = \int_{\bR} A(z,y) A(y,x) dy = \mu_i \{ y \in \bR \mid \text{$z \le y$ and $y \le x$} \}.
\end{displaymath}
The set we are taking the measure of is bounded, and thus behaves the same with respect to all the measures. If $z \le x$ then the measure is~1 and otherwise it is~0. We thus find $A^2=A$ for all $\mu_i$. Similarly, $B^2=B$ for all $\mu_i$.

Next we consider the product $AB$. We have
\begin{displaymath}
(AB)(z,x) = \int_{\bR} A(z,y) B(y,x) dy =  \mu_i \{ y \in \bR \mid \text{$z \le y$ and $x \le y$} \}.
\end{displaymath}
Here we are taking the measure of a set of the form $[a, \infty)$. For $\mu_1$ and $\mu_3$ such intervals have measure~0, and so $AB=0$. For $\mu_2$ and $\mu_4$ such intervals have measure~1, and so $AB$ is the constant function~1; this can be expressed as $A+B-1$, where 1 is the $\bR \times \bR$ identity matrix. The product $BA$ is similar, but with the roles of $\mu_2$ and $\mu_3$ switched.

The above computations show that products in the endomorphism algebras of $\sC(\bR)$ and $\sA(\bR)$ are not the same. In fact, $\End(\sC(\bR))$ is isomorphic to $k^3$ as an algebra, while $\End(\sA(\bR))$ is non-commutative and has non-zero radical, it is the path algebra of the one arrow quiver.

\subsection{Delannoy paths} \label{ss:paths}

An \defn{$(m,n)$ Delannoy path} is a path in the plane from $(0,0)$ to $(m,n)$ composed of steps of the form $(1,0)$, $(0,1)$, or $(1,1)$. The number of $(m,n)$ Delannoy paths is the Delannoy number $D(m,n)$, which is an interesting combinatorial quantity \cite{Banderier}.

Delannoy paths are useful for us since they parametrize orbits on products. Consider a point $(y,x)$ in $\bR^{(n)} \times \bR^{(m)}$. We associate to it an $(m,n)$ Delannoy path $\alpha$. Mark the points $y_1, \ldots, y_n$ and $x_1, \ldots, x_m$ on the real line. Now walk the real line from $-\infty$ to $\infty$. Each time a $y$ point is encounted (that is not also an $x$ point), we take a step up; similarly, when an $x$ point is encountered (that is not also a $y$ point), we take a step to the right; when we come to a point that is both an $x$ and a $y$, we step diagonally. If $g \in \bbG$ then $(y,x)$ and $(gy,gx)$ have the same path. Moreover, the function
\begin{displaymath}
\bbG \backslash (\bR^{(n)} \times \bR^{(m)}) \to \{ \text{$(m,n)$ Delannoy paths} \}
\end{displaymath}
is a bijection. See \cite[\S 3.4]{line} for a proof of this assertion, and an illustrated example. Following \cite{line}, we say that an $(n,n)$-Delannoy path $p$ is \emph{quasi-diagonal} if it passes through every vertex along the main diagonal. See Figure~\ref{f:delpath} for examples.

For an $(m,n)$ Delannoy path $\alpha$, we let $O_{\alpha}$ be the corresponding orbit on $\bR^{(n)} \times \bR^{(m)}$. We let $C_{\alpha}$ be the indicator function of $O_{\alpha}$, which we regard as an $\bR^{(n)} \times \bR^{(m)}$ matrix. The $C_{\alpha}$'s, as $\alpha$ varies, forms a basis for the space
\begin{displaymath}
\Hom_{\cC}(\sC(\bR^{(m)}), \sC(\bR^{(n)})),
\end{displaymath}
and similarly for $\cA$.

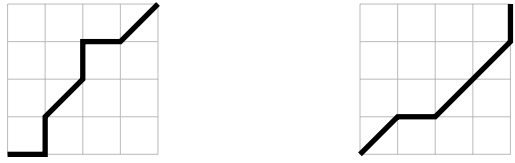
\begin{figure}
\begin{center}
\begin{tikzpicture}[scale=0.5]
\draw[step=1, color=gray!50] (0, 0) grid (4,4);
\draw[line width=2pt] (0,0)--(1,0)--(1,1)--(2,2)--(2,3)--(3,3)--(4,4);
\end{tikzpicture}
\qquad\qquad\qquad
\begin{tikzpicture}[scale=0.5]
\draw[step=1, color=gray!50] (0, 0) grid (4,4);
\draw[line width=2pt] (0,0)--(1,1)--(2,1)--(3,2)--(4,3)--(4,4);
\end{tikzpicture}
\end{center}
\caption{Two $(4,4)$-Delannoy paths. The left one is quasi-diagonal, while the right one is not.}
\label{f:delpath}
\end{figure}

\subsection{Pure maps} \label{ss:pure}

In \S \ref{ss:matrix-ex} we computed some matrix products in $\cA$ and $\cC$, and observed that they were different in general. However, we also saw that some products were the same. In fact, there is a large class of products that are the same, as we now show.

Let $M$ be a $\bR^{(n)} \times \bR^{(m)}$ matrix. We say that $M$ is \defn{pure of type $\bb$} if its support is contained in the set of pairs $(y,x)$ with $y_n \le x_m$. A product of pure matrices of type $\bb$ is also pure of type $\bb$ in either $\cC$ or $\cA$; moreover, the products in these two categories agree. Indeed, for an $\bR^{(n)} \times \bR^{(m)}$ matrix $A$ and an $\bR^{(m)} \times \bR^{(l)}$ matrix $B$, we have
\begin{displaymath}
AB(z,x)\;=\;\int_{\bR^{(m)}} A(z,y)B(y, x) dy.
\end{displaymath}
Now, if $B$ is pure of type $\bb$, then only non-zero values occur on $\hat{\bbG}$-sets with coordinates bounded from above. These sets have the same measure under $\mu_1$ and $\mu_2$, and the product coincides for both measures. If $A$ is also pure of type $\bb$, then the value on the left will be zero unless $z_n<x_l$, as the product inside the integral is only non-zero when $z_n<y_m<x_l$. Hence $AB$ is again of pure type $\bullet$.  We define \defn{pure of type $\ww$} similarly, but using $y_n \ge x_m$; these matrices have similar properties.

\subsection{The functor \texorpdfstring{$\Phi$}{Phi}} \label{ss:functor} 

Let $\ol{\bR}=\bR \cup \{\infty\}$, as in \S \ref{ss:four-meas}. By \cite[\S 1.4]{delmap}, there is an essentially unique faithful tensor functor
\begin{displaymath}
\Phi \colon \cA \to \cC, \qquad \Phi(\sA(\bR)) = \sC(\ol{\bR}) =\sC(\bR)\oplus\bbone.
\end{displaymath}
In more detail, the totally ordered $\bbG$-set $\ol{\bR}$ leads to an ordered \'etale algebra $\sC(\ol{\bR})$ in $\cC$ in the sense of \cite[\S 7.2]{delmap}, which is Delannic of type~2 in the sense of \cite[\S 7.3]{delmap}; this follows either from direct verification or the general results of \cite[\S 7.5]{delmap}. By \cite[Theorem~7.13]{delmap}, we therefore obtain the tensor functor $\Phi$ above. We have an algebra isomorphism \begin{displaymath}
\Phi(\sA(\bR^{(n)})) = \sC(\ol{\bR}{}^{(n)}) \cong \sC(\bR^{(n)}) \oplus \sC(\bR^{(n-1)}).
\end{displaymath}
The first isomorphism follows from \cite[Proposition~6.5]{delmap}, or direct verification, while the second is induced from the isomorphism of $\bbG$-sets $\ol{\bR}{}^{(n)} \cong \bR^{(n)} \sqcup \bR^{(n-1)}$.

Suppose $A$ is an endomorphism of $\sA(\bR^{(n)})$, that is, an $\bR^{(n)} \times \bR^{(n)}$ matrix. Then $\Phi(A)$ is an endomorphism of $\sC(\bR^{(n)}) \oplus \sC(\bR^{(n-1)})$, which we can think of as a $2 \times 2$ block matrix. For $x=\{x_1<\cdots<x_{n-1}\}$ in $\bR^{(n-1)}$, let $x^+=\{x_1<\cdots<x_{n-1}<\infty\}\in\overline{\bR}{}^{(n)}$.  We have
\begin{displaymath}
\Phi(A) = \begin{pmatrix} A_1 & A_2 \\ A_3 & A_4 \end{pmatrix},
\end{displaymath}
where, for $x,y\in \bR^{(n)}$ and $u,v\in\bR^{(n-1)}$,
\begin{align*}
A_1(y,x) &= A(y,x) & A_3(v,x) &= A(v^+,x) \\
A_2(y,u) &= A(y,u^+) & A_4(v,u) &= A(v^+,u^+).
\end{align*}
Note that it does make sense to plug-in $\infty$ into $A$, as $A$ is given by some formula in the language of total orders; in fact, one can evaluate $A$ at elements of any totally ordered set. If $A$ is pure then $\Phi(A)$ is either upper or lower triangular.

\section{The additive tensor category $\cA$} \label{s:A}

In this section, we investigate the structure of the category $\cA$. In particular, we prove Theorem~\ref{mainthm1}, which describes the indecomposable objects of $\cA$ and the maps between them. We establish some other results too, such as the tensor product rule for indecomposables.

\subsection{Krull--Schmidt categories} \label{ss:krull-schmidt}

We begin by recalling some general background material; we refer to \cite{krause} for additional details. A \defn{Krull--Schmidt category} is an additive category in which every object decomposes into a finite direct sum of objects having local endomorphism rings. Equivalently, a category is Krull--Schmidt if and only if it is Karoubian and the endomorphism ring of any object is semi-perfect. A general property of Krull--Schmidt categories is that an object decomposes into a direct sum of indecomposable objects in a unique way (up to permuting factors).

Since a finite dimensional $k$-algebra is semi-perfect, it follows that a $k$-linear Karoubian category with finite dimensional $\Hom$ spaces is Krull--Schmidt. Thus $\uPerm(G, \mu)^{\rm kar}$ is a Krull--Schmidt category, for any pro-oligomorphic group $G$ and measure $\mu$. In particular, the category $\cA$ is Krull--Schmidt. We will use this property when we classify the indecomposable objects in $\cA$ below.

\subsection{Construction of indecomposables} \label{ss:indecomp-constr}

For a weight $\lambda$ of length $n$, we define\footnote{This set was denoted $\Psi_{\lambda}$ in \cite[\S 5]{line}.} $\cE_{\lambda} \subset \bR^{(n)} \times \bR^{(n)}$ to be the set of pairs $(y,x)$ satisfying the following:
\begin{itemize}
\item $y_i \le x_i$ whenever $\lambda_i=\bb$
\item $x_i \le y_i$ whenever $\lambda_i=\ww$
\item $y_i<x_{i+1}$ and $x_i<y_{i+1}$.
\end{itemize}
We let $E_{\lambda}$ be the indicator function of $\cE_{\lambda}$, thought of as a matrix. In the Delannoy category $\cC$, it was shown in \cite[\S 5.3]{line} that $E_{\lambda}$ is the idempotent that projects onto the unique copy of $L_{\lambda}$ in $\sC(\bR^{(n)})$. Since $E_{\lambda}$ is a pure map of type $\lambda_n$, it follows that $E_{\lambda}$ is also idempotent in the second Delannoy category $\cA$ (see \S \ref{ss:pure}). We now come to an important definition:

\begin{definition}
For a weight $\lambda$ of length $n$, we let $M_{\lambda}=E_{\lambda} \sA(\bR^{(n)})$.
\end{definition}

We now study these objects. We note that if $\lambda=\varnothing$ is the empty word then $M_{\lambda}=\bbone$ is the unit object. Assuming $\lambda$ has non-zero length, we let $\lambda^{\flat}$ denote the weight obtained by deleting the final letter. We study $\cA$ via the faithful tensor functor $\Phi\colon \cA\to\cC$.

\begin{proposition} \label{prop:Phi-E}
Suppose $\lambda$ has length $n \ge 1$. Then
\begin{displaymath}
\Phi(E_{\lambda}) = \begin{pmatrix} E_{\lambda} & \ast \\ 0 & E_{\lambda^{\flat}} \end{pmatrix} \quad \text{or} \quad
\Phi(E_{\lambda}) = \begin{pmatrix} E_{\lambda} & 0 \\ \ast & E_{\lambda^{\flat}} \end{pmatrix}
\end{displaymath}
where the first case happens if $\lambda_n=\bb$, and the second if $\lambda_n=\ww$. Hence
\begin{displaymath}
\Phi(M_{\lambda})=L_{\lambda} \oplus L_{\lambda^{\flat}}.
\end{displaymath}
In particular, $M_{\lambda} \cong M_{\mu}$ if and only if $\lambda=\mu$, and $M_\lambda^\ast\cong M_{\lambda^\vee}$, where $\lambda^\vee$ is as in \S \ref{ss:delcat}.
\end{proposition}

\begin{proof}
The formula for $\Phi(E_{\lambda})$ follows from the general description of $\Phi$ given in \S \ref{ss:functor}. The formula for $\Phi(M_{\lambda})$ follows directly from this. Finally, since $L_{\lambda^{\vee}}$ occurs with multiplicity one in $\Phi(\sA(\bR^{(n)}))$ \cite[Theorem~4.7]{line}, it follows from Krull--Schmidt that there is a unique (up to isomorphism) indecomposable summand $M$ of $\sA(\bR^{(n)})$ such that $\Phi(M)$ contains $L_{\lambda^{\vee}}$. Since both $M_{\lambda}^*$ and $M_{\lambda^{\vee}}$ have this property, they must be isomorphic.
\end{proof}

\begin{proposition} \label{prop:Xind}
We have $\End(M_{\lambda})=k$, and so $M_{\lambda}$ is indecomposable. The dimension of $M_{\lambda}$ is~1 if $\lambda$ is empty, and~0 otherwise.
\end{proposition}

\begin{proof}
Let $n=\ell(\lambda)$. If $n=0$ the result is clear, so suppose $n>0$. We first claim that any endomorphism $C$ of $\sA(\bR^{(n)})$ has vanishing trace. Indeed, we have
\begin{displaymath}
\tr(C) = \int_{\bR^{(n)}} C(x,x) dx.
\end{displaymath}
Since $C$ is $G$-invariant and $\bR^{(n)}$ is a transitive set, $C(x,x)$ is a constant function. Since the $\mu_2$-volume of $\bR^{(n)}$ is~0 (see \S \ref{ss:four-meas}), the claim follows.

Since $M_{\lambda}$ is a summand of $\sA(\bR^{(n)})$, it follows that any endomorphism of $M_{\lambda}$ also has vanishing trace. Thus its dimension is zero. Now, since $\Phi$ is faithful, we have an inclusion of algebras
\begin{displaymath}
\End(M_{\lambda}) \subset \End(\Phi(M_{\lambda})).
\end{displaymath}
Since $\Phi$ is a tensor functor, it is compatible with traces. On the target, the trace is not identically zero; for instance, $L_{\lambda}$ has non-zero dimension. Thus this inclusion is proper. Since the target is two-dimensional, it follows that the source is one dimensional, which completes the proof.
\end{proof}

\subsection{Degenerate morphisms} \label{ss:degen}

Fix $n\in\bZ_{>0}$ throughout \S \ref{ss:degen}, and let $R$ be the endomorphism ring of $\sA(\bR^{(n)})$. We say an element of $R$ is \defn{degenerate} if it is a sum of morphisms that factor through $\sA(\bR^{(n-1)})$. Let $\fa \subset R$ be the space of degenerate morphisms; one easily sees that it is a two-sided ideal of $R$. We will require the following result, which estimates the size of $\fa$. Recall that $C_{\alpha}$ is the matrix associated to the Delannoy path $\alpha$ (\S \ref{ss:paths}). In what follows, we regard $C_{\alpha}$ as a morphism in $\cA$.

\begin{proposition} \label{prop:degen}
The quotient $R/\fa$ is spanned by the $C_{\alpha}$ where $\alpha$ is a quasi-diagonal path that has no diagonal segments; in particular, $R/\fa$ has dimension $\le 2^n$.
\end{proposition}

In fact, we will see below that the dimension is exactly $2^n$. We require a number of lemmas before proving the proposition.

\begin{lemma} \label{lem:degen-1}
Let $\alpha$ be an $(n,n)$ Delannoy path that contains three consecutive segments where the first is either right or diagonal, the second is right, and the third is either right or diagonal. Then $C_{\alpha}$ belongs to $\fa$. The same statement holds with ``right'' changed to ``up.''
\end{lemma}

\begin{proof}
Suppose there are three consecutive right segments. This means that points in $O_{\alpha}$ satisfy inequalities
\begin{displaymath}
x_j,y_{i-2}<y_{i-1}<y_i<y_{i+1}<x_{j+1},y_{i+2}
\end{displaymath}
for some $j$ (by convention $x_0=-\infty$ and $x_{n+1}=+\infty$ and similar for $y$). Clearly, the value $C_{\alpha}(y,x)$ does not depend on the precise value of $y_i$. More precisely, with $p=p_{n,i}$, we have
\begin{displaymath}
C_\alpha(y,x)\;=\; C_\beta(p(y),x)\;=\; (p^\ast C_\beta) (y,x),
\end{displaymath}
where $O_\beta \subset \bR^{(n-1)}\times\bR^{(n)}$ is the image of $O_{\alpha}$ under $p \times \id$. It follows that $C_{\alpha}=p^* \circ C_{\beta}$, which shows that $C_{\alpha}$ belongs to $\fa$. The other cases are similar.
\end{proof}

\begin{lemma} \label{lem:degen-2}
Let $\alpha$ be an $(n,n)$ Delannoy path such that the first segment of $\alpha$ is right, and the second segment of $\alpha$ is either right or diagonal. Then $C_{\alpha}$ belongs to $\fa$. The same statement holds with ``right'' changed to ``up.''
\end{lemma}

\begin{proof}
This is similar to the proof of Lemma~\ref{lem:degen-1}. For instance, if $\alpha$ begins with two right segments then we have $y_1<y_2<x_1,y_3$ for points in $O_{\alpha}$, and so $(C_{\alpha} \phi)(y)$ is independent of $y_1$, and we in fact get a relation $C_{\alpha}=p^* C_{\beta}$ as before.
\end{proof} 

Now we will use the two previous lemmas, which are concerned with local properties of Delannoy paths, to conclude something about the global property of being quasi-diagonal. As quasi-diagonal is not a local property, the proof becomes more involved.

\begin{lemma} \label{lem:degen-3}
Let $\alpha$ be an $(n,n)$-Delannoy path that is not quasi-diagonal. Then $C_{\alpha} \in \fa$.
\end{lemma}

\begin{proof}
Let $i$ be maximal such that $\alpha$ passes through $(j,j)$ for all $0 \le j \le i$. We prove the statement by induction on $i$. The base case, where $i=0$, follows from Lemma~\ref{lem:degen-2}. We now carry out the induction step.

There are four possibilities for what the path $\alpha$ does starting from $(i,i)$: (a) right twice; (b) right then diagonal; (c) up then diagonal; (d) up twice. The analysis in each case is similar, so we just handle case (a). If the segment coming into $(i,i)$ is either horizontal or diagonal then the result follows from Lemma~\ref{lem:degen-1}; similarly if the segment coming out of $(i+2,i)$ is either horizontal or diagonal. We thus have the following picture:
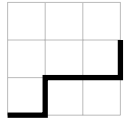
\begin{center}
\begin{tikzpicture}[scale=0.5]
\draw[step=1, color=gray!50] (0, 0) grid (3,3);
\draw[line width=2pt] (0,0)--(1,0)--(1,1)--(2,1)--(3,1)--(3,2);
\end{tikzpicture}
\end{center}
where the point $(1,1)$ in the diagram is really $(i,i)$. We thus see that $O_{\alpha}$ has defining inequalities including
\begin{displaymath}
y_{i-1}<x_{i-1}<y_i<y_{i+1}<x_i.
\end{displaymath}
Note that the segment coming into $(i,i-1)$ must be horizontal by the maximality of $i$.

Let $p=p_{n,i-1} \colon \bR^{(n)} \to \bR^{(n-1)}$ be the map omitting the $i-1$ coordinate, and let
\begin{displaymath}
q \colon \bR^{(n)} \times \bR^{(n)} \to \bR^{(n)} \times \bR^{(n-1)}
\end{displaymath}
be the map $p \times \id$. Let $\beta$ be the $(n,n-1)$ Delannoy path where we delete the vertical segment from $\alpha$ coming into $(i,i)$, and slide the remainder of the diagram down one unit. Then $q(O_{\alpha})=O_{\beta}$. Now consider $q^{-1}(O_{\beta})$. This is a union of orbits $O_{\alpha_1}, \ldots, O_{\alpha_r}$. These orbits consist of the points $(y,x)$ satisfying the inequalities defining $O_{\alpha}$, except those involving $x_i$. One of these orbits, say $O_{\alpha_1}$, is $O_{\alpha}$.

We claim that $C_{\alpha_j}$ belongs to $\fa$ for $j \ne 1$. To see this, consider the possible position for $x_{i-1}$. If $x_{i-1} < y_{i-1}$ then $\alpha_j$ has three consecutive right segments, and so $C_{\alpha_j} \in \fa$ by Lemma~\ref{lem:degen-1}. Similarly, if $x_{i-1}=y_{i-1}$ then $\alpha_j$ has a diagonal segment followed by two right segments, and so $C_{\alpha_j} \in \fa$ again. We can thus suppose $y_{i-1}<x_{i-1}$. We cannot have $x_{i-1}<y_i$, as this would mean we have $\alpha_1$. Thus $y_i \le x_{i-1}$. However, this means that $\alpha_j$ does not pass through $(i,i)$, and so $C_{\alpha_j} \in \fa$ by the inductive hypothesis.

By matrix multiplication, we have
\begin{displaymath}
C_{\beta} \circ  p_* = \sum_{j=1}^r C_{\alpha_j}.
\end{displaymath}
Since the left side belongs to $\fa$ and all terms on the right with $j>1$ belong to $\fa$, it follows that $C_{\alpha_1}=C_{\alpha}$ belongs to $\fa$, which completes the proof.
\end{proof}

\begin{lemma}\label{lem:degen-4}
Let $\alpha$ be a quasi-diagonal $(n,n)$ Delannoy path, and suppose $\alpha$ has a diagonal segment from $(i,i)$ to $(i+1,i+1)$. Let $\beta$ and $\gamma$ be the paths obtained by replacing this segment with ``up then right'' and ``right then up'' respectively. Then
\begin{displaymath}
C_{\alpha}+C_{\beta}+C_{\gamma} \in \fa.
\end{displaymath}
\end{lemma}

\begin{proof}
Let $p \colon \bR^{(n)} \to \bR^{(n-1)}$ be the projection map omitting the $i+1$th coordiante, and let
\begin{displaymath}
q \colon \bR^{(n)} \times \bR^{(n)} \to \bR^{(n-1)} \times \bR^{(n-1)}\
\end{displaymath}
be the map $p \times p$. Then $q(O_{\alpha})=O_{\ol{\alpha}}$, where $\ol{\alpha}$ is obtained by deleting the diagonal segment from $(i,i)$ to $(i+1,i+1)$ in $\alpha$, and sliding the rest of the diagram down and to the left.

Now, suppose $\epsilon$ is a quasi-diagonal $(n,n)$ Delannoy path such that $q(O_{\epsilon})=O_{\ol{\alpha}}$. To specify a quasi-diagonal path, we have to choose one of three options for the path from $(j,j)$ to $(j+1, j+1)$, for each $j$. The condition $q(O_{\epsilon})=O_{\ol{\alpha}}$ means that $\epsilon$ makes the same choice as $\alpha$ for all $j \ne i$. We thus see that $\epsilon$ must be one of $\alpha$, $\beta$, or $\gamma$. Moreover, each of these three possibilities works for $\epsilon$.

The above discussion shows that we have an orbit decomposition
\begin{displaymath}
q^{-1}(O_{\ol{\alpha}}) = O_{\alpha} \sqcup O_{\beta} \sqcup O_{\gamma} \sqcup \bigsqcup_{j=1}^r O_{\delta_j},
\end{displaymath}
where each $\delta_j$ is not quasi-diagonal. A simple computation shows that the matrix $p^* C_{\ol{\alpha}} p^*$ is the indicator function of $q^{-1}(O_{\ol{\alpha}})$. We thus find
\begin{displaymath}
p^* \, C_{\ol{\alpha}} \, p_* =C_{\alpha}+C_{\beta}+C_{\gamma} + \sum_{j=1}^r C_{\delta_j}.
\end{displaymath}
Since the left side and the $C_{\delta_j}$ belong to $\fa$, the result follows.
\end{proof}

\begin{proof}[Proof of Proposition~\ref{prop:degen}]
By Lemma~\ref{lem:degen-3}, $R/\fa$ is spanned by the $C_\alpha$ with $\alpha$ quasi-diagonal. If $\alpha$ is quasi-diagonal and has a diagonal segment then, by Lemma~\ref{lem:degen-4}, the image of $C_\alpha$ in $R/\fa$ is in the span of the $C_\beta$ with $\beta$ quasi-diagonal and strictly fewer diagonal segments. The conclusion follows by induction.
\end{proof}

\subsection{Classification of indecomposables}\label{s:classindec}

We now show that the indecomposable objects of $\cA$ that we have constructed account for all such objects.

\begin{theorem}\label{thm:class}
Every indecomposable object of $\cA$ is isomorphic to $M_{\lambda}$ for some $\lambda$.
\end{theorem}

Recall that $\cA$ is a Krull--Schmidt category (\S \ref{ss:krull-schmidt}). This will play an important role throughout the proof. We say that an indecomposable object $M$ of $\cA$ has \defn{level $n$} if $M$ is a summand of $\sA(\bR^{(n)})$ with $n$ minimal. Every indecomposable has some level, and there are only finitely many indecomposables of a given level up to isomorphism; both statements follow from the Krull--Schmidt property.

\begin{lemma} \label{lem:class-1}
The indecomposable object $M_{\lambda}$ has level $n=\ell(\lambda)$.
\end{lemma}

\begin{proof}
By definition, $M_{\lambda}$ is a summand of $\sA(\bR^{(n)})$, and thus has level $\le n$. By Proposition~\ref{prop:Phi-E}, $L_{\lambda}$ is a summand of $\Phi(M_{\lambda})$. If $m<n$, then $L_{\lambda}$ is not a summand of
\begin{displaymath}
\Phi(\sA(\bR^{(m)})) = \sC(\bR^{(m)}) \oplus \sC(\bR^{(m-1)}),
\end{displaymath}
and so $M_{\lambda}$ does not have level $<n$.
\end{proof}

\begin{lemma} \label{lem:class-2}
Let $e$ be a primitive idempotent endomorphism of $\sA(\bR^{(n)})$. If $e$ is degenerate (\S \ref{ss:degen}) then $e \sA(\bR^{(n)})$ has level $<n$.
\end{lemma}

\begin{proof}
Put $X=\sA(\bR^{(n)})$, $Y=\sA(\bR^{(n-1)})$, and $M=eX$. Since $e$ is degenerate, we have an expression $e=\sum_{i=1}^r b_i a_i$ for some morphisms $a_i \colon X \to Y$ and $b_i \colon Y \to X$. Let $a \colon X \to Y^{\oplus r}$ and $b \colon Y^{\oplus r} \to X$ be the maps defined by the $a_i$'s and $b_i$'s, so that $e=ba$. One easily sees that $f=aeb$ is an idempotent endomorphism of $Y^{\oplus r}$, and that the maps
\begin{displaymath}
fae \colon X \to Y^{\oplus r}, \qquad ebf \colon Y^{\oplus r} \to X
\end{displaymath}
induce mutually inverse isomorphisms between $eX$ and $f Y^{\oplus r}$. Thus $M$ is a summand of $Y^{\oplus r}$. By Krull--Schmidt, it is thus a summand of $Y$, and so $M$ has level $<n$ as required.
\end{proof}

\begin{proof}[Proof of Theorem~\ref{thm:class}]
Let $Q_1, \ldots, Q_r$ be the indecomposable objects of level $\le n$, labeled so that $Q_1, \ldots, Q_t$ have level $n$ and $Q_{t+1}, \ldots, Q_r$ have level $<n$. By Krull--Schmidt, we have an isomorphism
\begin{displaymath}
\sA(\bR^{(n)}) = Q_1^{\oplus m_1} \oplus \cdots \oplus Q_r^{\oplus m_r}
\end{displaymath}
where the $m_i$ are uniquely determined multiplicities. Let $R$ be the endomorphism algebra of $\sA(\bR^{(n)})$. For $1 \le i \le t$, let $e_i \in R$, be the projection onto the first $Q_i$ summand; note that $m_1, \ldots, m_t$ are non-zero, by definition of level. The $e_i$'s are thus orthogonal primitive idempotents in $R$. 

Let $\fa \subset R$ be the ideal of degenerate morphisms (\S \ref{ss:degen}), and let $\pi \colon R \to R/\fa$ be the quotient map. For $1 \le i \le t$ we have $e_i \not\in \fa$ by Lemma~\ref{lem:class-2}, and so $\pi(e_i)$ is a non-zero idempotent of $R/\fa$. Since these idempotents are orthogonal, they are necessarily linearly independent. We thus have
\begin{displaymath}
t \le \dim(R/\fa) \le 2^n,
\end{displaymath}
where the second inequality comes from Proposition~\ref{prop:degen}. Since the $M_{\lambda}$ with $\ell(\lambda)=n$ are $2^n$ indecomposables of level $n$ (Lemma~\ref{lem:class-1}), it follows that they are the only indecomposables of level $n$. Since $n$ is arbitrary, the result follows.
\end{proof}

\subsection{Maps of indecomposables}

We now determine the maps between the indecomposables objects of $\cA$, thereby completing the proof of Theorem~\ref{mainthm1}.

\begin{theorem} \label{thm:maps}
For $\lambda \in \Lambda$, there are non-zero maps
\begin{displaymath}
d_\lambda \colon M_{\lambda \ww} \to M_{\lambda}, \qquad
u_\lambda \colon M_{\lambda} \to M_{\lambda \bb}, \qquad
u_\lambda\circ d_\lambda \colon M_{\lambda \ww} \to M_{\lambda \bb}.
\end{displaymath}
These are the only maps (up to scalar multiples) between $M_{\alpha}$ and $M_{\beta}$ with $\alpha \ne \beta$.
\end{theorem}

\begin{proof}
We break the proof into a number of cases. Throughout we use that $\Phi$ is faithful, and the explicit computation of $\Phi(M_{\lambda})$ from Proposition~\ref{prop:Phi-E}.

\textit{Case 1: $M_{\lambda\ww} \to M_{\lambda}$.} Let $p \colon \bR^{(n+1)} \to \bR^{(n)}$ be the projection omitting the final coordinate, where $n=\ell(\lambda)$, so that we have an associated map $p_* \colon \sA(\bR^{(n+1)}) \to \sA(\bR^{(n)})$. We claim that $p_* E_{\lambda\ww} \ne 0$. Indeed, take $(y,x) \in \cE_{\lambda \ww}$. We have
\begin{displaymath}
(p_\ast E_{\lambda\ww})(p(y),x) \;=\;
\int_{\bR^{(n+1)}}\;\delta_{p(y),p(z)}\; E_{\lambda\ww}(z,x) dz
\;=\;\int_{\bR} E_{\lambda\ww}(y_1,\ldots, y_{n},z_{n+1},x) dz_{n+1}.
\end{displaymath}
The integrand on the right is~1 if $z_{n+1}$ belongs to the interval $[x_{n+1}, \infty)$, and~0 otherwise. Thus the integral is equal to the volume of this interval, which is~1; here we are using the explicit description of the measure $\mu_2$ from \S \ref{ss:four-meas}. This establishes the claim.

The above calculation shows that $p_*$ restricts to a non-zero map $M_{\lambda \ww} \to \sA(\bR^{(n)})$. Decomposing $\sA(\bR^{(n)})$ into indecomposables, we see that there is a non-zero map $M_{\lambda \ww} \to M_{\mu}$ for some weight $\mu$ of length $\le n$. Now, we have an injection
\begin{displaymath}
\Hom_{\cA}(M_{\lambda \ww}, M_{\mu}) \to \Hom_{\cC}(\Phi(M_{\lambda \ww}), \Phi(M_{\mu})).
\end{displaymath}
If $\mu \ne \lambda$ then the target vanishes, and so we must have $\mu=\lambda$. In this case, the target is one-dimensional. We thus see that there is a non-zero map $d_{\lambda}$ and it spans the source.

\textit{Case 2: $M_{\lambda} \to M_{\lambda \bb}$.} Applying duality to Case~1, we see that there is a non-zero map $u_{\lambda}$ which spans $\Hom(M_{\lambda}, M_{\lambda \bb})$.

\textit{Case 3: $M_{\lambda \ww} \to M_{\lambda \bb}$.} Consider the composition
\begin{displaymath}
\xymatrix{
M_{\lambda \ww} \ar[r]^-{d_{\lambda}} & M_{\lambda} \ar[r]^-{u_{\lambda}} & M_{\lambda \bb}. }
\end{displaymath}
Applying $\Phi$, the two maps are isomorphisms on the unique $L_{\lambda}$ summands, and so the composition is non-zero. Since $\Hom(\Phi(M_{\lambda \ww}), \Phi(M_{\lambda \bb}))$ is one-dimensional, we see that $u_{\lambda} \circ d_{\lambda}$ spans the space $\Hom(M_{\lambda\ww}, M_{\lambda\bb})$.

\textit{Case 4: $M_{\lambda} \to M_{\lambda \ww}$.} Supppose we have a non-zero map $f \colon M_{\lambda} \to M_{\lambda \ww}$. Then $\Phi(f)$ is non-zero on the unique copy of $L_{\lambda}$ in the source and target. Thus, as in Case~3, we see that $f \circ d_{\lambda}$ is non-zero. Since $\Phi(d_{\lambda})$ kills $L_{\lambda \ww}$, we see that $f \circ d_{\lambda}$ is not an isomorphism. But this is a contradiction, since $\End(M_{\lambda \ww})$ is spanned by the identity (Proposition~\ref{prop:Xind}). We conclude $\Hom(M_{\lambda}, M_{\lambda \ww})=0$.

\textit{Case 5: $M_{\lambda \bb} \to M_{\lambda}$.} Applying duality to Case~4, we see that $\Hom(M_{\lambda \bb}, M_{\lambda})$ vanishes.

\textit{Case 6: $M_{\lambda \bb} \to M_{\lambda \ww}$.} Suppose we have a non-zero map $f \colon M_{\lambda \bb} \to M_{\lambda \ww}$. Then $\Phi(f)$ is non-zero on the unique copy of $L_{\lambda}$ in the source and target. Thus, as in Case~3, we see that $d_{\lambda} \circ f$ is non-zero. But this contradicts Case~5. We thus find that $\Hom(M_{\lambda \bb}, M_{\lambda \ww})$ vanishes.

\textit{Remaining cases.} Suppose we have a non-zero map $M_{\alpha} \to M_{\beta}$, with $\alpha \ne \beta$. Applying $\Phi$, we obtain a non-zero map $\Phi(M_{\alpha}) \to \Phi(M_{\beta})$. We thus see that the sets $\{\alpha, \alpha^{\flat} \}$ and $\{\beta, \beta^{\flat} \}$ must have non-empty intersection. This implies that we are in one of the first six cases, and completes the proof.
\end{proof}

\subsection{Semi-simplification} \label{ss:ss}

Let $\fn$ be the ideal of negligible morphisms in $\cA$ (see \cite{AK, EtingofOstrik}). Proposition~\ref{prop:Xind} implies that the identity map of $M_{\lambda}$ is negligible if $\lambda$ is non-empty. It follows that the semi-simplification $\cA/\fn$ is equivalent to the category of vector spaces\footnote{This can also be seen using the method of \cite[\S 4.6]{arboreal}.}. We let
\begin{displaymath}
\Theta \colon \cA \to \Vec
\end{displaymath}
be the semi-simplification functor. This is the unique tensor functor that kills all $M_{\lambda}$ with $\lambda$ non-empty. The functor $\Theta$ also kills each $\sA(\bR^{(n)})$ with $n$ positive, since this object does not contain $\bbone$ as a summand; one can see this directly, or using Corollary~\ref{cor:schwartz-decomp} below. This functor will play an important role in our analysis of $\bD(\cB)$ in \S \ref{s:B2}. For now, we record one useful consequence of it.

\begin{proposition} \label{prop:notsummand}
If $\lambda,\mu \in \Lambda$ are weights, at least one of which is non-empty, then $\bbone$ is not a summand of $M_{\lambda} \otimes M_{\mu}$.
\end{proposition}

\begin{proof}
Indeed, we have
\begin{displaymath}
\Theta(M_{\lambda} \otimes M_{\mu}) = \Theta(M_{\lambda}) \otimes \Theta(M_{\mu}) = 0,
\end{displaymath}
and so $M_{\lambda} \otimes M_{\mu}$ cannot contain $\bbone$ as a summand since $\Theta(\bbone)=k$.
\end{proof}

Using our thorough understanding of $\cA$, we can completely classify its tensor ideals:

\begin{proposition}\label{prop:negl}
The only proper tensor ideals in $\cA$ are $\fn$ and the zero ideal.
\end{proposition}

\begin{proof}
By \cite[Theorem~3.1.1]{Selecta}, see also \cite{AK}, tensor ideals in $\cA$ are in natural bijection with subfunctors of 
\begin{displaymath}
\Hom(\bbone,-) \colon \cA \to \Vec.
\end{displaymath}
By Theorem~\ref{thm:maps}, there are only two indecomposable objects that admit non-zero morphisms from $\bbone$, namely $\bbone$ itself and $M_\bb$. The map $\bbone\to M_{\bb}$ generates the unique non-zero proper subfunctor, and hence the unique non-zero proper ideal, which is $\fn$.
\end{proof}

\subsection{The Grothendieck ring} \label{ss:groth}

Let $\rK(-)$ denote the Grothendieck group of an abelian category, and let $\rK^{\oplus}(-)$ denote the split Grothendieck group of an additive category. Thus $\rK^{\oplus}(\cA)$ is the free $\bZ$-module with basis $[M_{\lambda}]$ for $\lambda \in \Lambda$, and $\rK(\cC)=\rK^{\oplus}(\cC)$ is the free $\bZ$-module with basis $[L_{\lambda}]$ for $\lambda \in \Lambda$. Both $\rK(\cC)$ and $\rK^{\oplus}(\cA)$ are rings, via the tensor product. Moreover, the functor $\Phi$ induces a ring homomorphism
\begin{displaymath}
\phi \colon \rK^{\oplus}(\cA) \to \rK(\cC), \qquad [M_{\lambda}] \mapsto [L_{\lambda}] + [L_{\lambda^{\flat}}],
\end{displaymath}
where, as usual, the $[L_{\lambda^{\flat}}]$ term is omitted when $\lambda$ is empty.

\begin{proposition} \label{prop:groth}
The ring homomorphism $\phi$ is an isomorphism.
\end{proposition}

\begin{proof}
There is an upper triangular change of basis between the $[L_{\lambda}]$ and $\phi([M_{\lambda}])$ in $\rK(\cC)$.
\end{proof}

\begin{corollary} \label{cor:lyndon}
As a commutative $\bQ$-algebra, $\bQ \otimes \rK^{\oplus}(\cA)$ is freely generated by the classes $[M_{\lambda}]$ as $\lambda$ varies over Lyndon words.
\end{corollary}

\begin{proof}
This follows from the corresponding result for $\rK(\cC)$ \cite[Corollary~7.21]{line}.
\end{proof}

\begin{corollary} \label{cor:schwartz-decomp}
For $n \ge 1$ we have
\begin{displaymath}
\sA(\bR^{(n)}) = \bigoplus M_{\lambda}^{\oplus m(\lambda)}, \qquad
m(\lambda) = \binom{n-1}{\ell(\lambda)-1},
\end{displaymath}
where the sum is taken over weights $\lambda$ of length $\le n$.
\end{corollary}

\begin{proof}
Let $X$ and $Y$ be the two sides of the equation. It suffices to verify $[X]=[Y]$ in $\rK^{\oplus}(\cA)$. By Proposition~\ref{prop:groth}, it thus suffices to verify $\phi([X])=\phi([Y])$. Write 
\begin{displaymath}
\phi([X]) = \sum_{\lambda} a_{\lambda} [L_{\lambda}], \qquad
\phi([Y]) = \sum_{\lambda} b_{\lambda} [L_{\lambda}].
\end{displaymath}
We have
\begin{displaymath}
\Phi(\sA(\bR^{(n)})) = \sC(\bR^{(n)}) \oplus \sC(\bR^{(n-1)})).
\end{displaymath}
Thus, by \cite[Theorem~4.7]{line}, we have
\begin{displaymath}
a_{\lambda} = \binom{n}{\ell(\lambda)} + \binom{n-1}{\ell(\lambda)}
\end{displaymath}
On the other hand, Proposition~\ref{prop:Phi-E} shows that the multiplicity of $L_{\lambda}$ in $\Phi(M_{\mu})$ is~1 if $\mu \in \{\lambda, \lambda\ww, \lambda\bb\}$, and~0 otherwise. We thus find
\begin{displaymath}
b_{\lambda} = \binom{n-1}{\ell(\lambda)-1} + 2\binom{n-1}{\ell(\lambda)}.
\end{displaymath}
The equality $a_{\lambda}=b_{\lambda}$ follows from Pascal's rule.
\end{proof}

\subsection{The tensor product rule}

We now give an explicit rule for tensor products of $M_{\lambda}$'s. For this, we must first recall the explicit tensor product rule in the Delannoy category. Let $\lambda$ and $\mu$ be weights of lengths $m$ and $n$. An \defn{$(m,n)$-ruffle} is a pair of maps $\rho_1 \colon [m] \to [\ell]$ and $\rho_2 \colon [n] \to [\ell]$ such that $\rho_1$ and $\rho_2$ are each injective and order preserving, and $[\ell]=\im(\rho_1) \cup \im(\rho_2)$. We say that there is a \defn{collision} at $i\in [\ell]$ if $i \in \im(\rho_1) \cap \im(\rho_2)$. Suppose $i$ is a collision, and write $i=\rho_1(a)$ and $i=\rho_2(b)$ We say that the collision is \defn{white} if $\lambda_a$ and $\mu_b$ are both white, \defn{black} if they are both black, and \defn{neutral} otherwise. Let $S \subset [\ell]$ denote the set of neutral collisions. A \defn{marking} of the ruffle is a function
\begin{displaymath}
\rho_3 \colon S \to \{ \bb, \ww, \varnothing \}.
\end{displaymath}
We associate to a marked ruffle $\rho$ a weight $\nu$ as follows. For $i \in [\ell]$, we define $\nu_i$ as follows:
\begin{itemize}
\item If $i$ is not a collision then $\nu_i=\lambda_a$ if $i=\rho_1(a)$ and $\nu_i=\mu_b$ if $i=\rho_2(b)$.
\item If $i$ is a black (resp.\ white) collision then $\nu_i$ is black (resp.\ white).
\item If $i$ is a neutral collision then $\nu_i=\rho_3(i)$.
\end{itemize}
Finally, $\nu$ is the concatenation of the weights $\nu_1, \ldots, \nu_{\ell}$ (each of which has length~0 or~1).

Let $\Omega_{\lambda,\mu}$ denote the set of marked ruffles, and we let $\omega \colon \Omega_{\lambda,\mu} \to \Lambda$ be the function assigning to a marked ruffle the corresponding weight. The tensor product rule is then
\begin{displaymath}
L_{\lambda} \otimes L_{\mu} = \bigoplus_{\rho \in \Omega_{\lambda,\mu}} L_{\omega(\rho)}.
\end{displaymath}
This is \cite[Corollary~7.3]{line}. In words, the simples on the right side are obtained by ruffling the words $\lambda$ and $\mu$, with the following rule for collisions: white collisions remain white, black collisions remain black, and neutral collisions assume each possibility of white, black, or empty.

Define $\Omega'_{\lambda,\mu}$ to be the subset of $\Omega_{\lambda,\mu}$ consisting of marked ruffles $\rho$ satisfying the following condition: if $\rho$ has length $\ell$ then $\rho_3(\ell)$ (if defined) is not $\varnothing$. In words, this means that if we have a neutral collision in the final position, we must use either white or black; empty is not allowed.

\begin{theorem} \label{thm:tensor-rule}
We have
\begin{displaymath}
M_{\lambda} \otimes M_{\mu} = \bigoplus_{\rho \in \Omega'_{\lambda,\mu}} M_{\omega(\rho)}.
\end{displaymath}
\end{theorem}

\begin{proof}
If either $\lambda$ or $\mu$ is empty then the result is clear, so assume this is not the case. Since $\phi \colon \rK^{\oplus}(\cA) \to \rK(\cC)$ is a ring isomorphism (Proposition~\ref{prop:groth}), it suffices to verify the identity after passing to the target. This means we want to prove
\begin{displaymath}
(L_{\lambda} \oplus L_{\lambda^{\flat}}) \otimes (L_{\mu} \oplus L_{\mu^{\flat}}) = \bigoplus_{\rho \in \Omega'_{\lambda,\mu}} (L_{\omega(\rho)} \oplus L_{\omega(\rho)^{\flat}})
\end{displaymath}
To do this, we construct functions
\begin{displaymath}
f,g \colon \Omega'_{\lambda,\mu} \to \Omega_{\lambda,\mu} \amalg \Omega_{\lambda,\mu^{\flat}} \amalg \Omega_{\lambda^{\flat}, \mu} \amalg \Omega_{\lambda^{\flat}, \mu^{\flat}}
\end{displaymath}
such that:
\begin{enumerate}[(i)]
\item $f$ and $g$ are jointly bijective, i.e., they are both injective and their images are a disjoint cover of the target.
\item $\omega(\rho)=\omega(f(\rho))$ and $\omega(\rho)^{\flat}=\omega(g(\rho))$
\end{enumerate}
This will establish the above isomorphism by providing a bijection between the simple summands of the source and target. The definition of $f$ is easy: it is simply the inclusion of $\Omega'_{\lambda,\mu}$ into $\Omega_{\lambda,\mu}$. This is clearly compatible with $\omega$.

Let $\rho \in \Omega'_{\lambda,\mu}$ be given, and let $\ell$ be the length of $\rho$. First suppose that $\ell$ is not a collision, and $\ell=\rho_1(m)$. We define $g(\rho) \in \Omega_{\lambda^{\flat},\mu}$ by simply restricting $\rho$, that is, $g(\rho)_1$ is the restriction of $\rho_1$ to $[m-1]$, while $g(\rho)_2$ and $g(\rho)_3$ are equal to $\rho_2$ and $\rho_3$. It is clear that $\omega(g(\rho))=\omega(\rho)^{\flat}$. If $\ell=\rho_2(n)$, we define $g(\rho) \in \Omega_{\lambda,\mu^{\flat}}$ in a similar fashion.

Now suppose that $\ell$ is a white collision. We then define $g(\rho) \in \Omega_{\lambda^{\flat}, \omega^{\flat}}$ by restriction again, that is, $g(\rho)_1$ is the restriction of $\rho$ to $[m-1]$, $g(\rho)_2$ is the restriction of $\rho_2$ to $[n-1]$, and $g(\rho)_3$ is the restriction of $\rho_3$ to $S \setminus \{\ell\}$. Once again, it is clear that $\omega(g(\rho))=\omega(\rho)^{\flat}$. The case where $\ell$ is a black collision is similar.

Finally, suppose that $\ell$ is a neutral collision. To define $g(\rho)$, we must arbitrarily introduce some asymmetry between black and white. If $\rho_3(\ell)$ is black, we define $g(\rho) \in \Omega_{\lambda,\mu}$ to be the same as $\rho$, except for one difference: we put $g(\rho)_3(\ell)=\varnothing$. If $\rho_3(\ell)$ is white, we define $g(\rho) \in \Omega_{\lambda^{\flat}, \mu^{\flat}}$ by restriction. Again, it is clear that $\omega(g(\rho))=\omega(\rho)^{\flat}$.

We have thus defined $f$ and $g$, and explained why (ii) holds. A moment's thought shows that (i) also holds.
\end{proof}

\begin{example}
We have
\begin{displaymath}
L_{\bb} \otimes L_{\ww} = L_{\bb\ww} \oplus L_{\ww\bb} \oplus L_{\bb} \oplus L_{\ww} \oplus L_{\varnothing}.
\end{displaymath}
The first two simples on the right come from the two shuffles (ruffles without collisions), while the final three terms come from the ruffle where the two letters collide. We have
\begin{displaymath}
M_{\bb} \otimes M_{\ww} = M_{\bb\ww} \oplus M_{\ww\bb} \oplus M_{\bb} \oplus M_{\ww}.
\end{displaymath}
Again, the first two terms come from shuffles, while the final two terms come from the ruffle where the two letters collide; in this case, the white and black letters cannot cancel since the collision occurs at the end of the word. We have
\begin{displaymath}
(L_{\bb} \oplus L_{\varnothing}) \otimes (L_{\ww} \otimes L_{\varnothing}) = (L_{\bb\ww} \oplus L_{\ww\bb} \oplus L_{\bb} \oplus L_{\ww} \oplus L_{\varnothing}) \oplus L_{\bb} \oplus L_{\ww} \oplus L_{\varnothing}.
\end{displaymath}
The proof of Theorem~\ref{thm:tensor-rule} matches the trivial representation in $\Phi(M_{\bb})$ with the first trivial representation above, and the one in $\Phi(M_{\ww})$ with the second.
\end{example}

The tensor product rule gives another proof of Proposition~\ref{prop:notsummand}. It also allows us to prove the following result of a similar nature.

\begin{corollary}\label{cor:notsummand2}
If $\lambda,\mu \in \Lambda_{\ww}$ then every indecomposable summand of $M_{\lambda} \otimes M_{\mu}$ has the form $M_{\nu}$ with $\nu \in \Lambda_{\ww}$.
\end{corollary}

\section{The pre-Tannakian category $\cD$} \label{s:D}

In this section, we study the category $\cD$. The main result is that it admits the structure of a pre-Tannakian category (Theorem~\ref{mainthm2}). Note that while we arrived at $\cD$ as the Ringel dual of the pre-sheaf category $\cB$ (see \S \ref{s:over}), in this section we study $\cD$ directly. The category $\cB$ and the Ringel dual perspective will be discussed in \S \ref{s:B}.

\subsection{The combinatorial category}\label{def:D}

We recall the definition of the combinatorial category $\fD$ from \S \ref{ss:overview-D}. The set of objects is the set of weights $\Lambda$. Let $\Lambda^{\alt}$ denote the set of alternating weights, i.e., those that do not contain $\ww\ww$ or $\bb\bb$ as a substring, and let $\Lambda^{\alt}_{\ww}$ and $\Lambda^{\alt}_{\bb}$ be the sets of non-empty alternating weights ending in $\ww$ and $\bb$. The morphism space $\Hom_{\fD}(\lambda, \mu)$ is one dimensional (with a distinguished basis vector) if $\lambda=\mu$, or $\mu \in \lambda \Lambda^{\alt}_{\ww}$, or $\lambda \in \mu \Lambda^{\alt}_{\bb}$, and otherwise vanishes. The composition of two distinguished morphisms $\lambda \to \mu$ and $\mu \to \nu$ is the distinguished morphism $\lambda \to \nu$, if it exists, and is otherwise zero. 

\subsection{The module category} \label{def:Dmod}

Let $\cD^{\inf}$ be the category of $\fD$-modules, i.e., functors from $\fD$ to the category of vector spaces. We say that a $\fD$-module $M$ is \defn{pointwise finite} if $M(\lambda)$ is finite dimensional for all $\lambda$, and \defn{finite} if it is pointwise finite and $M(\lambda)$ vanishes for all but finitely many $\lambda$. We let $\cD^{\pf}$ and $\cD$ be the pointwise finite and finite subcategories of $\cD$.

Let $M$ be a $\fD$-module and let $\Xi$ be a subset of $\Lambda$. We say that $M$ is the \defn{full $\fD$-module} on $\Xi$ if $M(\lambda)=k$ for $\lambda \in \Xi$ and $M(\lambda)=0$ for $\lambda \not\in \Xi$, and where distinguished morphisms between two elements of $\Xi$ act by the identity. Given $\Xi$, the full $\fD$-module on $\Xi$ may or may not exist. We let $\bS_{\lambda}$ be the full $\fD$-module on $\{\lambda\}$. These are exactly the simple $\fD$-modules. We let $\bP_{\lambda}$ be the projective $\fD$-module $\Hom_{\fD}(\lambda, -)$. The module $\bP_{\lambda}$ is pointwise finite, but not finite. It is multiplicity-free: $\bS_{\mu}$ appears with multiplicity at most one in $\bP_{\lambda}$ for each $\mu$. In fact, $\bP_\lambda$ is the full $\fD$-module on the set
\begin{equation} \label{eq:setproj}
\{\lambda\}\cup\lambda\Lambda_{\ww}^{\alt}\cup\{\mu\mid \lambda\in \mu\Lambda^{\alt}_{\bb}\}.
\end{equation}
 We will freely use the following observation: if $M$ is the full module on a set $\Xi$, and $\bP_\lambda\to M$ is a non-zero morphism, then the image is the full module on the intersection of $\Xi$ and the set in \eqref{eq:setproj}.

The category $\fD$ has a self-duality, given on objects by $\lambda \mapsto \lambda^{\vee}$. This induces a duality on $\fD$-modules, denoted $(-)^{\vee}$, as follows:
\begin{displaymath}
M^{\vee}(\lambda) = M(\lambda^{\vee})^*,
\end{displaymath}
where $(-)^*$ is the ordinary vector space dual. Duality is defined on all $\fD$-modules, and is involutory on the pointwise finite and finite subcategories. We have $\bS_{\mu}^{\vee}=\bS_{\mu^{\vee}}$. The dual of $\bP_{\lambda^{\vee}}$ is the injective envelope $\bI_{\lambda}$ of $\bS_{\lambda}$ (see \cite[\S 3.4]{brauercat}).

\subsection{Structure of morphisms}

We define the \defn{basic morphisms} of $\fD$ to be the distinguished morphisms
\begin{displaymath}
\lambda \ww\bb \to \lambda, \qquad \lambda \bb \bb \to \lambda \bb, \qquad \bb \to \varnothing,
\end{displaymath}
\begin{displaymath}
\lambda \to \lambda \bb\ww, \qquad \lambda \ww \to \lambda \ww \ww, \qquad \varnothing \to \ww,
\end{displaymath}
where $\lambda$ is an arbitrary weight. They are important due to the following result:

\begin{proposition} \label{prop:basic}
Every distinguished morphism in $\fD$ factors uniquely into a composition of basic morphisms.
\end{proposition}

\begin{proof} 
Consider a distinguished morphism $\lambda \to \mu$ where $\mu=\lambda \alpha$ and $\alpha \in \Lambda_{\ww}^{\alt}$. We show that there is a unique sequence of weights
\begin{displaymath}
\lambda=\mu_0,\mu_1,\mu_2,\ldots,\mu_{n-1},\mu_n=\mu
\end{displaymath}
such that each $\mu_i$ is obtained from $\mu_{i-1}$ by either adding $\bb\ww$, removing $\ww\bb$, adding $\ww$ in case $\mu_{i-1}$ does not end in $\bb$, or removing $\bb$ in case $\mu_i$ does not end in $\ww$. This will establish the proposition for the kind of morphism under consideration. Duality then gives the other kind of morphism (where $\lambda = \mu \beta$ with $\beta \in \Lambda^{\alt}_{\bb}$).

It is clear that the only option is first to remove letters and then to add letters, since adding letters always results in a weight ending in $\ww$ and removing letters always requires a weight ending in $\bb$. Next we divide into cases. Assume first that $\alpha$ is of even length. Then it follows that adding $\ww$ or removing $\bb$ is not allowed at any intermediate step. Indeed, each of the four types of step results in a weight that has a balance of black dots minus white dots either the same or increased. It then quickly follows that the only option is precisely adding $\bb\ww$ in every step. Next, assume that $\alpha$ has odd length. It then similarly follows that every intermediate step will consist of adding $\bb\ww$ or removing $\ww\bb$ and, precisely once, either add $\ww$ or remove $\bb$. If $\lambda$ ends in $\ww$, the only option is to add first $\ww$, followed by adding a number of $\bb\ww$. In case $\lambda$ ends in $\bb$, the only option is first to remove $\bb$ followed by adding a number of $\bb\ww$.
\end{proof}

\begin{corollary}\label{cor:D-ext-simple}
The space $\Ext^1_{\cD}(\bS_\lambda, \bS_\mu)$ is one-dimensional if there is a basic morphism $\lambda \to \mu$, and otherwise vanishes. A basic morphism exists in exactly the following cases:
\begin{enumerate}
\item $\mu=\lambda\bb\ww$
\item $\lambda=\mu\ww\bb$
\item $\mu=\lambda\ww$ and $\lambda$ does not end in $\bb$
\item $\lambda=\mu \bb$ and $\mu$ does not end in $\ww$.
\end{enumerate}
\end{corollary}

\begin{proof}
A minimal projective presentation of $\bS_\lambda$ follows from the natural isomorphism
\begin{displaymath}
\Hom_{\cD}(\bP_{\lambda}, \bP_{\mu}) = \Hom_{\fD}(\mu, \lambda)
\end{displaymath}
and the proposition.
\end{proof}

\begin{example}
The morphism $\lambda \ww \bb \to \lambda \ww$ is not basic. Its basic factorization is
\begin{displaymath}
\lambda \ww \bb \to \lambda \to \lambda \ww.
\end{displaymath}
\end{example}

\subsection{Standard modules} \label{ss:Dstd}

We define the \defn{standard module} $\bDelta_{\lambda}$ via the presentation
\begin{displaymath}
\bP_{\lambda\ww}\oplus \bP_{\lambda\bb\ww}\to \bP_\lambda\to \bDelta_\lambda\to 0.
\end{displaymath}
Since $\bP_\lambda$ is the full module on the set in \eqref{eq:setproj}, it follows quickly that the image of $\bP_{\lambda\ww}\oplus \bP_{\lambda\bb\ww}\to \bP_\lambda$ is the full module on the subset $\lambda\Lambda_{\circ}^{\alt}$. Hence $\bDelta_\lambda$ is the full module on the set $\{\lambda \} \cup \{ \mu \mid \lambda \in \mu \Lambda^{\alt}_{\bb} \}$, from which we see that $\bDelta_{\lambda}$ has finite length. Yet equivalently, we can define $\bDelta_\lambda$ as the maximal quotient of $\bP_\lambda$ that does not contain any simple constituent $\bS_\mu$ with $\ell(\mu)>\ell(\lambda)$. Similarly, we define the \defn{costandard module} $\bnabla_{\lambda}$ to be the full $\fD$-module on the set $\{\lambda\} \cup \{\mu \mid \lambda \in \mu \Lambda^{\alt}_{\ww} \}$. These two kinds of modules are dual: we have $\bnabla_{\lambda^{\vee}} = \bDelta_{\lambda}^{\vee}$.

\begin{proposition}\label{prop:D-tilt}
We have the following:
\begin{enumerate}
\item If $\lambda \in \Lambda_{\ww}$ then $\bS_{\lambda}=\bDelta_{\lambda}$ and there is a short exact sequence
\begin{displaymath}
0 \to \bS_{\lambda} \to \bnabla_{\lambda} \to \bDelta_{\lambda^{\flat}} \to 0
\end{displaymath}
\item If $\lambda \in \Lambda_{\bb}$ then $\bS_{\lambda}=\bnabla_{\lambda}$ and there is a short exact sequence
\begin{displaymath}
0 \to \bnabla_{\lambda^{\flat}} \to \bDelta_{\lambda} \to \bS_{\lambda} \to 0
\end{displaymath}
\item If $\lambda=\varnothing$ then $\bDelta_{\lambda}=\bnabla_{\lambda}=\bS_{\lambda}$.
\end{enumerate}
\end{proposition}

\begin{proof}
(a) By construction, the socle of $\bnabla_\lambda$ is $\bS_\lambda$. The quotient $\bnabla_\lambda/\bS_\lambda$ is the full module on the set $\{\mu \mid \lambda \in \mu \Lambda^{\alt}_{\ww} \}$, which coincides with $\{\mu \mid \lambda^\flat \in \mu \Lambda^{\alt}_{\bb} \}\cup\{\lambda^\flat\}$, and so the full module is actually $\bDelta_{\lambda^\flat}$. (b) follows from (a) by duality, and (c) is immediate.
\end{proof}

The composition series of (co)standard modules can be derived either from Corollary~\ref{cor:D-ext-simple} or from iterating Proposition~\ref{prop:D-tilt}. For completeness, we record the precise statement for standard modules; the statement for costandard modules can be obtained by duality. We omit the proof since the statement will not be used.

\begin{proposition}\label{prop:uniserial}
The module $\bDelta_{\lambda}$ is uniserial. Precisely,
\begin{enumerate}
\item If $\lambda=\kappa(\ww\bb)^j$ with $\kappa$ not ending in $\bb$, then $\bDelta_\lambda$ is of Loewy length $2j+1$, with
\begin{displaymath}
\Rad^i(\bDelta_\lambda)/\Rad^{i+1}(\bDelta_\lambda)\;\cong\; \begin{cases} \bS_{\kappa (\ww\bb)^{j-i}}&\mbox{for } i\le j,\\
 \bS_{\kappa (\ww\bb)^{i-j-1}\ww}&\mbox{for } j< i\le 2j.\end{cases}
\end{displaymath}
\item If $\lambda=\kappa\bb(\ww\bb)^j$ with $\kappa$ not ending in $\ww$, then $\bDelta_\lambda$ is of Loewy length $2j+2$, with 
\begin{displaymath}
\Rad^i(\bDelta_\lambda)/\Rad^{i+1}(\bDelta_\lambda)\;\cong\; \begin{cases} \bS_{\kappa\bb (\ww\bb)^{j-i}}&\mbox{for } i\le j,\\
 \bS_{\kappa (\bb\ww)^{i-j-1}}&\mbox{for } j< i\le 2j+1.\end{cases}
\end{displaymath}
\end{enumerate}
Moreover, if $\Rad^i(\bDelta_\lambda)/\Rad^{i+1}(\bDelta_\lambda) = \bS_{\mu}$ then $i$ is the length of the factorization of the morphism $\lambda \to \mu$ into basic morphisms.
\end{proposition}

We illustrate the proposition with an example.

\begin{example}
Consider the standard module $\bDelta_{\ww\bb\ww\bb}$. It has five simple constituents, the simples corresponding to $\varnothing$, $\ww$, $\ww\bb$, $\ww\bb\ww$, and $\ww\bb\ww\bb$. The relevant basic morphisms of $\fD$ are
\begin{displaymath}
\ww\bb\ww\bb \to \ww\bb \to \varnothing \to \ww \to \ww\bb\ww.
\end{displaymath}
This picture shows that $\bDelta_{\ww\bb\ww\bb}$ is a uniserial module with top $\bS_{\ww\bb\ww\bb}$ and socle $\bS_{\ww\bb\ww}$.
\end{example}

We conclude this section with an important $\Ext$ calculation.

\begin{proposition}\label{prop:HomExt}
For $\lambda,\mu\in \Lambda$ and $i\in\bN$, we have
\begin{displaymath}
\Ext^i_{\ast}(\bDelta_\lambda,\bnabla_\mu)=
\begin{cases}
k & \text{if $\lambda=\mu$ and $i=0$} \\
0 & \text{otherwise,}
\end{cases}
\end{displaymath}
where $\ast$ is any of $\cD^{\inf}$, $\cD^{\pf}$, $\Ind\cD$, or $\cD$.
\end{proposition}

\begin{proof}
We have a projective resolution of $\fD$-modules $P^{\bullet} \to \bDelta_{\lambda}$, where $P^0=\bP_{\lambda}$ and $P^{-n}=\bP_{\lambda \ww^n} \oplus \bP_{\lambda (\bb\ww)^n}$ for $n \ge 1$. Changing $\lambda$ to $\mu^{\vee}$ and dualizing, we obtain a similar injective resolution $\nabla_{\mu} \to I^{\bullet}$. Since these resolutions belong to $\cD^{\pf}$, we can compute the $\Ext$ group in question in $\cD^{\inf}$ or in $\cD^{\pf}$ using either resolution. Now, for $n>0$, the space $\Hom(P^{-n}, \bnabla_{\mu})$ vanishes if $\ell(\lambda) \ge \ell(\mu)$, while the space $\Hom(\bDelta_{\lambda}, I^n)$ vanishes for $\ell(\mu) \ge \ell(\lambda)$. The space $\Hom(P^0, \bnabla_{\mu})$ vanishes unless $\lambda=\mu$ or $\mu \in \lambda \Lambda^{\alt}_{\ww}$, while the space $\Hom(\bDelta_{\lambda}, I^0)$ vanishes unless $\lambda=\mu$ or $\lambda \in \mu \Lambda^{\alt}_{\bb}$. Combining all of this gives the result when $\ast$ is $\cD^{\inf}$ or $\cD^{\pf}$.

We identify $\Ind\cD$ with the full subcategory of $\cD^{\inf}$ spanned by those modules that are unions of modules in $\cD$. The injective $\bI_{\lambda}$ belongs to $\Ind\cD$. Indeed, the weights appearing in $\bI_{\lambda}$ that are longer than $\lambda$ end in $\bb$, while the weights appearing in a projective $\bP_{\mu}$ that are longer than $\mu$ end in $\ww$. Thus a submodule of $\bI_{\lambda}$ generated in degree $\mu$ is supported on weights of length at most $\max(\ell(\lambda), \ell(\mu))$, and so any finitely generated submodule of $\bI_{\lambda}$ has finite length. We thus see that the injective resolution $I^{\bullet}$ of $\bnabla_{\mu}$ belongs to $\Ind\cD$. Computing with this resolution, we find
\begin{displaymath}
\Ext^i_{\Ind\cD}(M, \bnabla_\mu) = \Ext^i_{\cD^{\inf}}(M, \bnabla_\mu)
\end{displaymath}
for any object $M$ of $\Ind\cD$ and any $i \ge 0$. Applying this with $M=\bDelta_{\lambda}$ gives the result when $\ast$ is $\Ind\cD$. Finally, if $M$ and $N$ are objects of $\cD$ then we have an identification
\begin{displaymath}
\Ext^i_{\cD}(M, N) = \Ext^i_{\Ind\cD}(M, N)
\end{displaymath}
by \cite[Theoren 15.3.1]{KS}. This completes the proof.
\end{proof}

\subsection{Tilting modules} \label{ss:Dtilt}

For each $\lambda\in\Lambda$ we define the indecomposable \defn{tilting module} $\bT_\lambda$ as
\begin{displaymath}
\bT_\lambda=\begin{cases}
\bDelta_\lambda &\mbox{if}\; \lambda\in\Lambda_{\bb}\\
\bnabla_\lambda &\mbox{if}\; \lambda\in\Lambda_{\ww}\\
\bS_\lambda=\bDelta_\lambda=\bnabla_\lambda &\mbox{if}\;\lambda=\varnothing.
\end{cases}
\end{displaymath}
Equivalently, $\bT_\lambda$ is the full $\fD$-module on the set $\{\lambda\}\cup\{\mu \mid \lambda \in \mu \Lambda^{\alt}\}$. We note that $\bT_{\lambda}^{\vee} = \bT_{\lambda^{\vee}}$. A \defn{tilting module} is a $\fD$-module that is isomorphic to a direct sum of $\bT_{\lambda}$'s. The term tilting module will be justified later, but for now we note that each tilting module has both a finite standard and a finite costandard filtration by Proposition~\ref{prop:D-tilt}. That proposition also gives some important resolutions and short exact sequences.

\begin{corollary} \label{Cor:TiltRes}
We have the following:
\begin{enumerate}
\item If $\lambda=\kappa\bb^i$ with $\kappa \in \Lambda_{\ww} \cup \{\varnothing\}$ then we have an exact sequence
\begin{displaymath}
0\to \bT_{\kappa}\to \bT_{\kappa\bb} \to\cdots\to \bT_{\kappa\bb^{i-1}} \to \bT_{\lambda}\to \bS_\lambda\to 0.
\end{displaymath}
\item If $\lambda=\kappa\ww^i$ with $\kappa \in \Lambda_{\bb} \cup \{\varnothing\}$ then we have an exact sequence
\begin{displaymath}
0\to \bS_\lambda \to \bT_{\lambda} \to \bT_{\kappa\ww^{i-1}}\to\cdots\to \bT_{\kappa\ww}\to \bT_{\kappa}\to 0.
\end{displaymath}
\end{enumerate}
\end{corollary}

\begin{corollary}\label{Cor:SES}
We have the following:
\begin{enumerate}
\item If $\lambda \in \Lambda_{\bb} \cup \{\varnothing\}$ then there is a short exact sequence
\begin{displaymath}
0\to \bS_{\lambda}\to \bT_{\lambda\bb} \to \bS_{\lambda\bb} \to 0.
\end{displaymath}
\item If $\lambda \in \Lambda_{\ww} \cup \{\varnothing\}$ then there is a short exact sequence
\begin{displaymath}
0\to \bT_{\lambda}\to \bT_{\lambda\bb} \to \bS_{\lambda\bb} \to 0.
\end{displaymath}
\end{enumerate}
\end{corollary}

The proposition has one additional important corollary.

\begin{corollary} \label{cor:D-tilt-quot}
Every simple module is a quotient of a tilting module: precisely, $\bS_{\lambda}$ is a quotient of $\bT_{\lambda \ww}$.
\end{corollary}

\begin{proof}
The proposition shows that $\bDelta_{\lambda}$ is a quotient of $\bT_{\lambda \ww}$, and $\bS_{\lambda}$ is a quotient of $\bDelta_{\lambda}$.
\end{proof}

\begin{remark}
It is not true that every finite $\fD$-module is a quotient of a tilting module: indeed, this would imply that $\cD$ is the abelian envelope of $\cA$ by \cite[Theorem~2.2.1]{CEOP}, contradicting Theorem~\ref{mainthm3}. An example of such a module is given by the indecomposable length two module with top $\bS_{\varnothing}$ and socle $\bS_{\bb\ww}$. However, every finite $\fD$-module is a subquotient of a tilting module; this follows from Proposition~\ref{prop:equiv:triangle} below.
\end{remark}

We now examine maps between tilting modules.

\begin{proposition}\label{prop:homT}
We have the following:
\begin{enumerate}
\item The space $\End(\bT_{\lambda})$ is one dimensional.
\item The space $\Hom(\bT_{\lambda\ww}, \bT_{\lambda})$ is one dimensional with a distinguished basis vector $d_{\lambda}$.
\item The space $\Hom(\bT_{\lambda}, \bT_{\lambda\bb})$ is one dimensional with a distinguished basis vector $u_{\lambda}$.
\item The space $\Hom(\bT_{\lambda\ww}, \bT_{\lambda\bb})$ is one dimensional, and spanned by $u_{\lambda} d_{\lambda}$.
\item We have $\Hom(\bT_{\lambda}, \bT_{\mu})=0$ outside of the above four cases.
\end{enumerate}
\end{proposition}

\begin{proof}
The dimensions of the morphism spaces follow from the filtrations of the tilting modules given in Proposition~\ref{prop:D-tilt} in combination with Proposition~\ref{prop:HomExt} (more precisely the description of morphisms and the vanishing of $\Ext^1$). For example, to compute the dimension of $\Hom(\bT_{\lambda\ww},\bT_{\lambda \bb})$ we can observe that $\bT_{\lambda\ww}$ is an extension of $\bDelta_{\lambda}$ and $\bDelta_{\lambda\ww}$, while $\bT_{\lambda\bb}$ is an extension of $\bnabla_{\lambda}$ and $\bnabla_{\lambda\bb}$. The common occurrence of label $\lambda$ yields the one dimension of morphisms.

Finally we prove the claim about composition in (d), we can observe that in $\bT_{\lambda\ww}\to\bT_\lambda\to \bT_{\lambda\bb}$ either the left morphism is surjective or the right morphism is injective.
For example, assume that $\lambda\not\in\Lambda_{\bb}$. In this case $\bT_\lambda\cong\bnabla_\lambda$, while $\bnabla_\lambda\subset\bT_{\lambda\bb}$ by Proposition~\ref{prop:D-tilt}. 
\end{proof}

This proposition has a very important corollary, which explains how the category $\cD$ is connected to the second Delannoy category $\cA$. Let $\Tilt(\cD)$ be the full subcategory of $\cD$ spanned by tilting modules.

\begin{corollary} \label{cor:D-tilt-A}
We have a $k$-linear equivalence $\Psi \colon \cA \to \Tilt(\cD)$ given by $\Psi(M_{\lambda}) = \bT_{\lambda}$.
\end{corollary}

\begin{proof}
This follows from comparing Theorem~\ref{thm:maps} and Proposition~\ref{prop:homT}.
\end{proof}

The following result describes the derived category of $\cD$ purely in terms of tilting modules. We write $\bD(-)$ for the bounded derived category and $\bD^+(-)$ for the bounded above derived category. We write $\bK(-)$ and $\bK^+(-)$ for the analogous homotopy categories.

\begin{proposition}\label{prop:equiv:triangle}
The natural functor $\bK(\Tilt{\cD}) \to \bD(\cD)$ is an equivalence.
\end{proposition}

\begin{proof}
Let $M$ and $N$ be objects of $\bK(\Tilt{\cD})$, and consider the map
\begin{displaymath}
\Hom_{\bK(\Tilt{\cD})}(M, N) \to \Hom_{\bD(\cD)}(M, N).
\end{displaymath}
If $M$ and $N$ are each concentrated in a single degree then this map is an isomorphism. The key point is that all higher $\Ext$ groups vanish between tilting modules vanish, as a consequence of the $\Ext$-vanishing in Proposition~\ref{prop:HomExt} and the filtrations in Proposition~\ref{prop:D-tilt}. We now see that this map is an isomorphism for general $M$ and $N$ by induction on the number of non-zero terms in the complexes. Thus the functor $\bK(\Tilt{\cD})\to \bD(\cD)$ is fully faithful. To complete the proof, we must show it is essentially surjective. For this, it is enough to show that simple objects are in the essential image, which follows from Corollary~\ref{Cor:TiltRes}
\end{proof}

\begin{remark}
In Proposition~\ref{prop:equiv:triangle}, it is crucial that we work with bounded complexes: the functor $\bK^+(\Tilt\cD)\to\bD^+(\cD)$ is not faithful. Indeed, consider the complex
\begin{displaymath}
\cdots\to\bT_{\ww\ww\ww}\to \bT_{\ww\ww}\to \bT_{\ww}\to\bT_{\varnothing}\to0.
\end{displaymath}
This complex is not zero in $\bK^+(\Tilt\cD)$: indeed, there is no non-zero morphism $\bT_{\varnothing} \to \bT_{\ww}$, and so the identity is not null homotopic. However, the above complex is exact everywhere (see, e.g., Corollary~\ref{Cor:SES}), so it is zero in $\bD^+(\cD)$.
\end{remark}

\subsection{The tensor structure}\label{sec:tensorD}

Let $\cA'=\Tilt(\cD)$ be the category tilting objects in $\cD$. Recall (Corollary~\ref{cor:D-tilt-A}) that we have an equivalence of categories $\cA' \cong \cA$. We give $\cA'$ the structure of a tensor category by transferring the tensor structure from $\cA$. In this way, $\bK(\cA')$ is a tensor triangulated category. On the other hand, the equivalence $\bK(\cA')\cong \bD(\cD)$ in Proposition~\ref{prop:equiv:triangle} equips $\bK(\cA')$ with a t-structure whose heart $\bK(\cA')^{\heartsuit}$ is equivalent to $\cD$. Our main result establishes the compatibility of the two structures:

\begin{theorem}\label{Thm:Tensor}
The subcategory $\bK(\cA')^{\he}\subset \bK(\cA')$ is a rigid monoidal subcategory. Consequently, the category $\cD$ is pre-Tannakian. 
\end{theorem}

We will freely use that for a distinguished triangle $X\to Y\to Z\to X[1]$ in $\bK(\cA') $, the property $X,Z \in \bK(\cA')^{\heartsuit}$ implies that $Y \in \bK(\cA')^{\heartsuit}$.

We start the proof with some preparatory lemmas. Any tensor product $\otimes$ and duality $(-)^\ast$ refers to the tensor product on $\bK(\cA')$. With slight abuse of notation, we write $\bS_\lambda$ also for the object in $\heart$ that is represented by the corresponding (co)resolution in $\bK(\cA')$ provided by Corollary~\ref{Cor:TiltRes}, or equivalently the image of $\bS_\lambda$ under $\cD\subset\bD(\cD)\cong\bK(\cA')$.

\begin{lemma}\label{Lem:Tech}
Consider $X,Y,Z \in \heart$ such that $X\otimes Z\in\heart$, $Y\otimes Z\in\heart$ and $Z^\ast\otimes T\in\heart$ for all $T\in\cA'$. For any monomorphism $X\to Y$ in $\heart$, the induced morphism $X\otimes Z\to Y\otimes Z$ is also a monomorphism in $\heart$.
\end{lemma}

\begin{proof}
Since the simple module $\bS_\lambda$ is the quotient of $\bT_{\lambda\ww}$ (Corollary~\ref{cor:D-tilt-quot}) a morphism $f$ in $\heart$ is a monomorphism if and only if $f\circ g\not=0$, for $g$ running over all non-zero morphisms from objects $T\in\cA'$. By adjunction, we thus need to show that, for any non-zero morphism $Z^\ast\otimes T\to X$ (in $\bK(\cA')$), the composite $Z^\ast\otimes T\to X\to Y$ is not zero. By assumption, all these objects actually live in $\heart$, so the conclusion follows since $X\to Y$ is a monomorphism, by assumption.
\end{proof}

\begin{lemma}\label{Lem:HeartDual}
The duality functor $(-)^*$ on $\bK(\cA')$ preserves $\heart$, and satisfies $\bS_\lambda^\ast \cong \bS_{\lambda^{\vee}}$ for any weight $\lambda$. 
\end{lemma}

\begin{proof}
The second statement follows from Corollary~\ref{Cor:TiltRes} (duality takes the resolution in (a) to the one in (b), and vice versa). To show that duality preserves the heart, we now argue by induction on length.
\end{proof}

\begin{proof}[Proof of Theorem~\ref{Thm:Tensor}]
First we show that $\bS_\lambda \otimes \bT_\mu \in \heart$ for arbitrary weights $\lambda,\mu$.
We do this by induction on the length of the weight $\lambda$, the case of the empty weight being trivial.
Assume first that $\lambda$ ends in $\bb$, and $\lambda^\flat$ does not end in $\bb$, then the conclusion follows from Corollary~\ref{Cor:SES}(b) and Lemma~\ref{Lem:Tech}. Indeed, the latter shows that $\bT_{\lambda^\flat} \otimes \bT_\mu \to \bT_\lambda \otimes \bT_\mu$ remains a monomorphism in $\cD$. Since the complex $0 \to \bT_{\lambda^\flat} \to \bT_\lambda \to 0$ represents $\bS_\lambda$, the complex $0 \to \bT_{\lambda^\flat} \otimes \bT_\mu\to \bT_\lambda \otimes \bT_\mu \to 0$ represents $\bS_\lambda\otimes \bT_\mu$, and its cohomology is thus indeed contained in degree $0$. If $\lambda$ ends in $\bb$ and $\lambda^\flat$ does not end in $\ww$, we can use Corollary~\ref{Cor:SES}(a). Indeed, we can assume that $\bS_{\lambda^\flat}\otimes \bT_\mu\in \heart$, by the induction hypothesis, so we can again apply Lemma~\ref{Lem:Tech}. We can deal with the case where $\lambda$ ends in $\ww$ similarly.
 
Next we prove that $\bS_\lambda \otimes \bS_\mu\in\heart$, by applying the same argument as above, but with $\bT_\mu$ replaced with $\bS_\mu$. Indeed, we can now apply Lemma~\ref{Lem:Tech} with $Z=\bS_\mu$ since we have proved that $\bS_\mu^\ast\otimes T\in\heart$ for all $T\in\cA'$.

The fact that $X\otimes Y\in\heart$ for arbitrary $X,Y\in\heart$ can now be proved by induction on the lengths of $X,Y$.

We thus see that $\otimes$ gives $\heart$ the structure of a tensor category. It is rigid since it is a full monoidal subcategory of a rigid category that is closed under duality (Lemma~\ref{Lem:HeartDual}).
\end{proof}

\begin{remark} \label{remarkdualD}
We have two dualities on $\cD$, the pointwise duality $(-)^{\vee}$ and the monoidal duality $(-)^*$. They agree on the category of tilting modules, since we observed $\bT^\vee_\lambda\cong \bT_{\lambda^\vee}\cong \bT_\lambda^\ast$. Both functors are fully faithful and send distinguished morphisms to distinguished morphisms. Hence they are isomorphic on $\cD$. We thus see that at least part of the monoidal structure on $\cD$ (duality) can be be defined without reference to the oligomorphic theory.
\end{remark}

\subsection{A mapping property} \label{ss:D-map}

We now give a characterization of exact functors $\cD \to \cT$, where $\cT$ is an arbitrary $k$-linear abelian category. Suppose we have a $k$-linear functor $F \colon \Tilt(\cD) \to \cT$. Consider the following two conditions:
\begin{itemize}
\item[(SF)]  For every $\lambda \in \Lambda_{\bb} \cup \{\varnothing\}$, the following complex is exact:
\begin{displaymath}
\cdots \to F(\bT_{\lambda\ww\ww}) \to F(\bT_{\lambda\ww}) \to F(\bT_{\lambda}) \to 0
\end{displaymath}
\item[(CF)]  For every $\lambda \in \Lambda_{\ww} \cup \{\varnothing\}$, the following complex is exact:
\begin{displaymath}
0 \to F(\bT_{\lambda}) \to F(\bT_{\lambda\bb}) \to F(\bT_{\lambda\bb\bb}) \to \cdots
\end{displaymath}
\end{itemize}
The labels (SF) and (CF) stand for ``standard filtered'' and ``costandard filtered,'' as we explain in Remark~\ref{rmk:B-mod-SF}. 

\begin{proposition} \label{prop:exact-extension}
The functor $F$ satisfies (SF) and (CF) if and only if it can be extended to an exact functor $\cD \to \cT$. Moreover, any two exact extensions are isomorphic.
\end{proposition}

\begin{proof}
Suppose an exact extension $G \colon \cD \to \cT$ exists. The complex 
\begin{displaymath}
\cdots\to\bT_{\lambda\ww\ww}\to\bT_{\lambda\ww}\to\bT_\lambda\to 0,
\end{displaymath}
is exact in $\cD$. Applying $G$, we see that the resulting complex in $\cT$ is exact, and so $F$ satisfies (SF). The proof of (CF) is similar.

Conversely, suppose $F$ satisfies (SF) and (CF). Consider the triangulated functor
\begin{displaymath}
\tilde{F} \colon \bK(\Tilt(\cD)) \to \bD(\cT)
\end{displaymath}
extending $F$. Identifying $\bK(\Tilt(\cD))$ with $\bD(\cD)$, we claim that $\tilde{F}$ is t-exact, i.e., it maps $\cD \subset \bD(\cD)$ into $\cT\subset\bD(\cT)$. It suffices to verify that $\tilde{F}$ maps simple objects into $\cT$. By Corollary~\ref{Cor:TiltRes}, all simple objects in $\cD$ have a finite tilting (co)resolution, and the (SF) and (CF) conditions ensure that $\tilde{F}$ sends these to complexes in $\cT$ that have cohomology only in the appropriate degree. We have thus shown that $\tilde{F}$ is t-exact, and so the restriction of $\rH^0 \circ \tilde{F}$ to $\cD$ is an exact extension of $F$.

We now prove uniqueness. Suppose $G$ is an arbitrary exact extension of $F$. Since $G$ is exact, there is a functor $\tilde{G} \colon \bD(\cD) \to \bD(\cT)$ obtained by applying $G$ to a complex, and we have $G = \rH^0\circ \tilde{G}$ when restricted to $\cD \subset \bD(\cD)$. Under the identification $\bD(\cD) = \bK(\Tilt(\cD))$, the functor $\tilde{G}$ is identified with the functor $\tilde{F}$ in the previous paragraph. Thus $G$ is identified with $\rH^0 \circ \tilde{F}$, and so uniqueness follows.
\end{proof}

\begin{proposition} \label{prop:exact-tensor-extension}
Suppose $\cT$ and $F$ have tensor structures and $F$ satisfies (SF) and (CF). Then there is an exact tensor functor $\cD \to \cT$ extending $F$.
\end{proposition}

\begin{proof}
The functor $\rH^0 \circ \tilde{F}$ clearly admits a tensor structure in this case.
\end{proof}

\begin{remark} \label{rmk:exact-extension}
If $\cT$ is a rigid abelian tensor category and $F$ is a tensor functor then $F$ satisfies (SF) if and only if it satisfies (CF) since the two complexes are dual to each other.
\end{remark}

\subsection{The Grothendieck ring} 
We now look at the Grothendieck ring of $\cD$.

\begin{proposition} \label{prop:groth-D}
We have the following:
\begin{enumerate}
\item We have a ring isomorphism $i \colon \rK^{\oplus}(\cA) \to \rK(\cD)$ satisfying $i([M_{\lambda}])=[\bT_{\lambda}]$.
\item We have a ring isomorphism $j \colon \rK(\cD) \to \rK(\cC)$ satisfying $j([\bT_{\lambda}])=[L_{\lambda}] \oplus [L_{\lambda^{\flat}}]$.
\item The algebra $\bQ \otimes \rK(\cD)$ is freely generated by the $[\bT_{\lambda}]$ where $\lambda$ is a Lyndon word.
\item In characteristic~0, $i$ and $j$ are isomorphisms of $\lambda$-rings; in particular, the Adams operations on $\rK(\cD)$ are trivial (i.e., the identity).
\end{enumerate}
\end{proposition}

\begin{proof}
(a) The tensor equivalence $\cA \to \Tilt(\cD)$ induces a ring isomorphism $\rK^{\oplus}(\cA) \to \rK^{\oplus}(\Tilt(\cD))$. Composing with the natural ring homomorphism $\rK^{\oplus}(\Tilt(\cD)) \to \rK(\cD)$ gives a ring homomorphism $i$ satisfying the stated formula. By construction, $\bT_{\lambda}$ contains $\bS_{\lambda}$ with multiplicity one, and all other simple constituents have the form $\bS_{\mu}$ with $\ell(\mu)<\ell(\lambda)$. It follows that the classes $[\bT_{\lambda}]$ are $\bZ$-linearly independent. On the other hand, these classes span by Corollary~\ref{Cor:TiltRes}. Thus $i$ maps a $\bZ$-basis to a $\bZ$-basis, and is therefore an isomorphism.

(b) Recall the ring isomorphism $\phi \colon \rK^{\oplus}(\cA) \to \rK(\cC)$ from \S \ref{ss:groth}. We take $j=\phi \circ i^{-1}$.

(c) This follows from (a) and Corollary~\ref{cor:lyndon}.

(d) Since $i$ and $\phi$ are induced from tensor functors, they are maps of $\lambda$-rings, and so the same is true for $j$. Since the Adams operations on $\rK(\cC)$ are trivial \cite[Theorem~8.2]{line}, the result follows.
\end{proof}

\begin{remark} \label{rem:Z} 
The category $\cD$ is independent of the coefficient field, in many ways. The classification of simple objects in $\cD$ is independent of the field $k$, essentially by construction; in fact, $\cD$ admits a $\bZ$-form, in which the simple objects assemble into $\bZ$-flat families. The $\Ext^1$-quiver of simple objects is also independent of $k$, by Corollary~\ref{cor:D-ext-simple}. Finally, the Grothendieck ring is also independent of $k$, as Proposition~\ref{prop:groth-D} shows.
\end{remark}

\subsection{Highest weight structure} \label{ss:D-traingular}

By Proposition~\ref{prop:HomExt}, the category $\cD$ is a \defn{lower finite highest weight category} in the sense of \cite{BS}, see for instance \cite[Corollary~3.64]{BS}. The partial order on the labeling set $\Lambda$ of simple objects corresponding to this structure is given by $\mu<\lambda$ if and only if $\ell(\mu)<\ell(\lambda)$. The indecomposable modules $\bT_\lambda$ then correspond to the indecomposable tilting modules as defined and classified in \cite[Theorem~4.2]{BS}, justifying our terminology.

In the next section, we will consider the ``Ringel dual'' of $\cD$. Using those results the fact that $\cD$ is a highest weight category is immediate, and we could have observed that Corollary~\ref{cor:D-tilt-A} is an instance of the general principle in \cite[Corollary~4.30]{BS}.

Finally we point out that the highest weight structure on $\cD$ is the reflection of a ``triangular'' structure on $\fD$, in the sense\footnote{Actually, we must make a slight modification since \cite{brauercat} only treats the upper finite case.} of \cite{brauercat}. More precisely, let $\fD^+$ (resp.\ $\fD^-$) be the subcategory of $\fD$ containing all objects, but only those morphisms $\lambda \to \mu$ with $\ell(\lambda) \ge \ell(\mu)$ (resp.\ $\ell(\lambda) \le \ell(\mu)$). We think of $\fD^+$ as the ``upwards'' subcategory, even though its morphisms go from longer words to shorter words. Similarly, $\fD^-$ is the ``downwards'' subcategory. We note that duality interchanges the two subcategories, that is, if $f \colon \lambda \to \mu$ belongs to $\fD^+$ then $f^{\vee}$ belongs to $\fD^-$.

We prove the key factorization property for a triangular structure:

\begin{proposition}
For $\lambda,\mu \in \Lambda$, the natural map
\begin{displaymath}
\bigoplus_{\rho \in \Lambda} \Hom_{\fD^+}(\rho, \mu) \otimes \Hom_{\fD^-}(\lambda, \rho) \to \Hom(\lambda, \mu)
\end{displaymath}
induced by composition is an isomorphism.
\end{proposition}

\begin{proof}
If $\lambda \to \rho$ is a morphism in $\fD^-$ and $\lambda \ne \rho$ then $\rho \in \Lambda_{\ww}$. Similarly, if $\rho \to \mu$ is a morphism in $\fD^+$ and $\mu \ne \rho$ then $\rho \in \Lambda_{\bb}$. It follows that there are at most two non-zero terms in the sum, namely, when $\rho=\lambda$ or $\rho=\mu$. Now, by duality, we may as well assume that $\ell(\lambda) \le \ell(\mu)$. Then the $\rho=\lambda$ term of the sum also vanishes. We are thus reduced to showing that the map
\begin{displaymath}
\Hom(\mu,\mu) \otimes \Hom(\lambda, \mu) \to \Hom(\lambda, \mu)
\end{displaymath}
is an isomorphism, which is clear.
\end{proof}

\begin{remark}
The proof of Proposition~\ref{prop:basic} shows that every distinguished morphism $f$ in $\fD$ factors as $g \circ h$, where $h$ is a composition of basic morphisms that go from longer to shorter words, and $g$ is a composition of basic morphisms that go from short to longer words. This suggests that we might have a triangular structure where morphisms like $g$ are upwards and those like $h$ are downwards. However, this does not work, as it is possible for $g \circ h$ to vanish when $g$ is upwards and $h$ is downwards (in this sense), and this contradicts the triangular category axioms. For example, the composition $\bb \to \varnothing \to \ww$ is zero. The correct triangular structure on $\fD$ is the one described above.
\end{remark}

\section{The pre-sheaf category $\cB$} \label{s:B}

In this section, we study the pre-sheaf category $\cB$. We introduce a number of special modules and establish their basic properties. In particular, we describe the tilting modules and show that $\cD$ is Ringel dual to $\cB$. The material in this section will be used in the subsequent section where we analyze the derived category of $\cB$.

\subsection{The combinatorial category} \label{ss:B-cat}

We recall the definition of the combinatorial category $\fB$ from \S \ref{ss:overview-B}.  The set of objects is the set of weights $\Lambda$. The morphism spaces
\begin{displaymath}
\Hom_{\fB}(\lambda, \lambda), \quad
\Hom_{\fB}(\lambda, \lambda\ww), \quad
\Hom_{\fB}(\lambda\bb, \lambda), \quad
\Hom_{\fB}(\lambda\bb, \lambda\ww)
\end{displaymath}
are one-dimensional with a distinguished basis vector, while all other morphism spaces vanish. To specify the composition law, it suffices to describe the three compositions
\begin{displaymath}
\lambda\bb \to \lambda \to \lambda\ww, \qquad
\lambda \to \lambda\ww \to \lambda\ww\ww, \qquad
\lambda\bb\bb \to \lambda\bb \to \lambda.
\end{displaymath}
The first is the distinguished morphism, while the second two vanish. The category $\fB$ is anti-equivalent to the category of indecomposable objects in $\cA$ by Theorem~\ref{thm:maps}, via $\lambda \mapsto M_{\lambda}$; this is the reason we are interested in $\fB$. 

\subsection{The module category} \label{ss:B-mod}

We let $\cB^{\inf}$ be the category of $\fB$-modules, i.e., functors from $\fB$ to the category of vector spaces. Since $\fB$ is anti-equivalent to the category of indecomposable objects of $\cA$ (Theorem~\ref{mainthm1}), a $\fB$-module is the same thing as a ($k$-linear) pre-sheaf on the category $\cA$. We say that a $\fB$-module $V$ is \defn{pointwise finite} if $V(\lambda)$ is finite dimensional for all $\lambda$, and \defn{finite} if it is pointwise finite and $V(\lambda)$ vanishes for all but finitely many $\lambda$. We write $\cB^{\pf}$ and $\cB$ for the pointwise finite and finite subcategories of $\cB^{\inf}$.

Concretely, to give a $\fB$-module $V$, one must give a vector space $V(\lambda)$ for each weight $\lambda$ and linear maps
\begin{displaymath}
V(\lambda) \to V(\lambda \ww), \qquad V(\lambda \bb) \to V(\lambda)
\end{displaymath}
for each weight $\lambda$, such that the diagrams
\begin{displaymath}
V(\lambda) \to V(\lambda\ww) \to V(\lambda\ww\ww) \to \cdots, \qquad
\cdots \to V(\lambda\bb\bb) \to V(\lambda\bb) \to V(\lambda).
\end{displaymath}
are chain complexes. We will see (Proposition~\ref{prop:stanfil-B}) that the homology of these complexes (or lack thereof) connects to the structure of $V$ in an interesting way.

The category $\fB$ is self-dual; the duality is given on objects by $\lambda \mapsto \lambda^{\vee}$. This induces a duality on modules: for a $\fB$-module $V$, we define the \defn{dual}, denoted $V^{\vee}$, by
\begin{displaymath}
V^{\vee}(\lambda) = V(\lambda^{\vee})^*
\end{displaymath}
where the dual on the outside is the usual vector space dual. This operation is defined on all $\fB$-modules, and is involutory on the pointwise finite and finite subcategories.

\subsection{Structural modules} \label{ss:B-basic}

We now define a number of $\fB$-modules. As in \S \ref{def:Dmod}, we say that $V$ is the \defn{full $\fB$-module} on $\Xi \subset \Lambda$ if $V(\lambda)$ is $k$ for $\lambda \in \Xi$ and vanishes for $\lambda \not\in \Xi$, and whenever $\lambda \to \mu$ is a distinguished morphism of objects in $\Xi$ the map $V(\lambda) \to V(\mu)$ is the identity.

\textit{(a) Simple modules.} We let $\bbS_{\lambda}$ denote the full $\fB$-module on the set $\{\lambda\}$. This module is simple, and every simple object of $\cB$ is isomorphic to $\bbS_{\lambda}$ for a unique $\lambda$. We have $\bbS_{\lambda}^{\vee}=\bbS_{\lambda^{\vee}}$.

\textit{(b) Standard modules.} We let $\stan_{\lambda}$ denote the full $\fB$-module on the set $\{\lambda, \lambda \ww\}$. This is the \defn{standard module}. There is a non-split short exact sequence
\begin{equation}\label{eq:sesDelta}
0\to \bbS_{\lambda\circ}\to \stan_\lambda\to \bbS_\lambda \to 0.
\end{equation}
The term standard module is justified by the observations in (d) below, see also \S \ref{ss:tria-B}.

\textit{(c) Co-standard modules.} We let $\cost_{\lambda}$ denote the full $\fB$-module on the set $\{\lambda, \lambda\bb\}$. This is the \defn{costandard module}. There is a non-split short exact sequence
\begin{displaymath}
0\to \bbS_\lambda\to\cost_\lambda \to \bbS_{\lambda\bb}\to 0.
\end{displaymath}
Standard and costandard objects are dual: $\cost_{\lambda^{\vee}}=\stan_{\lambda}^{\vee}$.

\textit{(d) Projective modules.} We let $\bbP_{\lambda}$ be the $\fB$-module represented by $\lambda$, i.e.,
\begin{displaymath}
\bbP_{\lambda}(-)=\Hom_{\fB}(\lambda, -).
\end{displaymath}
This object is projective by Yoneda's lemma, and easily seen to be indecomposable; moreover, one easily sees that these account for all the indecomposable projectives in $\cB^{\inf}$. The behavior of $\bbP_{\lambda}$ depends on the final letter of $\lambda$. If $\lambda \in \Lambda_{\ww} \cup \{\varnothing\}$ then $\bbP_{\lambda}=\stan_{\lambda}$. If $\lambda \in \Lambda_{\bb}$ then $\bbP_{\lambda}$ is the full module on $\{\lambda, \lambda\ww, \lambda^{\flat}, \lambda^{\flat}\ww\}$. The relevant piece of $\fB$ looks like:
\begin{displaymath}
\xymatrix{
\lambda\ww & \lambda^{\flat}\ww \\
\lambda \ar[r] \ar[u] & \lambda^{\flat} \ar[u] }
\end{displaymath}
From this, we see that, for $\lambda\in\Lambda_{\bb}$, there is a short exact sequence
\begin{displaymath}
0 \to \stan_{\lambda^{\flat}} \to \bbP_{\lambda} \to \stan_{\lambda} \to 0,
\end{displaymath}
and $\bbP_{\lambda}$ has socle $\bbS_{\lambda^\flat\ww} \oplus \bbS_{\lambda\ww}$ and Loewy length three. In all cases, $\bbP_{\lambda}$ has finite length, and thus lives in the category $\cB$. If we think of $\cB$ as pre-sheaves on $\cA$ then $\bbP_{\lambda}$ is the pre-sheaf corepresented by $M_{\lambda}$. The functor $\cA \to \cB$ defined by $M_{\lambda} \mapsto \bbP_{\lambda}$ is an equivalence onto the category of projective objects of $\cB$.

\textit{(e) Injective modules.} We let $\bbI_{\lambda}$ be the $\fB$-module defined by
\begin{displaymath}
\bbI_{\lambda}(-) = \Hom_{\fB}(-, \lambda)^*.
\end{displaymath}
In other words, $\bbI_{\lambda}=\bbP_{\lambda^{\vee}}^{\vee}$, and its structure follows from this. The object $\bbI_{\lambda}$ is indecomposable injective, and these account for all such objects. If $\lambda \in \Lambda_{\bb} \cup \{\varnothing\}$ then $\bbI_{\lambda}=\cost_{\lambda}$; otherwise, there is a short exact sequence
\begin{displaymath}
0 \to \cost_{\lambda} \to \bbI_{\lambda} \to \cost_{\lambda^{\flat}} \to 0,
\end{displaymath}
and $\bbI_{\lambda}$ is the full object on $\{\lambda, \lambda\bb, \lambda^{\flat}, \lambda^{\flat}\bb\}$.

\textit{(f) The $\bbQ$ modules.} For a weight $\lambda$, we let $\bbQ_{\lambda}$ be the full $\fB$-module on $\{\lambda, \lambda\ww, \lambda\bb\}$. The relevant piece of $\fB$ looks like:
\begin{displaymath}
\lambda\bb \to \lambda \to \lambda\ww.
\end{displaymath}
From this, one sees that $\bbQ_{\lambda}$ is a uniserial module with top $\bbS_{\lambda\bb}$, socle $\bbS_{\lambda\ww}$, and $\bbS_{\lambda}$ in the middle. We have $\bbQ_{\lambda}^{\vee}=\bbQ_{\lambda^{\vee}}$ and $\bbQ_{\lambda}=\bbP_{\lambda\bb}/\bbS_{\lambda\bb\ww}$.
The significance of these modules may not be apparent at the moment, but in \S \ref{s:B2} and \ref{s:env} they will be quite important. For now we just point out a short exact sequence.

\begin{lemma}\label{lem:PQI}
For each non-empty $\lambda\in\Lambda$, we have a short exact sequence
\begin{displaymath}
0\to \bbP_\lambda\to \bbQ_\lambda\oplus \bbQ_{\lambda^\flat}\to\bbI_\lambda\to 0.
\end{displaymath}
\end{lemma}

\begin{proof}
We prove the case where $\lambda\in\Lambda_{\bb}$, the other case can for instance be obtained by applying $(-)^\vee$. Then $\bbQ_{\lambda^\flat}$, being the full module on $\{\lambda,\lambda^\flat,\lambda^\flat\ww\}$ is a quotient of $\bbP_{\lambda}$. There is also a non-zero morphism $\bbP_\lambda\to\bbQ_\lambda$, such that a generic linear combination of morphisms embeds $\bbP_\lambda$ into $\bbQ_\lambda\oplus \bbQ_{\lambda^\flat}$ with quotient the full module on $\{\lambda,\lambda\bb\}$, which is $\bbI_\lambda$.
\end{proof}

\subsection{Standard filtrations}

Let $V$ be a pointwise finite $\fB$-module. An (ascending) \defn{standard filtration} of $V$ is a chain 
\begin{displaymath}
0=F_0 \subset F_1 \subset F_2\subset \cdots
\end{displaymath}
of $\fB$-submodules such that $V=\bigcup_{i \ge 0} F_i$ and each $F_i/F_{i-1}$ is a standard module. The projective module $\bbP_{\lambda}$ has a standard filtration by the short exact sequence in \S \ref{ss:B-basic}(d). We now provide some characterizations of standard filtered objects. Let $\fB^+$ be the subcategory of $\fB$ with the same objects, but where the only morphisms are scalar multiples of the identities and the maps $\lambda \to \lambda \ww$; similarly define $\fB^-$ using the maps $\lambda\bb\to\lambda$. We call $\fB^+$ and $\fB^-$ the \defn{upwards} and \defn{downwards} categories. 

\begin{proposition} \label{prop:stanfil-B}
Let $V$ be a pointwise finite $\fB$-module. The following are equivalent:
\begin{enumerate}
\item $V$ has a standard filtration.
\item $V$ is projective as a $\fB^+$-module.
\item $V$ has finite projective dimension as a $\fB^+$-module.
\item For any $\lambda \in \Lambda_{\bb} \cup \{\varnothing\}$, the complex
\begin{displaymath}
0 \to V(\lambda) \to V(\lambda\ww) \to V(\lambda\ww\ww) \to \cdots
\end{displaymath}
is everywhere exact.
\end{enumerate}
\end{proposition}

\begin{proof}
Let $\cX$ be the category of cochain complexes of $k$-vector spaces that are supported in non-negative cohomological degrees. The simple objects of $\cX$ are complexes that are one-dimensional in a single degree, and zero elsewhere. The indecomposable projectives of $\cX$ are the complexes that are one-dimensional in two consecutive degrees (and zero elsewhere), with non-zero differential. Every indecomposable object of $\cX$ is either simple or projective. From this we see that an object of $\cX$ with finite projective dimension is projective, and also that a complex is projective if and only if its cohomology groups vanish.

Given a $\fB^+$-module $V$ and a weight $\lambda$ in $\Lambda_{\bb} \cup \{\varnothing\}$, let $V_{\lambda}$ be the complex appearing in (d). Then $V \mapsto (V_{\lambda})$ is an equivalence between the category $\fB^+$-modules and the direct product of copies of $\cX$ indexed by the set $\Lambda_{\bb} \cup \{\varnothing\}$. Moreover, we see that the standard $\fB$-module $\stan_{\lambda}$ correspond to an indecomposable projective (in an appropriate copy of $\cX$ depending on $\lambda$). From this, and the first paragraph, we see that (b), (c), and (d) are equivalent, and that they are implied by (a).

Finally, we show that (d) implies (a). Let $\lambda$ be a weight of minimal length with $V(\lambda)\not=0$. Under the assumption in (d), we have $\stan_\lambda\subset V$. Moreover, the complexes corresponding to $V/\stan_\lambda$ are still exact, so we can iterate to construct a standard filtration.
\end{proof}

There is a dual notion too. A (descending) \defn{costandard filtration} on $V$ is a chain
\begin{displaymath}
V=F^0 \supset F^1 \supset \cdots
\end{displaymath}
of $\fB$-submodules such that $\bigcap_{i \ge 0} F^i=0$ and each $F^i/F^{i+1}$ is a costandard module. There is an analog of the above proposition in the costandard case. Duality interchanges standard filtered and costandard filtered modules (in the pointwise finite category).

\begin{remark} \label{rmk:B-mod-SF}
By Proposition~\ref{prop:stanfil-B}(d), a $\fB$-module $V$ admits a standard filtration if and only if the functor $\Tilt(\cD) \to \Vec$ given by $\bT_{\lambda} \mapsto V(\lambda)^*$ satisfies condition (SF) from \S \ref{ss:D-map}, where we identify with $\fB$ with the category of indecomposable tilting objects in $\cD$.
\end{remark}

\begin{remark}
The definition of (co)standard filtration here agrees with the definitions given in \cite[\S 3.3]{BS} (though there the word ``flag'' is used). 
\end{remark}

\subsection{Tilting objects} \label{ss:B-tilt}

For a weight $\lambda$, we define the \defn{tilting module} $\bbT_{\lambda}$ to be the full object on $\{\lambda w\mid w\in \Lambda^{\alt}\}$. This set of weights can be pictured as followed:
\begin{displaymath}
\cdots \leftarrow \lambda\bb\ww\bb \to \lambda\bb\ww \leftarrow \lambda\bb \to
\lambda \to \lambda \ww \leftarrow \lambda \ww\bb \to \lambda\ww\bb\ww \leftarrow \cdots
\end{displaymath}
From this picture, it is clear that $\bbT_{\lambda}$ exists, and that its dual is $\bbT_{\lambda^{\vee}}$. The name is justified by the following result:

\begin{proposition} \label{prop:BTlambda}
There is a natural ascending filtration $F_{\bullet}$ of $\bbT_\lambda$ such that
\begin{displaymath}
F_1/F_0=\stan_{\lambda}, \qquad F_i/F_{i-1} = \stan_{\lambda (\ww\bb)^{i-1}} \oplus \stan_{\lambda \bb (\ww\bb)^{i-2}}.
\end{displaymath}
for $i \ge 2$, and a natural descending filtration $F^{\bullet}$ such that
\begin{displaymath}
F^0/F^1=\cost_{\lambda}, \qquad F^i/F^{i-1} = \cost_{\lambda (\bb\ww)^{i-1}} \oplus \cost_{\lambda \ww (\bb\ww)^{i-2}}.
\end{displaymath}
for $i \ge 2$. In particular, $\bbT_{\lambda}$ has both a standard and costandard filtration.
\end{proposition}

\begin{proof}
It is clear from the above picture that $\stan_{\lambda}$ is a $\fB$-submodule of $\bbT_{\lambda}$. The quotient $\bbT_{\lambda}/\stan_{\lambda}$ is a direct sum $M_1 \oplus M_2$, where $M_1$ and $M_2$ are the full modules on the following diagrams
\begin{displaymath}
\lambda\bb \to \lambda\bb\ww \leftarrow \lambda\bb\ww\bb \to \lambda\bb\ww\bb\ww \leftarrow \cdots
\end{displaymath}
\begin{displaymath}
\lambda \ww\bb \to \lambda\ww\bb\ww \leftarrow \lambda\ww\bb\ww\bb \to \lambda\ww\bb\ww\bb\ww \leftarrow \cdots
\end{displaymath}
The first $2i$ weights in the top diagram form a submodule of $M_1$, and the consecutive quotients of this filtration are $\stan_{\lambda \bb (\ww\bb)^i}$ for $i \ge 0$. Similarly, the first $2i$ weights in the bottom diagram form a submodule of $M_2$, and the quotients in this filtration are $\stan_{\lambda (\ww\bb)^i}$ for $i \ge 1$. Together, these give the filtration of $F_{\bullet}$. The filtration $F^{\bullet}$ comes a similar construction (or duality). These fitlrations can be refined to give a standard and costandard filtration of $\bbT_{\lambda}$.
\end{proof}

We now examine maps between these tilting modules.

\begin{proposition} \label{prop:tilt-map}
For $\lambda,\mu \in \Lambda$, we have 
\begin{displaymath}
\Hom(\bbT_\lambda,\bbT_\mu)\;=\;\begin{cases}
k&\mbox{ if }\; \mu=\lambda ,\\
k&\mbox{ if }\; \mu=\lambda a\quad \mbox{with }a\in \Lambda^{\alt}_{\bb},\\
k&\mbox{ if }\;\lambda=\mu b\quad\mbox{with } b\in \Lambda^{\alt}_{\ww},\\
0&\mbox{otherwise}.
\end{cases}
\end{displaymath}
In particular, $\bbT_{\lambda}$ is indecomposable. Additionally, the composition map
\begin{displaymath}
\Hom(\bbT_\lambda, \bbT_\mu) \times \Hom(\bbT_\mu, \bbT_\nu) \;\to\; \Hom(\bbT_\lambda, \bbT_\nu)
\end{displaymath}
is non-zero when each of the three morphism spaces is non-zero.
\end{proposition}

\begin{proof}
We first explain why the stated morphisms exist. We begin by considering the example of $\bbT_{\lambda} \to \bbT_{\lambda\ww\bb}$. The following is the relevant diagram
\begin{displaymath}
\xymatrix@C=2em@R=2em{
\cdots & \lambda \bb \ar[l] \ar[r]  & \lambda \ar[r] & \lambda \ww & \lambda \ww\bb \ar[l] \ar[r] \ar[d] & \lambda \ww\bb\ww \ar[d] & \lambda\ww\bb\ww\bb \ar[l] \ar[r] \ar[d] & \cdots \\
\cdots & \lambda\ww\bb\bb\ww\bb \ar[l] \ar[r] & \lambda \ww\bb\bb\ww & \lambda \ww\bb\bb \ar[r] \ar[l] & \lambda \ww \bb \ar[r] & \lambda \ww\bb\ww & \lambda \ww\bb\ww\bb \ar[l] \ar[r] & \cdots }
\end{displaymath}
The top row is $\bbT_{\lambda}$ and the bottom is $\bbT_{\lambda\ww\bb}$. Starting at $\lambda \ww\bb$ and going to the right, all weights in the two modules coincide, and these are the only weights at which they overlap. We thus have a map $f$ as above where each vertical arrow is the identity on $k$ (and the map vanishes everywhere else). Clearly, this is the unique map up to scalar multiple. We also observe that the map $f(\nu) \colon \bbT_{\lambda}(\nu) \to \bbT_{\lambda\ww\bb}(\nu)$ is the identity if both sides are $k$, and vanishes otherwise. The other cases are essentially the same as this one.

We now explain why these are the only non-zero maps. Assume that a morphism $\bbT_\lambda\to\bbT_\mu$ does not have $\bbS_{\lambda\ww}\subset\bbT_\lambda$ in its kernel. It follows that $\lambda=\mu$, and that the morphism is a scalar multiple of the identity. Next consider a morphism $\bbT_\lambda\to\bbT_\mu$ with $\bbS_{\lambda\ww}\subset\bbT_\lambda$ but not $\stan_\lambda\subset\bbT_\lambda$ in its kernel. This requires $\lambda=\mu \beta$ with $\beta \in \Lambda^{\alt}_{\ww}$ and again the morphism is unique up to scalar. Finally consider a morphism $\bbT_\lambda\to\bbT_\mu$ with $\stan_\lambda\subset\bbT_\lambda$ in its kernel. Assuming the morphism is not zero, there is a minimal length $\alpha \in \Lambda^{\alt}_{\bb}$ such that the constituent $\bbS_{\lambda \alpha}$ of $\bbT_\lambda$ is not sent to zero. Then $\mu =\lambda \alpha$ and the morphism is again unique up to scalar.

Finally, we prove that composition
\begin{equation}\label{eq:compU}
\Hom(\bbT_\lambda, \bbT_\mu)\otimes \Hom(\bbT_\mu, \bbT_\nu) \;\to\; \Hom(\bbT_\lambda, \bbT_\nu)
\end{equation}
is an isomorphism whenever both sides are non-zero (and thus one-dimensional). For this, we start by observing that, if $\bbT_{\lambda} \to \bbT_{\mu}$ is a non-zero morphism and we denote its image by $I$, then, for arbitrary $\lambda,\mu,\kappa$
\begin{displaymath}
\dim I(\kappa)\;=\;\dim \bbT_\lambda(\kappa)\,\dim \bbT_\mu(\kappa),
\end{displaymath}
where all dimensions involved are either 1 or 0; this follows from the discussion in the first paragraph.

There are three cases where both dimensions in \eqref{eq:compU} are $1$. The first is $\bbT_{\mu \beta}\to \bbT_{\mu} \to \bbT_{\mu \alpha}$ where $\beta \in \Lambda^{\alt}_{\ww}$ and $\alpha \in \Lambda^{\alt}_{\bb}$ have different parity, so that either $\mu \alpha$ is obtained from $\mu \beta$ by adding an element of $\Lambda^{\alt}$ or vice versa. We focus on the former option, the other case being similar, so that $\alpha$ and $\beta$ have differing parity and $\alpha$ is longer than $\beta$. In this case, $\bbT_{\rho}(\lambda \alpha)\not=0$ for $\rho\in\{\mu, \mu \alpha, \mu \beta\}$, showing non-vanishing of the composition. The second option is $\bbT_{\lambda} \to \bbT_{\lambda \alpha} \to \bbT_{\lambda \alpha \gamma}$ for $\alpha \in \Lambda^{\alt}_{\bb}$ and $\gamma \in \Lambda^{\alt}_{\bb}$, with the latter of even length (equivalently $\alpha\gamma \in \Lambda^{\alt}$). Again we can observe that $\bbT_{\rho}(\lambda \alpha \gamma)\not=0$ for $\rho=\{\lambda,\lambda \alpha ,\lambda \alpha \gamma\}$. The third case, about composing morphisms of the form $\bbT_{\lambda \beta}\to \bbT_{\lambda}$, is dual to the second.
\end{proof}

The following corollary is essentially a concise rephrasing of the proposition. It is the raison d'\^etre for the category $\fD$.
\begin{corollary}
The category $\fD$ is equivalent to the full subcategory of $\cB$ spanned by the $\bbT_{\lambda}$, via $\lambda \mapsto \bbT_{\lambda}$.
\end{corollary}

\begin{remark}
In Proposition~\ref{prop:B-ext}(f) below, we establish the expected $\Ext$ vanishing between (co)standard $\fB$-modules. As a consequence, we have the usual formula for dimensions of morphisms spaces between tilting modules:
\begin{displaymath}
\dim\Hom(\bbT_\lambda, \bbT_\mu)\;=\;\sum_{\nu} (\bbT_\lambda:\stan_\nu)(\bbT_\mu:\cost_\nu).
\end{displaymath}
The multiplicities in the right side can be computed from Proposition~\ref{prop:BTlambda}:
\begin{displaymath}
(\bbT_\lambda:\stan_\nu)=\begin{cases}
1 & \text{if $\nu \in \{\lambda\} \cup \lambda \Lambda^{\alt}_{\bb}$} \\
0 & \text{otherwise.}
\end{cases}
\quad\text{and}\quad
(\bbT_\mu:\cost_\nu)=\begin{cases}
1 & \text{if $\nu \in \{\mu\} \cup \mu \Lambda^{\alt}_{\ww}$} \\
0 & \text{otherwise.}
\end{cases}
\end{displaymath}
This provides another proof of the dimension formula in Proposition~\ref{prop:tilt-map}.
\end{remark}

\subsection{The tensor product} \label{ss:B-tensor}

Let $\Proj(\cB)$ denote the category of projective objects in $\cB$. As we have seen, the functor $\cA \to \Proj(\cB)$ given by $M_{\lambda} \mapsto \bbP_{\lambda}$ is an equivalence of categories. We can therefore transfer the tensor product from $\cA$ to $\Proj(\cB)$, and then extend this tensor product by cocontinuity to all of $\cB^{\inf}$. If we identify $\cB^{\inf}$ with the pre-sheaf category on $\cA$, this tensor product is Day convolution. Since $\Proj(\cB)$ is closed under the tensor product (by definition), it follows that the finite length category $\cB$ is as well.

Essentially by definition, the functor $\cA \to \cB$ is a tensor functor. Since tensor functors map rigid objects to rigid objects, and preserve duals, it follows that $\bbP_{\lambda}$ is a rigid object of $\cB$ with dual $\bbP_{\lambda^{\vee}}$. Since $\bbP_{\lambda}$ is rigid, it follows that the functor $\bbP_{\lambda} \otimes -$ is exact; in other words, $\bbP_{\lambda}$ is a flat module. Note that $\bbP_{\lambda}^{\vee}=\bbI_{\lambda^{\vee}}$ is not the monoidal dual of $\bbP_{\lambda}$. In general, $(-)^{\vee}$ does not interact well with the tensor product (in contrast to what we observed for $\cD$ in Remark~\ref{remarkdualD}). However, we do have the following instance of compatibility:

\begin{proposition}\label{prop:dualP}
For a $\fB$-module $X$ and a weight $\lambda$, we have a natural isomorphism
\begin{displaymath}
(\bbP_{\lambda} \otimes X)^{\vee} \cong \bbP_{\lambda^{\vee}} \otimes X^{\vee}.
\end{displaymath}
\end{proposition}

\begin{proof}
For a finite projective $\fB$-module $P$, let $P^*$ be its dual with respect to the tensor structure. Thus $\bbP_{\mu}^*=\bbP_{\mu^{\vee}}$. We have
\begin{displaymath}
\Hom(\bbP_{\mu}, X^{\vee}) = \Hom(\bbP_{\mu^{\vee}}, X)^* = \Hom(\bbP_{\mu}^*, X)^*,
\end{displaymath}
where in the first step we have used the mapping property for $\bbP_{\mu}$. It follows that for any $P$, we have a natural identification
\begin{displaymath}
\Hom(P, X^{\vee}) = \Hom(P^*, X)^*.
\end{displaymath}
If $Q$ is another finite projective $\fB$-module, then
\begin{align*}
\Hom(P, (Q \otimes X)^{\vee})
&= \Hom(P^*, Q \otimes X)^*
= \Hom(P^* \otimes Q^*, X)^* \\
&= \Hom(P \otimes Q, X^{\vee})
= \Hom(P, Q^* \otimes X^{\vee}).
\end{align*}
Taking $Q=\bbP_{\lambda}$ and $P=\bbP_{\mu}$ (with $\mu$ arbitrary) yields the result.
\end{proof}

\begin{remark} \label{rmk:dualP}
Taking $X=\bbP_{\varnothing}=\bbone$, we find $\bbI_{\lambda}=\bbP_{\lambda} \otimes \bbI_{\varnothing}$.
\end{remark}

The following result is not needed for the sequel, but we include it for completeness.

\begin{proposition} \label{prop:B-ten-hw}
Let $X$ be a pointwise finite $\fB$-module.
\begin{enumerate}
\item If $X$ has a standard filtration then so does $\bbP_{\lambda} \otimes X$.
\item If $X$ has a costandard filtration then so does $\bbP_{\lambda} \otimes X$.
\item If $X$ is a tilting module then so is $\bbP_{\lambda} \otimes X$.
\end{enumerate}
\end{proposition}

\begin{proof}
(a) The standard object $\stan_{\mu}$ has finite projective dimension in $\cB$, and so $\bbP_{\lambda} \otimes \stan_{\mu}$ also has finite projective dimension in $\cB$, and thus as a $\fB^+$-module as well, and so it has a standard filtration by Proposition~\ref{prop:stanfil-B}. The general case follows since $\bbP_{\lambda} \otimes -$ is exact and commutes with filtered colimits. (Note: it is not difficult to show that $X \otimes Y$ is pointwise finite when $X \in \cB$ and $Y \in \cB^{\pf}$.)

(b) This follows from (a) and duality. Precisely, we have
\begin{displaymath}
\bbP_{\lambda} \otimes X=(\bbP_{\lambda^{\vee}} \otimes X^{\vee})^{\vee}.
\end{displaymath}
Since $X^{\vee}$ has a standard filtration, so does its tensor product with $\bbP_{\lambda^{\vee}}$, and thus $\bbP_{\lambda} \otimes X$ has a costandard filtration.

(c) follows from (a) and (b).
\end{proof}

\begin{remark}
It would be interesting to know if the class of (co)standard filtered modules is closed under tensor product. We have verified some small cases, but do not have a complete proof.
\end{remark}

\subsection{Highest weight structure} \label{ss:tria-B}

Recall that the upwards category $\fB^+$ contains the identity morphisms and the morphisms $\lambda \to \lambda \ww$, and the downwards category $\fB^-$ is defined analogously. Any morphism $f \colon \lambda \to \mu$ factors uniquely (up to non-zero scalars) as $f=g \circ h$ where $h$ is in the downwards category and $g$ is the upwards category. (In particular, a composition of $g \circ h$ with $g \in \fB^+$ and $h \in \fB^-$ non-zero is again non-zero, as otherwise $f=0$ would have multiple factorizations.) It follows that $\fB$ is a triangular category in the sense of \cite[Definition~4.1]{brauercat}. The general theory of \cite{brauercat} provides constructions of (co)standard objects, which specialize to objects explicitly constructed above. We see that $\cB^{\pf}$ is an upper finite highest weight category, in the sense of \cite{BS}. The notation of (co)standard and tilting modules used in this section correspond to this highest weight structure.

\section{The derived category of $\cB$} \label{s:B2}

In this section we study the derived category of $\cB$. Our main result is Theorem~\ref{mainthm4}, which gives a semi-orthogonal decomposition of $\bD(\cB)$ into three pieces equivalent to $\bD(\Vec)$, $\bD(\cC)$ and $\bD(\cD)$. This semi-orthogonal decomposition is a key tool used in the subsequent section when we classify the local envelopes of $\cA$.

\subsection{Some homological calculations} \label{ss:homprop}

Since the projective objects of $\cB$ have finite length, any object has a resolution whose terms are finite length projectives. In particular, the $\Ext$ groups between any two objects are finite dimensional. We now compute some explicitly. To begin, we point out some easy resolutions. For any weight $\lambda$, we have a standard resolution
\begin{equation} \label{eq:resLD}
\cdots \to \stan_{\lambda\ww\ww} \to \stan_{\lambda\ww} \to \stan_{\lambda} \to \bbS_{\lambda} \to 0.
\end{equation}
If $\lambda \in \Lambda_{\ww} \cup \{\varnothing\}$ then this is a projective resolution. For any weight $\lambda$, we have a projective resolution
\begin{equation}\label{eq:resQP}
\cdots \to \bbP_{\lambda\bb\ww\ww} \to \bbP_{\lambda\bb\ww} \to \bbP_{\lambda\bb} \to \bbQ_{\lambda} \to 0.
\end{equation}
Write $\lambda=\mu \bb^n$ where $\mu \in \Lambda_{\ww} \cup \{\varnothing\}$ and $n\in\bN$. Then we have a finite projective resolution
\begin{equation}\label{eq:DP}
0 \to \bbP_{\mu} \to \cdots \to \bbP_{\mu \bb^{n-1}} \to \bbP_{\mu \bb^n} \to \stan_{\lambda} \to 0.
\end{equation}

\begin{proposition}\label{prop:ExtB}
Let $\lambda, \nu \in \Lambda$, and write $\lambda=\mu \bb^n$ where $\mu \in \Lambda_{\ww} \cup \{\varnothing\}$. Then:
\begin{enumerate}
\item $\Ext^i(\bbS_{\lambda}, \bbS_{\nu})$ vanishes unless $\nu=\mu \bb^{n-i}$ or $\nu=\lambda \ww^i$, in which case it is one dimensional.
\item $\Ext^i(\stan_{\lambda}, \bbS_{\nu})$ vanishes unless $\nu=\mu \bb^{n-i}$, in which case it is one dimensional.
\end{enumerate}
\end{proposition}
\begin{proof}
Part (b) follows immediately from \eqref{eq:DP}. Part (a) for $\lambda\not\in\Lambda_{\bb}$, {\it i.e.} $n=0$, follows from the projective resolution in \eqref{eq:resLD}. Finally, for $\lambda\in\Lambda_{\bb}$, we can consider the short exact sequence
\begin{displaymath}
0\to\stan_{\lambda^\flat}\oplus\bbS_{\lambda\ww}\to\bbP_\lambda\to \bbS_\lambda\to 0.
\end{displaymath}
We can thus reduce (a) for $\lambda\in\Lambda_{\bb}$ to (b) and (a) for $\lambda\not\in\Lambda_{\bb}$.
\end{proof}

\begin{remark}
The category $\cB$ admits a grading where the morphisms $\lambda\to\lambda\ww$ and $\lambda\bb\to\lambda$ are declared to be in degree 1. This grading makes $\cB$ quadratic. Moreover, as follows rapidly from Proposition~\ref{prop:ExtB} or its proof, it is actually Koszul in the sense of \cite{MOS}. In fact, it is even standard Koszul, meaning that the standard objects for the highest weight structure have linear projective resolutions.
\end{remark}

\begin{proposition} \label{prop:B-ext}
Consider $\lambda,\mu\in\Lambda$ and $i\in\bN$.
\begin{enumerate}
\item If $\lambda \in \Lambda_{\ww} \cup \{\varnothing\}$ then $\dim \Ext^i(\bbS_\lambda, \bbQ_\mu) = \delta_{i,0} \delta_{\lambda,\mu\ww}$.
\item We have $\dim \Ext^i(\cost_\lambda,\bbQ_\mu)=0$.
\item We have $\dim \Ext^i(\bbQ_\lambda,\bbQ_\mu)=\delta_{i,0}\delta_{\lambda,\mu}$.
\item We have $\dim \Ext^i(\bbI_\lambda,\bbQ_\mu)=0$.
\item We have $\dim\Ext^i(\bbI_{\lambda}, \bbS_{\varnothing})=0$.
\item We have $\dim \Ext^i(\stan_\lambda, \cost_\mu) = \delta_{i,0} \delta_{\lambda,\mu\ww}$.
\end{enumerate}
\end{proposition}

\begin{proof}
Observe that any composite of non-zero morphisms $\bbP_{\kappa\ww}\to\bbP_{\kappa}\to \bbQ_{\kappa}$ is not zero. Part (a) then follows from the projective resolution \eqref{eq:resLD} (for $\lambda$ not ending in $\bb$).
For (b), we have a short exact sequence
\begin{displaymath}
0\to \bbS_{\lambda\ww}\oplus\bbS_{\lambda\bb\ww}\to \bbP_{\lambda\bb}\to \cost_\lambda\to 0.
\end{displaymath}
Observing that 
\begin{displaymath}
\Hom(\bbP_{\lambda\bb},\bbQ_\mu)\to\Hom(\bbS_{\lambda\ww}\oplus \bbS_{\lambda\bb\ww},\bbQ_\mu)
\end{displaymath}
is an isomorphism, we find that morphisms and first extensions vanish. Moreover, for $i>0$, we have
\begin{displaymath}
\Ext^{i+1}(\cost_\lambda,\bbQ_\mu)\;\cong\;\Ext^{i}(\bbS_{\lambda\ww},\bbQ_\mu)\oplus \Ext^{i}(\bbS_{\lambda\bb\ww},\bbQ_\mu),
\end{displaymath}
meaning the claim reduces to part (a). For part (c) we take the short exact sequence
\begin{displaymath}
0\to \bbS_{\lambda\ww}\to \bbQ_\lambda\to \cost_{\lambda}\to 0,
\end{displaymath}
so the conclusion follows from the long exact sequence of homology, together with parts (a) and (b). Finally, part (d) follows from (b) and the long exact sequence of homology, since $\bbI_\lambda$ has a finite costandard filtration. Part (e) follows from the injective coresolution, dual to \eqref{eq:resLD},
\begin{displaymath}
0\to \bbS_{\varnothing}\to\bbI_\varnothing\to\bbI_{\ww}\to\bbI_{\ww\ww}\to\cdots.
\end{displaymath}
Part (f) is a standard property of highest weight categories, see \cite{BS}, but also follows directly from the projective resolution~\eqref{eq:DP} and its dual version for the costandard modules.
\end{proof}

\subsection{Semi-orthogonal decompositions} \label{ss:semi-orthog-bg}

We recall some generalities on semi-orthogonal decompositions; we refer to \cite{BK} for additional details. As in \cite[\S 4.3]{BK}, a (finite) semi-orthogonal decomposition of a triangulated category $\cT$ is a sequence $\cA_1,\cA_2,\ldots,\cA_n$ of strictly full triangulated subcategories such that:
\begin{enumerate}
\item We have $\Hom_{\cT}(X_i, X_j)=0$ for all $1 \le i < j \le n$ and $X_i \in \cA_{i}$ and $X_j \in \cA_j$.
\item $\cT$ is generated by $\cA_1,\ldots,\cA_n$, i.e., the smallest strictly full triangulated subcategory of $\cT$ containing $\cA_1,\ldots,\cA_n$ is equal to $\cT$.
\end{enumerate}
We abbreviate this to ${\cT}=\langle {\cA}_{1},\ldots ,{\cA}_{n}\rangle$.
The above structure is equivalent to the data of a filtration $0=\cT_n\subset\cT_{n-1}\subset\cT_{n-2}\subset\cdots\subset\cT_0=\cT$ by strictly full triangulated subcategories that is right admissible (see \cite{BK}), and $\cA_i$ is the right orthogonal of $\cT_{i}$ in $\cT_{i-1}$. In particular, having a semiorthogonal decomposition implies that every object of $\mathcal {T}$ has a canonical filtration whose graded pieces are (successively) in the subcategories ${\mathcal {A}}_{1},\ldots ,{\mathcal {A}}_{n}$. That is, for each object $X$ of ${\mathcal {T}}$, there is a sequence
\begin{displaymath}
0=X_{n}\to X_{n-1}\to \cdots \to X_{0}=X
\end{displaymath}
of morphisms in $ {\mathcal {T}}$ such that the cone of $ X_{i}\to X_{i-1}$ is in ${\mathcal {A}}_{i}$, for each $i$. 

\subsection{The decomposition of $\bD(\cB)$} \label{ss:semi-orthog}

We now return to the triangulated category $\bD(\cB)$, and define three strictly full triangulated subcategories:
\begin{itemize}
\item $\cB_1$ is the triangulated subcategory generated by $\bbS_{\varnothing}$.
\item $\cB_2$ is the triangulated subcategory generated by the $\bbQ_{\lambda}$'s.
\item $\cB_3$ is the triangulated subcategory generated by the $\bbI_{\lambda}$'s.
\end{itemize}
We now prove part of Theorem~\ref{mainthm4}; the remainder will be proved below.

\begin{theorem} \label{thm:semi-orthog}
We have semi-orthogonal decompositions
\begin{displaymath}
\bD(\cB) = \langle \cB_1, \cB_2, \cB_3 \rangle = \langle \cB_2, \cB_1, \cB_3 \rangle.
\end{displaymath}
\end{theorem}

\begin{proof}
By Proposition~\ref{prop:B-ext}(a,d,e), the groups
\begin{displaymath}
\RHom(\bbI_{\lambda}, \bbQ_{\mu}), \qquad
\RHom(\bbI_{\lambda}, \bbS_{\varnothing}), \qquad
\RHom(\bbQ_{\lambda}, \bbS_{\varnothing})\cong \RHom(\bbS_{\varnothing},\bbQ_{\lambda} )
\end{displaymath}
vanish, which proves the necessary semi-orthogonality.

To complete the proof, we must show that the $\cB_i$ generate $\bD(\cB)$. Let $\cX$ be the triangulated subcategory they generate. It suffices to show that $\cX$ contains all simple objects. It contains $\bbS_{\varnothing}$ by definition. Since costandard modules have finite injective resolutions, it follows that $\cost_{\lambda}$ belongs to $\cX$ for all $\lambda$. The short exact sequence
\begin{displaymath}
0 \to \bbS_{\lambda\ww} \to \bbQ_{\lambda} \to \cost_{\lambda} \to 0
\end{displaymath}
shows that $\cX$ contains $\bbS_{\lambda\ww}$ for any weight $\lambda$.  We have a short exact sequence
\begin{displaymath}
0 \to \bbS_{\lambda} \to \cost_{\lambda} \to \bbS_{\lambda \bb} \to 0.
\end{displaymath}
Thus if $\bbS_{\lambda}$ belongs to $\cX$ then so does $\bbS_{\lambda \bb}$. It now follows from an easy inductive argument on length that $\cX$ contains $\bbS_{\lambda}$ for all $\lambda \in \Lambda$.
\end{proof}

Henceforth, for instance in the next lemma, we will work with and refer to the semi-orthogonal decomposition $\langle\cB_1,\cB_2,\cB_3\rangle$. We will see later in this section that $\bD(\cB)$ is a monoidal subcategory of $\bD^+(\cB)$. We now determine the graded pieces of its tensor unit $\bbP_\varnothing$; these will play an important role.

\begin{proposition} \label{prop:3graded}
The graded pieces of $\bbP_\varnothing$ are $\bbS_\varnothing\in\cB_1$, $\bbQ_{\varnothing}\in\cB_2$ and $\bbI_\varnothing[-1]\in\cB_3$.
\end{proposition}

\begin{proof}
The short exact sequences
\begin{displaymath}
0\to \bbS_\ww\to \bbP_\varnothing\to \bbS_\varnothing\to 0\qquad\mbox{and}\qquad 0\to \bbS_\ww\to \bbQ_\varnothing\to \bbI_\varnothing\to 0,
\end{displaymath}
and their associated distinguished triangles in $\bD(\cB)$ show that $\bbS_\ww$ belongs to $\langle \cB_2,\cB_3\rangle$ and that the filtration of $\bbP_\varnothing$ is given by
\begin{displaymath}
0\to \bbI_\varnothing[-1]\to \bbS_\ww\to \bbP_\varnothing
\end{displaymath}
leading to the desired graded pieces.
\end{proof}

\subsection{The functor to $\cC$} \label{ss:Phi-shriek}

Recall from \S \ref{ss:functor} the faithful tensor functor $\Phi \colon \cA \to \cC$. Identifying $\cA$ with the category of projective objects in $\cB$, we let
\begin{displaymath}
\Phi_! \colon \cB \to \cC
\end{displaymath}
denote the unique right-exact extension of this functor, which has again a natural tensor structure. This functor is given on projectives by
\begin{displaymath}
\Phi_!(\bbP_{\lambda}) = L_{\lambda} \oplus L_{\lambda^{\flat}},
\end{displaymath}
and is faithful on the projective category. We write $\rL \Phi_! \colon \bD^+(\cB)\to\bD^+(\cC)$ for the left-derived functor of $\Phi_!$. 

\begin{proposition} \label{prop:LPhi}
Write $\lambda=\mu \bb^n$ where $\mu \in \Lambda_{\ww} \cup \{\varnothing\}$ and $n \in \bN$. Then
\begin{align*}
(\rL \Phi_!)(\bbS_{\lambda}) &= L_{\mu^{\flat}}[n] &
(\rL \Phi_!)(\stan_{\lambda}) &= L_{\lambda} \oplus L_{\mu^{\flat}}[n] \\
(\rL \Phi_!)(\bbQ_{\lambda}) &= L_{\lambda} &
(\rL \Phi_!)(\cost_{\lambda}) &= 0
\end{align*}
In particular $\rL \Phi_!$ restricts to a functor $\rL \Phi_!:\bD(\cB)\to\bD(\cC)$.
\end{proposition}

\begin{proof}
We first handle the standard modules. 
Applying $\Phi$ to the projective resolution in \eqref{eq:DP}, we see that $\rL \Phi_!(\stan_{\lambda})$ is the complex
\begin{displaymath}
\cdots \to 0 \to L_{\mu} \oplus L_{\mu^{\flat}} \to \cdots \to L_{\mu \bb^{n-1}} \oplus L_{\mu \bb^{n-2}} \to L_{\mu \bb^n} \oplus L_{\mu \bb^{n-1}}\to 0.
\end{displaymath}
Since $\Phi$ is faithful on the projective category, all differentials in this complex are non-zero. This yields the formula for $\rL \Phi_!(\stan_{\lambda})$.

We next handle simples. First suppose $\lambda \in \Lambda_{\ww} \cup \{\varnothing\}$. 
Applying $\Phi$ to the projective resolution \eqref{eq:resLD}, we see that $\rL \Phi_!(\bbS_{\lambda})$ is 
\begin{displaymath}
\cdots \to L_{\lambda \ww\ww} \oplus L_{\lambda \ww} \to L_{\lambda \ww} \oplus L_{\lambda} \to L_{\lambda} \oplus L_{\lambda^{\flat}}\to 0.
\end{displaymath}
Again, since the differentials are all non-zero, the result follows. Now suppose $\lambda \in \Lambda_{\bb}$. We have a presentation
\begin{displaymath}
\bbP_{\lambda\ww} \oplus \bbP_{\lambda^{\flat}} \to \bbP_{\lambda} \to \bbS_{\lambda} \to 0.
\end{displaymath}
Applying $\Phi$ and using the fact that $\Phi$ is faithful on projectives, we see that $\Phi(\bbS_{\lambda})=0$. Next, applying $\Phi$ to the short exact sequence \eqref{eq:sesDelta}, we obtain a 5-term exact sequence
\begin{displaymath}
0 \to \rL_1 \Phi_!(\stan_{\lambda}) \to \rL_1 \Phi_!(\bbS_{\lambda}) \to L_{\lambda} \to L_{\lambda} \to \Phi(\bbS_{\lambda}) \to 0
\end{displaymath}
and isomorphisms $\rL_i \Phi_!(\bbS_{\lambda})=\rL_i \Phi_!(\stan_{\lambda})$ for $i \ge 2$. Since $\Phi(\bbS_{\lambda})=0$, the map $L_{\lambda} \to L_{\lambda}$ is an isomorphism, and so the map $\rL_1 \Phi_!(\stan_{\lambda}) \to \rL_1 \Phi_!(\bbS_{\lambda})$ is also an isomorphism. This completes the calculation of $\rL \Phi_!(\bbS_{\lambda})$, from which it follows that $\rL \Phi_!$ preserves the bounded derived category.

We next handle the $\bbQ$ modules. For any weight $\lambda$, we have a projective resolution \eqref{eq:resQP} and applying $\Phi$, the result follows; this is quite similar to the case $\bbS_{\lambda}$ where $\lambda \in \Lambda_{\ww}$.

Finally, we treat the costandard modules. We have a short exact sequence
\begin{displaymath}
0 \to \bbS_{\lambda\ww} \oplus \bbS_{\lambda\bb\ww} \to \bbP_{\lambda\bb} \to \cost_{\lambda} \to 0
\end{displaymath}
and a presentation
\begin{displaymath}
\bbP_{\lambda\ww} \oplus \bbP_{\lambda\bb\ww} \to \bbP_{\lambda\bb} \to \cost_{\lambda} \to 0.
\end{displaymath}
Applying $\Phi$ to the presentation and using faithfulness on projectives, we see that $\Phi(\cost_{\lambda})=0$. Applying $\Phi$ to the short exact sequence now shows that $\rL_i \Phi_!(\cost_{\lambda})=0$ for $i>0$; note that since $\Phi(\cost_{\lambda})=0$, the map $\Phi(\bbS_{\lambda\ww}) \oplus \Phi(\bbS_{\lambda\bb\ww}) \to \Phi(\bbP_{\lambda\bb})$ is an isomorphism, as it is surjective and the source and target both have length two.
\end{proof}

The following result shows how $\rL \Phi_!$ interacts with the semi-orthogonal decomposition of $\bD(\cB)$ (\S \ref{ss:semi-orthog}). By the \defn{kernel} of an additive functor we mean the full subcategory spanned by the objects that are mapped to zero.

\begin{proposition} \label{prop:LPhi2}
We have the following:
\begin{enumerate}
\item The functor $\rL \Phi_!$ restricts to an equivalence $\cB_2 \to \bD(\cC)$.
\item The kernel of $\rL \Phi_!$ is $\langle \cB_1, \cB_3 \rangle$.
\end{enumerate}
\end{proposition}

\begin{proof}
(a) follows from Propositions~\ref{prop:B-ext}(c) and~\ref{prop:LPhi}. Proposition~\ref{prop:LPhi} also shows that $\cB_1$ and $\cB_2$ are contained in the kernel of $\rL \Phi_!$. We now show that they exactly make the kernel. Thus suppose $X$ is an object of $\bD(\cB)$ that is killed by $\rL \Phi_!$. From the semi-orthogonal decomposition $\langle \cB_2, \cB_1, \cB_3 \rangle$ of $\bD(\cB)$ we have a triangle
\begin{displaymath}
X' \to X'' \to X \to
\end{displaymath}
where $X'$ belongs to $\cB_2$ and $X''$ to $\langle \cB_1, \cB_3 \rangle$. Since $\rL \Phi_!$ kills $X$ and $X''$, it also kills $X'$. But $\rL \Phi_!$ restricts to an equivalence on $\cB_2$. Thus $X'=0$, and so $X \cong X''$, which shows that $X$ belongs to $\langle \cB_1, \cB_3 \rangle$.
\end{proof}

We say that a $\fB$-module is \defn{white} if its simple constituents are of the form $\bbS_{\lambda}$ with $\lambda \in \Lambda_{\ww} \cup \{\varnothing\}$. Note that if $\lambda$ is such a weight then the projective $\bbP_{\lambda}$ is a white module. The following result shows that we can (nearly) test exactness of white sequences by applying $\Phi_!$; this will be used to study the derived tensor product on $\cB$ in \S \ref{ss:derten}.

\begin{corollary}\label{cor:exactmiddle}
Consider a sequence $M_1 \to M_2 \to M_3$ of white $\fB$-modules that composes to zero. The following are equivalent:
\begin{enumerate}
\item The homology of the sequence at $M_2$ is a direct sum of $\bbS_{\varnothing}$'s.
\item The sequence $\Phi_!(M_1) \to \Phi_!(M_2) \to \Phi_!(M_3)$ is exact in the middle.
\end{enumerate}
\end{corollary}

\begin{proof}
It follows from Proposition~\ref{prop:LPhi} that $\Phi_!$ sends left, resp.\ right, exact sequences of white modules to left, resp.\ right, exact sequences. Moreover, the only simple white module that $\Phi_!$ sends to zero is $\bbS_\varnothing$, from which the claim follows. One can make this more explicit by observing that the only indecomposable white objects are $\bbDelta_\lambda$ and $\bbS_\lambda$, with $\lambda\in\Lambda_{\ww}\cup\{\varnothing\}$.
\end{proof}

\subsection{The functor to $\cD$} \label{ss:Psi-shriek}

By construction of the tensor product on $\cD$ in \S \ref{sec:tensorD}, we have a fully faithful tensor functor $\Psi\colon \cA\to\cD$ satisfying $M_\lambda\mapsto \bT_\lambda$. Identifying $\cA$ with $\Proj(\cB)$, we again have a right-exact extension
\begin{displaymath}
\Psi_! \colon \cB \to \cD,
\end{displaymath}
which is a tensor functor. This functor satisfies $\Psi_!(\bbP_{\lambda})=\bT_{\lambda}$. We write $\rL \Psi_! \colon \bD^+(\cB) \to \bD^+(\cD)$ for the left derived functor of $\Psi$.

\begin{proposition}\label{prop:LPsi}
We have the following:
\begin{enumerate}
\item If $\lambda \in \Lambda_{\ww}$ then $\rL \Psi_!(\bbS_{\lambda})=\bDelta_{\lambda^{\flat}}$ and $\rL \Psi_!( \stan_{\lambda})=\bT_{\lambda}=\bnabla_\lambda$.
\item If $\lambda \in \Lambda_{\bb} \cup \{\varnothing\}$ then $\rL \Psi_!(\bbS_{\lambda}) = \ker(\bDelta_{\lambda} \to \bS_{\lambda})[1]$ and $\rL \Psi_!(\stan_{\lambda})=\bS_{\lambda}=\bnabla_\lambda$.
\item For $i \ge 2$, we have $\rL_i \Psi_!=0$ identically. 
\item For any $\lambda\in\Lambda$, we have $\rL\Psi_!(\bbQ_\lambda)=0$.
\end{enumerate}
In particular, $\rL\Psi_!$ restricts to a functor $\rL\Psi_! \colon \bD(\cB)\to\bD(\cD)$.
\end{proposition}

\begin{proof}
(a) Applying $\Psi_!$ to the projective resolution \eqref{eq:resLD}, we find that $\rL \Psi_!(\bbS_\lambda)$ is given by a complex
\begin{displaymath}
\cdots\to \bT_{\lambda\ww\ww}\to \bT_{\lambda\ww}\to \bT_\lambda\to 0,
\end{displaymath}
where all these tilting modules are costandard modules. Since $\Psi_!$ is faithful on projectives, all differentials are non-zero. By the dual of Proposition~\ref{prop:uniserial}, we know that, for $i>0$, the module $\bT_{\lambda\ww^i}$ has length two with socle $\bS_{\lambda\ww^i}$ and top $\bS_{\lambda\ww^{i-1}}$. The desired result thus follows from Proposition~\ref{prop:D-tilt}(a). The second claim in (a) follows immediately from $\stan_\lambda=\bbP_\lambda$.

(b) Comparison of the resolutions in \eqref{eq:DP} and Corollary~\ref{Cor:TiltRes}(a) shows that $\rL \Psi_!(\stan_{\lambda})=\bS_{\lambda}$. Next, short exact sequence \eqref{eq:sesDelta} and the results obtained thus far yields an exact sequence
\begin{displaymath}
0\to \rL_1 \Psi_!(\bbS_\lambda)\to \bDelta_\lambda\to \bS_\lambda\to \Psi_!(\bbS_\lambda) \to 0.
\end{displaymath}
It thus suffices to show that the middle arrow is not zero. However, since we can pre- and post-compose $\bbS_{\lambda\ww}\hookrightarrow \stan_\lambda$ to obtain the non-zero morphism $\bbP_{\lambda\ww}\to\bbP_{\lambda\bb}$, faithfulness of $\Psi_!$ on projectives completes the argument.

(c) This statement follows immediately from (a) and (b). This also yields the final statement about $\rL \Psi_!$ preserving the bounded derived categories.

(d) The projective resolution of $\bbQ_\lambda$ in \eqref{eq:resQP} is sent to an acyclic complex in $\bD^+(\cD)$ under $\Psi_!$, due to Corollary~\ref{Cor:TiltRes}.
\end{proof}

\begin{remark}
The functor $\Psi_!$ can be interpreted as a Ringel duality functor, explaining the observation $\Psi_!(\stan_\lambda)=\bnabla_\lambda$ for all $\lambda\in\Lambda$.
\end{remark}

We now see how $\rL \Psi_!$ interacts with the semi-orthogonal decomposition of $\bD(\cB)$ (\S \ref{ss:semi-orthog}).

\begin{proposition}\label{prop:LPsi2}
We have the following:
\begin{enumerate}
\item The functor $\rL \Psi_!$ restricts to an equivalence $\cB_3 \to \bD(\cD)$.
\item The kernel of $\rL \Psi_!$ is $\langle \cB_1, \cB_2 \rangle$.
\end{enumerate}
\end{proposition}

\begin{proof}
(a) Proposition~\ref{prop:LPsi}(b,d) shows that $\langle \cB_1, \cB_2 \rangle$ is contained in the kernel of $\rL \Psi_!$, and so it follows from Proposition~\ref{prop:3graded} that 
\begin{equation} \label{eq:I1}
\rL \Psi_!(\bbI_{\varnothing}) = \rL\Psi_!(\bbP_\varnothing[1]) = \bbone[1].
\end{equation}
Now, we have observed $\bbI_{\lambda} = \bbP_{\lambda} \otimes \bbI_{\varnothing}$ (Remark~\ref{rmk:dualP}). In fact, this can be upgraded to a functorial isomorphism
\begin{displaymath}
I = (I^{\vee})^* \otimes \bbI_{\varnothing}
\end{displaymath}
for $I$ an injective object of $\cB$. Note that $I^{\vee}$ is a projective object, and thus has a monoidal dual. Applying the monoidal functor $\rL \Psi_!$, we find
\begin{displaymath}
\rL \Psi_!(I) = \rL \Psi_!(I^{\vee})^*[1].
\end{displaymath}
We thus see that $\rL \Psi_!(\bbI_{\lambda})=\bT_{\lambda}[1]$. Moreover, since $\rL \Psi_!$ is fully faithful on the projective category $\Proj(\cB)$, the above identification shows that it is also fully faithful on the category $\Inj(\cB)$ of injective objects in $\cB$. It follows that $\rL \Psi_!$ induces an equivalence
\begin{displaymath}
\bK(\Inj\cB)\to \bK(\Tilt\cD).
\end{displaymath}
The source is canonically identified with $\cB_3$, and the target with $\bD(\cD)$ (Proposition~\ref{prop:equiv:triangle}).

(b) We have already remarked that $\langle \cB_1, \cB_2 \rangle$ is contained in the kernel of $\rL \Psi_!$. The same argument used in Proposition~\ref{prop:LPhi2} thus shows that the two coincide.
\end{proof}

\subsection{The semi-simplification functor} \label{ss:Theta-shriek}

Recall (\S \ref{ss:ss}) the semi-simplification tensor functor $\Theta \colon \cA \to \Vec$. As with $\Phi$ and $\Psi$, this admits a unique right-exact extension
\begin{displaymath}
\Theta_! \colon \cB \to \Vec.
\end{displaymath}
This functor is given on projective modules by $\Theta_!(\bbP_{\varnothing}) = k$ and $\Theta_!(\bbP_{\lambda}) = 0$ if $\lambda$ is a non-empty weight. We now compute the values of its derived functors in some cases.

\begin{proposition} \label{prop:LTheta}
The functor $\rL \Theta_!$ kills the modules $\bbQ_{\lambda}$ and $\cost_{\lambda}$ for any weight $\lambda$, as well as the modules $\bbS_{\lambda}$ and $\stan_{\lambda}$ unless $\lambda=\bb^n$. If $\lambda=\bb^n$ then $\rL \Theta_!(\bbS_{\lambda})=\rL \Theta_!(\stan_{\lambda}) = k[n]$. In particular, $\rL\Theta_!$ restricts to a functor $\bD(\cB)\to\bD(\Vec)$.
\end{proposition}

\begin{proof}
We begin with standard modules. Let $\lambda=\mu \bb^n$ where $\mu \in \Lambda_{\ww} \cup \{\varnothing\}$, so that we have a projective resolution
\begin{displaymath}
\cdots \to 0 \to \bbP_{\mu} \to \cdots \to \bbP_{\mu \bb^{n-1}} \to \bbP_{\mu \bb^n} \to \stan_{\lambda} \to 0.
\end{displaymath}
Applying $\Theta_!$, we see that all terms vanish, unless $\mu$ is empty in which case $\Phi(\bbP_{\mu})=k$ is the unique non-zero term. The result thus follows.

We now handle simple modules. If $\lambda \in \Lambda_{\ww}$ then applying $\Theta_!$ to the projective resolution \eqref{eq:resLD}
shows that $\rL \Theta_!(\bbS_{\lambda})=0$. If $\lambda \in \Lambda_{\bb} \cup \{\varnothing\}$ then applying $\Theta_!$ to the short exact sequence \eqref{eq:sesDelta}
shows $\rL \Theta_!(\bbS_{\lambda})=\rL \Theta_!(\stan_{\lambda})$, as required.
Since $\Theta_!$ kills each projective in the resolution \eqref{eq:resQP} of $\bbQ_\lambda$, we find $\rL \Theta_!(\bbQ_{\lambda})=0$.

Finally we come to the costandard modules. We have a short exact sequence
\begin{displaymath}
0 \to \bbS_{\lambda\ww} \to \bbQ_{\lambda} \to \cost_{\lambda} \to 0,
\end{displaymath}
from which we find $\rL \Theta_!(\cost_{\lambda})=0$, as required.
\end{proof}

We now examine how $\rL \Theta_!$ interacts with the semi-orthogonal decomposition of $\bD(\cB)$.

\begin{proposition}\label{prop:LTheta2}
We have the following:
\begin{enumerate}
\item The functor $\rL \Theta_!$ restricts to an equivalence $\cB_1 \to \bD(\Vec)$.
\item The kernel of $\rL \Theta_!$ is $\langle \cB_2, \cB_3 \rangle$
\end{enumerate}
\end{proposition}

\begin{proof}
Since $\Ext^i(\bbS_{\varnothing}, \bbS_{\varnothing})=0$ for $i>0$, see Proposition~\ref{prop:ExtB}(a) and $\rL \Theta_!(\bbS_{\varnothing})=k$, statement (a) follows. Proposition~\ref{prop:LTheta} shows that $\cB_2$ and $\cB_3$ are contained in the kernel of $\rL \Theta_!$. Arguing as in the proof of Proposition~\ref{prop:LPhi2}, we see that $\langle \cB_2, \cB_3 \rangle$ is exactly the kernel.
\end{proof}

\subsection{The derived tensor product} \label{ss:derten}

Recall from \S \ref{ss:B-tensor} that the category $\cB$ admits a tensor product $\otimes$, which is right exact. We let $\otimes^{\rL}$ be its left-derived functor. This gives the bounded above derived category $\bD^+(\cB)$ a symmetric monoidal structure. We first observe that the bounded derived category $\bD(\cB)$ is preserved.

\begin{proposition} \label{prop:bounded-tor}
$\bD(\cB)$ is a monoidal subcategory of $\bD^+(\cB)$.
\end{proposition}

\begin{proof}
We must show that if $X$ and $Y$ belong to $\cB$ then $\Tor_i(X,Y)=0$ for $i \gg 0$, where $\Tor_i(-,-)$ is the $i$th homology of the derived functor $\otimes^{\rL}$. It suffices, by d\'evissage, to treat the case where $X$ and $Y$ are simple.

First, observe that if $\lambda$ is a non-empty weight then $\bbP_{\lambda} \otimes \bbS_{\varnothing}$ vanishes. Indeed, we have
\begin{displaymath}
\Hom(\bbP_{\mu}, \bbP_{\lambda} \otimes \bbS_{\varnothing}) =
\Hom(\bbP_{\lambda^{\vee}} \otimes \bbP_{\mu}, \bbS_{\varnothing})
\end{displaymath}
by adjunction, and $\bbP_{\varnothing}$ is not a summand of $\bbP_{\lambda^{\vee}} \otimes \bbP_{\mu}$ by Proposition~\ref{prop:notsummand}. Thus if $\lambda \in \Lambda_{\ww}$ then $\Tor_i(\bbS_{\lambda}, \bbS_{\varnothing})=0$ for all $i \ge 0$ by \eqref{eq:resLD}, while if $\lambda=\varnothing$ then the same vanishing holds for $i \ge 1$.

Next, let $\lambda, \mu \in \Lambda_{\ww}$, and consider the projective resolutions $P^{\bullet} \to \bbS_{\lambda}$ and $Q^{\bullet} \to \bbS_{\mu}$ as in \eqref{eq:resLD}. We have $\Phi_!(P^{\bullet}) = L_{\lambda^{\flat}}[0]$ in $\bD(\cC)$ by Proposition~\ref{prop:LPhi}, and similarly $\Phi_!(Q^{\bullet}) = L_{\mu^{\flat}}[0]$. Thus, using exactness of the tensor product in $\cD$:
\begin{displaymath}
\Phi_!(P^\bullet\otimes Q^\bullet) = \Phi_!(P^\bullet) \otimes \Phi_!(Q^\bullet) = (L_{\lambda^\flat}\otimes L_{\mu^\flat})[0].
\end{displaymath}
We thus see that $\Phi_!(P^\bullet\otimes Q^\bullet)$ only has cohomology in degree~0. Now, the terms of $P^{\bullet}$ and $Q^{\bullet}$ are white projectives (see \S \ref{ss:Phi-shriek}). Hence, by the tensor product rule for $\cA \cong \Proj(\cB)$ (specifically Corollary~\ref{cor:notsummand2}), the terms of $P^{\bullet} \otimes Q^{\bullet}$ are white projectives. Appealing to Corollary~\ref{cor:exactmiddle}, we see that the cohomology of $P^{\bullet} \otimes Q^{\bullet}$ only has $\bbS_{\varnothing}$ in it. However, $\bbP_{\varnothing}$ and $\bbP_{\bb}$ (the only two projectives that have $\bbS_\varnothing$ as a simple constituent) do not appear in these resolutions. Indeed the latter has already been remarked, and the former follows from Proposition~\ref{prop:notsummand}. Consequently, $\Tor_i(\bbS_{\lambda}, \bbS_{\mu})=0$ for all $i>0$.

For a weight $\lambda$, let $b(\lambda)$ be the number of $\bb$ that appear at the end of $\lambda$; that is, $b(\lambda)=n$ if $\lambda=\mu \bb^n$ where $\mu \in \Lambda_{\ww} \cup \{\varnothing\}$. We prove finiteness of $\Tor_*(\bbS_{\lambda}, \bbS_{\mu})$ by induction on $b(\lambda)+b(\mu)$. The base case where the sum is zero follows from the above two paragraphs. Assume now we have proved the result when the sum is $\le b$, and consider $\lambda$ and $\mu$ with sum $b+1$. Without loss of generality, we assume $\lambda \in \Lambda_{\bb}$. Consider the short exact sequence
\begin{displaymath}
0\to K \to \bbP_\lambda\to\bbS_\lambda\to 0,
\end{displaymath}
where $K=\stan_{\lambda^\flat}\oplus \bbS_{\lambda\ww}$. All simple constituents of $K$ have the form $\bbS_{\nu}$ with $b(\nu)<b(\lambda)$. Thus, by induction, $\Tor_i(K, \bbS_{\mu})$ vanishes for $i \gg 0$. Since the projective $\bbP_{\lambda}$ is flat, $\Tor_i(\bbP_{\lambda}, \bbS_{\mu})$ vanishes for $i>0$. Thus $\Tor_i(\bbS_{\lambda}, \bbS_{\mu})$ vanishes for $i \gg 0$, as required.
\end{proof}

We examine how the monoidal structure on $\bD(\cB)$ interacts with the semi-orthogonal decomposition. Let $U_i \in \cB_i$ be the $i$th graded piece of the unit $\bbP_{\varnothing}$ (see Proposition~\ref{prop:3graded}).

\begin{theorem}\label{thm:tensor-SOD}
We have the following (for $i,j \in \{1,2,3\}$):
\begin{enumerate}
\item The category $\cB_i$ is closed under the tensor product on $\bD(\cB)$. It forms a rigid symmetric monoidal category with unit object $U_i$.
\item If $X\in\cB_i$ and $Y\in\cB_j$ with $i \ne j$ then $X\otimes^{\rL}Y=0$.
\item The equivalences $\cB_1\to\bD(\Vec)$, $\cB_2\to \bD(\cC)$ and $\cB_3\to\bD(\cD)$ from Propositions~\ref{prop:LTheta2}, \ref{prop:LPsi2} and~\ref{prop:LPhi2} are naturally symmetric monoidal.
\end{enumerate}
\end{theorem}

Before giving the proof, we require a lemma.

\begin{lemma}\label{lem:joint-kernel}
If an object of $\bD(\cB)$ is sent to zero by each of $\rL \Phi_!$, $\rL \Psi_!$, and $\rL \Theta_!$, it must be zero itself. Furthermore, if an object is sent to zero by $\rL \Theta_!$ and $\rL \Psi_!$ it belongs to $\cB_2$, and similarly for other pairs.
\end{lemma}

\begin{proof}
This follows quickly from Propositions~\ref{prop:LPhi2}, \ref{prop:LPsi2} and \ref{prop:LTheta2}. Indeed, if $X$ is in the kernel of $\rL \Theta_!$ it belongs to $\langle\cB_2,\cB_3\rangle$, so that its filtration looks like
\begin{displaymath}
0\to X_2\to X_1=X_0=X
\end{displaymath}
for some $X_2\in \cB_3$ and with the cone $C$ of $X_2\to X_1$ in $\cB_2$. We have $\rL\Psi_!(X_2)\to\rL\Psi_!(X)$ is an isomorphism. So if also $\Psi_!(X)=0$, it follows from the equivalence $\rL\Psi_! \colon \cB_3\to\bD(\cD)$ that $X_2=0$, and thus $X=X_1=C\in \cB_2$. Finally $\rL\Phi_!(X)=0$ will then similarly imply $X=X_1=C=0$.
\end{proof}

\begin{proof}[Proof of Theorem~\ref{thm:tensor-SOD}]
If $X,Y \in \cB_1$ then $X \otimes^{\rL} Y$ is killed by $\rL \Psi_!$ and $\rL \Theta_!$, and thus belongs to $\cB_1$ by Lemma~\ref{lem:joint-kernel}. Thus $\cB_1$ is closed under the tensor product. Similarly for $\cB_2$ and $\cB_3$. If $X \in \cB_i$ and $Y \in \cB_j$ with $i \ne j$ then $X \otimes^{\rL} Y$ is killed by each of $\rL \Phi_!$, $\rL \Psi_!$, and $\rL \Theta_!$, and thus vanishes by Lemma~\ref{lem:joint-kernel}. This proves (b).

Now, $\rL \Phi_! \colon \bD(\cB) \to \bD(\cC)$ is a symmetric monoidal functor, and so its restriction to $\cB_1$ is compatible with tensor product. Since $\rL \Phi_! \colon \cB_1 \to \bD(\cC)$ is an equivalence that is compatible with tensor product and $U_1$ maps to the unit of $\bD(\cC)$ (Proposition~\ref{prop:LPhi}), it follows that $U_1$ is the unit of $\cB_1$; thus $\cB_1$ is symmetric monoidal, as is the equivalence with $\bD(\cC)$. Since $\bD(\cC)$ is rigid, so is $\cB_1$. Similarly for $\cB_2$ and $\cB_3$ (using Propositions~equation~\ref{eq:I1} and~\ref{prop:LTheta} to see that $U_2$ and $U_3$ map to the unit). This proves (a) and (c).
\end{proof}

\begin{remark}
The subcategory $\cB_i$ is not a monoidal subcategory of $\bD(\cB)$, since the unit $U_i$ of $\cB_i$ is not the unit of $\bD(\cB)$. The $U_i$'s are orthogonal idempotent algebras, in the sense of \cite[Definition~2.3.1]{KVY}.
\end{remark}

\subsection{The Grothendieck ring}

By Proposition~\ref{prop:bounded-tor}, $\rK(\cB)$ is naturally a commutative ring. We can use the results obtained above to determine its structure:

\begin{proposition}
We have a ring isomorphism
\begin{displaymath}
(\rL \Phi_!, \rL \Psi_!, \rL \Theta_!) \colon \rK(\cB) \to \rK(\cC)\times \rK(\cD)\times \rK(\Vec).
\end{displaymath}
\end{proposition}

\begin{proof}
The three functors induce a ring homomorphism since they are symmetric monoidal and preserve the bounded derived categories, by Propositions~\ref{prop:LPhi},~\ref{prop:LPsi}, and~\ref{prop:LTheta}. Call this map $i$. Let $S \subset \rK(\cB)$ be the set consisting of the classes $[\bbI_{\lambda}]$ and $[\bbQ_{\lambda}]$, for $\lambda \in \Lambda$, together with the class $[\bbS_{\varnothing}]$. In the course of proving Theorem~\ref{thm:semi-orthog}, we show that these objects generate $\bD(\cB)$, and so it follows that $S$ generates $\rK(\cB)$ as a $\bZ$-module. We have
\begin{displaymath}
i([\bbQ_{\lambda}]) = ([L_{\lambda}], 0, 0), \quad
i([\bbI_{\lambda}]) = (0, -[\bT_{\lambda}], 0), \quad
i([\bbS_{\varnothing}]) = (0, 0, [k])
\end{displaymath}
by the same results used above, and also the proof of Proposition~\ref{prop:LPsi2}. It follows that $i$ maps $S$ to a $\bZ$-basis of the target (note that the $[\bT_{\lambda}]$'s are a basis of $\rK(\cD)$ by Proposition~\ref{prop:groth-D}). The result thus follows.
\end{proof}

\section{Local abelian envelopes} \label{s:env}

In this final section, we prove Theorem~\ref{mainthm3}, which states that the functors $\Phi$ and $\Psi$ account for all the local abelian envelopes of $\cA$. We also prove some additional results, e.g., we give a universal property for $\cD$.

\subsection{The main theorem}

Recall that we have three tensor functors
\begin{displaymath}
\Phi \colon \cA \to \cC, \qquad
\Psi \colon \cA \to \cD, \qquad
\Theta \colon \cA \to \Vec.
\end{displaymath}
We now show that these three functors are universal for appropriate tensor functors. To state our result, we introduce some terminology. Suppose $F \colon \cA \to \cT$ is a tensor functor to a pre-Tannakian category $\cT$. We say that $F$ has \defn{type $\Psi$} if there is an exact tensor functor $G \colon \cD \to \cT$ such that $F=G \circ \Psi$. We similarly define type $\Phi$ and type $\Theta$.

\begin{theorem} \label{thm:env}
Let let $F \colon \cA \to \cT$ be a tensor functor to a pre-Tannakian category $\cT$. Then $F$ has exactly one of type $\Phi$, type $\Psi$, or type $\Theta$.
\end{theorem}

We fix $F \colon \cA \to \cT$ for the duration of the proof. We let $F_! \colon \cB \to \cT$ be the unique right exact extension of $F$, where we identify $\cA$ with the category of projective objects in $\cB$, and we let $\rL F_! \colon \bD^+(\cB) \to \bD^+(\cT)$ be the left derived functor of $F_!$. If $F$ has type $\Phi$ then $F_!=G \circ \Phi_!$ and $\rL F_!=G \circ \rL \Phi_!$, and analogous statements hold for the other two types. Thus it is a consequence of the theorem that $\rL F_!$ preserves the bounded derived category, though this is not clear a priori.

Recall that $\bbP_{\varnothing}$ is the unit of $\bD(\cB)$ for the monoidal structure $\otimes^{\rL}$; we therefore denote it by $\bbone$ in what follows. Also, recall that
\begin{displaymath}
U_1=\bbS_{\varnothing}, \qquad U_2=\bbQ_{\varnothing}, \qquad U_3=\bbI_{\varnothing}[-1]
\end{displaymath}
are the graded pieces of $\bbone$ (Proposition~\ref{prop:3graded}), and $U_i$ is the monoidal unit of $\cB_i$ (Theorem~\ref{thm:tensor-SOD}). 
\begin{lemma} \label{lem:env-1}
The functor $\rL F_!$ kills exactly two of the $U_i$'s, and sends the third to the tensor unit of $\cT$. 
\end{lemma}

\begin{proof}
By standard properties of pre-Tannakian categories, a tensor product of two non-zero objects of $\cT$ is non-zero. It follows that the same property holds in $\bD^+(\cT)$ as well. That at least two of the three objects are killed thus follows from Theorem~\ref{thm:tensor-SOD}. Since $\bbP_{\varnothing}$ is the tensor unit of $\cB$, we have $F_!(\bbP_{\varnothing})=\bbone$, and since it is projective we have $\rL F_!(\bbP_{\varnothing})=\bbone$. It now follows that the $U_i$ which is not killed must map to $\bbone$ under $\rL F_!$.
\end{proof}

Note that $\rL \Phi_!$ kills $U_1$ and $U_3$, so if $F$ has type $\Phi$ then it does as well. Similarly for the other types. This shows that $F$ has at most one of the three types. Moreover, if $\rL F_!$ kills $U_1$ then $\rL F_!$ kills every object of $\cB_1$, since $U_1$ is the unit of $\cB_1$ (Theorem~\ref{thm:tensor-SOD}); similarly for $U_2$ and $U_3$. We will use this fact freely.

\begin{lemma} \label{lem:env-2}
Suppose $\rL F_!$ kills $U_1$ and $U_3$. Then $F$ has type $\Phi$.
\end{lemma}

\begin{proof}
Let $R \colon \cB_2 \to \bD(\cC)$ be the restriction of $\rL \Phi_!$. This is an equivalence of tensor categories (Proposition~\ref{prop:LPhi2}). Define a composite tensor functor
\begin{displaymath}
G' \colon \xymatrix@C=3em{
\bD(\cC) \ar[r]^-{R^{-1}} & \cB_2 \ar[r]^-{\rL F_!} & \bD^+(\cT). }
\end{displaymath}
The second functor is a tensor functor since $\rL F_!$ maps the unit $U_2$ of $\cB_2$ to the unit of $\cT$.

We claim that we have a natural isomorphism
\begin{equation} \label{eq:env-2}
\rL F_! \cong G' \circ \rL \Phi_!
\end{equation}
of tensor functors $\bD(\cB) \to \bD^+(\cT)$. To see this, first observe that the functor $\bD(\cB) \to \cB_2$ given by $R^{-1}\circ \rL \Phi_!$ is isomorphic to the functor $U_2 \otimes -$ (as tensor functors), as the two functors are isomorphic after post-composing with $R$. We thus have an isomorphism
\begin{displaymath}
G' \circ \rL \Phi_! \cong \rL F_!(U_2 \otimes -).
\end{displaymath}
However, this is isomorphic to $\rL F_!$, since $\rL F_!$ is monoidal and $\rL F_!(U_2) = \bbone$. This establishes the claim.

We next claim that $G'$ is t-exact. To prove this, it suffices to show that $G'$ maps simple objects of $\cC$ into $\cT$. First note that $G'(\bbone)=\bbone$ belongs to $\cT$ since $G'$ is monoidal. Next suppose $\lambda$ is a non-empty weight. We have
\begin{displaymath}
G'(L_{\lambda}) = \rL F_!(R(L_{\lambda})) = \rL F_!(\bbQ_{\lambda}),
\end{displaymath}
where we use Proposition~\ref{prop:LPhi} to compute $R(L_{\lambda})$. We have a short exact sequence
\begin{displaymath}
0 \to \bbP_{\lambda} \to \bbQ_{\lambda} \oplus \bbQ_{\lambda^{\flat}} \to \bbI_{\lambda} \to 0
\end{displaymath}
in $\cB$ from Lemma~\ref{lem:PQI}. Since $\rL F_!$ kills $U_3$, it kills all of $\cB_3$, and in particular kills $\bbI_{\lambda}$. We thus have a quasi-isomorphism
\begin{displaymath}
\rL F_!(\bbP_{\lambda}) \cong \rL F_!(\bbQ_{\lambda}) \oplus \rL F_!(\bbQ_{\lambda^{\flat}}).
\end{displaymath}
Since $\rL F_!(\bbP_{\lambda}) = F_!(\bbP_{\lambda})$ belongs to $\cT$, so does $\rL F_!(\bbQ_{\lambda})$. This establishes the claim.

We thus see that $G'$ is the derived functor of an (exact) tensor functor $G \colon \cC \to \cT$. Applying $\rH^0$ to \eqref{eq:env-2}, we have an isomorphism $F_! = G \circ \Phi_!$ of tensor functors $\cB \to \cT$. Restricting along the functor $\cA \to \cB$ yields an isomorphism $F=G \circ \Phi$, as required.
\end{proof}

\begin{lemma} \label{lem:env-3}
Suppose $\rL F_!$ kills $U_1$ and $U_2$. Then $F$ has type $\Psi$.
\end{lemma}

\begin{proof}
Since $\rL F_!$ kills $U_2$, it kills all of $\cB_2$, and thus each $\bbQ_{\lambda}$. Recall from \S \ref{ss:homprop} that we have a projective resolution
\begin{displaymath}
\cdots \to \bbP_{\lambda \bb\ww\ww} \to \bbP_{\lambda \bb\ww} \to \bbP_{\lambda \bb} \to \bbQ_{\lambda} \to 0
\end{displaymath}
in $\cB$. We thus see that applying $F_!$ to this resolution yields an exact complex. In other words, the complex
\begin{displaymath}
\cdots \to F(M_{\lambda \bb\ww\ww}) \to F(M_{\lambda \bb\ww}) \to F(M_{\lambda \bb}) \to 0
\end{displaymath}
is exact. Similarly, since $\rL F_!$ kills $U_1=\bbS_{\varnothing}$, the evaluation of $F$ on the complex $\cdots \to M_\ww\to M_\varnothing\to0$ is exact, so that $F$ satisfies the condition (SF) from \S \ref{ss:D-map}. Since $F$ is a tensor functor to a rigid tensor category, it follows that $F$ also satisfies (CF) (Remark~\ref{rmk:exact-extension}). We thus see that $F$ extends to an exact tensor functor $G \colon \cD \to \cT$ by Proposition~\ref{prop:exact-tensor-extension}, so $F=G \circ \Psi$ has type $\Psi$.
\end{proof}

\begin{lemma}
Suppose $\rL F_!$ kills $U_2$ and $U_3$. Then $F$ has type $\Theta$.
\end{lemma}

\begin{proof}
Since $\rL F_!$ kills $\cB_2$ and $\cB_3$, it kills $\bbQ_\lambda$ and $\bbI_\lambda$ for every weight $\lambda$. By Lemma~\ref{lem:PQI}, we thus have $0=F_!(\bbP_\lambda)=F(M_\lambda)$ for each non-empty $\lambda$. This forces $F$ to factor via the semi-simplifcation $\Theta \colon \cA\to\Vec$.
\end{proof}

\subsection{Local abelian envelopes} \label{ss:local}

We now recast Theorem~\ref{thm:env} into the language of local abelian envelopes. To allow for maximal generality, we enlarge the class of tensor categories under consideration. Precisely, we use categories $\cT$ satisfying the following condition:
\begin{itemize}
\item[$(\ast)$] $\cT$ is a $k$-linear rigid abelian tensor category, and $\End_{\cT}(\bbone)$ is a field.
\end{itemize}
For example, if $k$ has characteristic~0 then the 2-colimit of the sequence
\begin{displaymath}
\uRep(\fS_{-1}) \to \uRep(\fS_{-2}) \to \uRep(\fS_{-3}) \to \cdots
\end{displaymath}
satisfies $(\ast)$ but is not pre-Tannakian; here $\uRep(\fS_t)$ is Deligne's interpolation category, and the functors are restriction functors. Other examples are pre-Tannakian categories over field extensions of $k$. We note that Theorem~\ref{thm:env} (and our proof) remains valid if $\cT$ satisfies $(\ast)$.

We let $\Tens(\cT, \cT')$ denote the category of tensor functors $\cT \to \cT'$ and monoidal natural transformations. We let $\Tens^{\mathrm{ex}}(\cT, \cT')$ and $\Tens^{\mathrm{f}}(\cT, \cT')$ denote the full subcategories spanned by exact and faithful functors; these agree if $\cT$ and $\cT'$ satisfy $(\ast)$; see \cite[Theorem~2.4.1]{CEOP}.

Let $\cE$ be a rigid tensor category with $\End(\bbone)=k$. By \cite[Theorem~5.2.6]{HomKer} (see the remark in \cite[5.4.1]{HomKer}), there exists a family of faithful tensor functors $\{ \Gamma_i \colon \cA \to \cT_i \}_{i \in I}$, where each $\cT_i$ satisfies $(\ast)$, for which
\begin{displaymath}
\coprod_{i \in I} \Tens^{\mathrm{ex}}(\cT_i,\cT)\;\to\;\Tens^{\mathrm{f}}(\cA,\cT),
\end{displaymath}
is an equivalence, for every $\cT$ as in $(\ast)$. The $\Gamma_i$ are called the \defn{local abelian envelopes} of $\cE$. A priori, the endomorphism algebra $K_i$ of $\bbone$ in $\cT_i$ could be a proper field extension of $k$, even when $k$ is algebraically closed; see for example \cite[Remark~2.1.7(1)]{CEO}. Furthermore, $\cT_i$ need not be pre-Tannakian over $K_i$. When there is a single local abelian envelope, it is called \emph{the} abelian envelope. If $\cE$ is pre-Tannakian then $\cE$ is its own abelian envelope. If we are only interested in $\cT$ that are pre-Tannakian (over $k$), then we only need to consider those $\cT_i$ that are also pre-Tannakian (over $k$), since exact tensor functors are faithful here.

We now classify the local abelian envelopes of the second Delannoy category.

\begin{theorem}\label{thm:locenv}
For each tensor category $\cT$ as in $(\ast)$, there is an equivalence of categories
\begin{displaymath}
\Tens^{\mathrm{ex}}(\cC,\cT)\amalg \Tens^{\mathrm{ex}}(\cD,\cT)\;\to\;\Tens^{\mathrm{f}}(\cA,\cT),
\end{displaymath}
given by sending $F \colon \cC\to\cT$ to $F\circ \Phi$ and $G \colon \cD\to\cT$ to $G\circ\Psi$. In other words, $\cA$ has precisely two local abelian envelopes, given by $\Phi\colon \cA\to\cC$ and $\Psi:\cA\to\cD$.
\end{theorem}

\begin{proof}
This follows immediately from Theorem~\ref{thm:env}.
\end{proof}

In \cite{BaK}, a universal rigid abelian tensor category $T(\cE)$ is associated to any rigid tensor category $\cE$. Concretely, restriction along $\cE\to T(\cE)$ yields an equivalence
\begin{displaymath}
\Tens^{\mathrm{ex}}(T(\cE),\cF)\;\to\; \Tens(\cE,\cF),
\end{displaymath}
for every abelian rigid tensor category $\cF$ (without the assumption that $\End(\bbone)$ be a field).
If $\End(\bbone)=k$ in $\cE$ then its local abelian envelopes can be realised as Serre quotients of $T(\cE)$.

\begin{corollary}\label{cor:TA}
For the second Delannoy category $\cA$, we have
\begin{displaymath}
T(\cA)\;\cong\; \Vec\times \cC\times\cD.
\end{displaymath}
\end{corollary}

\begin{proof}
Denote by $Z$ the endomorphism algebra of the tensor unit in $T(\cA)$. It is von Neumann regular by \cite{BaK,Ka}, and for each $I_\alpha\in \operatorname{Spec} Z$, we have the corresponding tensor category $\cA_\alpha:=T(\cA)/\hspace{-1mm}/ I_\alpha$ as in $(\ast)$ with $\End(\bbone)=K_\alpha:=Z/I_\alpha$. It follows directly from the comparison of the theories of \cite{BaK} and \cite{HomKer}, see for instance \cite[Corollary~6.2]{Ka}, that the $\cA_\alpha$ are the local abelian envelopes of $\cA_\alpha$. Theorem~\ref{thm:locenv} and Proposition~\ref{prop:negl} show that $\operatorname{Spec} Z$ is finite (has size three), and therefore $Z=k^{3}$, and the final conclusion follows from \cite[Corollary~4.25]{Ka}.
\end{proof}

\begin{remark}
That $\Phi \colon \cA\to\cC$ is a local abelian envelope can also be seen directly, as follows.
Let $\cE$ be a $k$-linear abelian category such that $\Phi$ factors as a composite
\begin{displaymath}
\cA\xrightarrow{F}\cE\xrightarrow{G}\cC,
\end{displaymath}
for faithful $k$-linear functors $F$ and $G$, with $G$ exact. Suppose that every object in $\cE$ is a subquotient of an object in the image of $F$. We show that $G$ is an equivalence. Denote by $V_\lambda\in \cE$ the image of $F(u_\lambda\circ d_\lambda)$, for any weight $\lambda$.  By Proposition~\ref{prop:Phi-E}, $G(V_\lambda)\cong L_\lambda$, so in particular $V_\lambda$ is simple. Moreover, Proposition~\ref{prop:Phi-E} implies that there are short exact sequences
\begin{displaymath}
0\to V_{\lambda\ww} \to F(M_{\lambda\ww})\to V_\lambda\to 0\quad\mbox{and}\quad 0\to V_{\lambda} \to F(M_{\lambda\bb})\to V_{\lambda\bb}\to 0.
\end{displaymath}
However, by the existence of non-zero morphisms $F(M_{\lambda\ww})\to F(M_{\lambda\ww\bb})$ and $F(M_{\lambda\bb\ww})\to F(M_{\lambda\bb})$ (Theorem~\ref{thm:maps}), it follows that the sequences must be split. It now follows that $G$ is essentially surjective and fully faithful. If we now take $F$ to be the local envelope through which $\Phi$ factors, the above conditions hold, and so $G$ is an equivalence; this shows that $\Phi$ is a local envelope.
\end{remark}

\subsection{Universal property} \label{ss:univ}

We now describe a universal property for the tensor category $\cD$. This is closely related to the universal property for $\cA$ given in \cite{delmap}, which we first recall.

Let $\cT$ be a pre-Tannakian category. An \defn{ordered \'etale algebra} in $\cT$ is an \'etale algebra $A$ equipped with an ideal $I$ of $A \otimes A$ satisfying some axioms; see \cite[\S 7.2]{delmap}. If $A$ is an ordered \'etale algebra in $\Vec$ then $\Spec(A)$ is a totally ordered set via the relation $V(I)\subset \Spec A\times\Spec A$ . A \defn{2-Delannic algebra} (or \defn{Delannic algebra of type 2}) is an ordered \'etale algebra satisfying three numerical conditions; see \cite[\S 7.3]{delmap}. The algebra $\sA(\bR)$ in $\cA$ is 2-Delannic. Let $\Del_2(\cT)_{\rm isom}$ be the category of 2-Delannic algebras in $\cT$, where morphisms are isomorphisms of ordered \'etale algebras. The universal property for $\cA$ states that the functor
\begin{displaymath}
\Tens(\cA, \cT) \to \Del_2(\cT)_{\rm isom}, \qquad F \mapsto F(\sA(\bR))
\end{displaymath}
is an equivalence \cite[Theorem~7.13]{delmap}. In other words, $\sA(\bR)$ is the universal 2-Delannic algebra. We also show that a tensor functor $F$ as above is faithful if and only if $F(\sA(\bR))$ is not the zero algebra \cite[Proposition~7.14]{delmap}.

Let $A$ be an ordered \'etale algebra. We say that $A$ is \defn{bounded} if $A$ decomposes as a lexicographic sum $A = B \oplus \bbone$ for some ordered \'etale algebra $B$. The intuition here is that $\Spec(\bbone)$ is a maximal point in the totally ordered set $\Spec(A)$; this is literally true if $\cT=\Vec$. We say that $A$ is \defn{unbounded} if it is not bounded. We let $A^{(2)}$ be the algebra of ``increasing pairs'' of $A$. There are natural algebra homomorphisms $p_i^* \colon A \to A^{(2)}$ for $i=1,2$. See \cite[\S 7.2]{delmap} for the relevant definitions. If $A=\sA(\bR)$ then $A^{(2)}=\sA(\bR^{(2)})$ and the $p_i^*$ are dual to the projection maps $p_i \colon \bR^{(2)} \to \bR$. Also, in this case, $A$ is unbounded.

\begin{proposition}
Let $\cT$ be a pre-Tannakian category, let $F \colon \cA \to \cT$ be a faithful tensor functor, and let $A=F(\sA(\bR))$ be the corresponding 2-Delannic algebra of $\cT$. Then:
\begin{enumerate}
\item $F$ has type $\Phi$ $\iff$ $p_1^*$ is not injective $\iff$ $A$ is bounded.
\item $F$ has type $\Psi$ $\iff$ $p_1^*$ is injective $\iff$ $A$ is unbounded.
\end{enumerate}
\end{proposition}

\begin{proof}
Suppose $A$ is bounded, and write $A=B \oplus \bbone$. Then $A^{(2)} = B^{(2)} \oplus B$, and the map $p_1^*$ restricts to~0 on $\bbone \subset A$. Thus $p_1^*$ is not injective. The algebra $\Phi(\sA(\bR)) =\sC(\bR) \oplus \bbone$ is bounded, and so $p_1^*$ is not injective in this case. On the other hand, $\Psi$ is fully faithful, and so $\Psi(\sA(\bR))^{(2)}=\Psi(\sA(\bR^{(2)}))$ is a simple algebra \cite[Proposition~7.14]{delmap}, and so $p_1^*$ is injective in this case.

If $F$ has type $\Phi$ then $F=G \circ \Phi$ for an exact (and thus faithful) tensor functor $G$. Since formation of $p_1^*$ is compatible with tensor functors and $p_1^*$ is not injective for $\Phi$, it follows that $p_1^*$ is not injective for $F$. Similarly, if $F$ has type $\Psi$ then $p_1^*$ is injective. The converses follow since $F$ must have one of the two types: for example, if $p_1^*$ is injective then $F$ cannot have type $\Phi$, and so it must have type $\Psi$. We have thus established the first equivalences in both (a) and (b).

If $F$ has type $\Phi$ then $A$ is bounded, since the algebra for $\Phi$ has a maximum. We have already seen that if $A$ is bounded then $p_1^*$ is not injective, and so $F$ has type $\Phi$. This completes the proof of (a), and (b) now follows as well.
\end{proof}

The following is our universal property for $\cD$.

\begin{corollary} \label{cor:D-univ}
Let $\cT$ be a pre-Tannakian category. Giving an exact tensor functor $\cD \to \cT$ is equivalent to giving a non-zero unbounded 2-Delannic algebra in $\cT$.
\end{corollary}

We now look at some examples. Let $m$ be an odd positive integer, and let $n$ be an even non-negative integer. The lexicographic sum
\begin{displaymath}
A_{m,n} = \sC(\bR^{(m)}) \oplus \sC(\bR^{(n)})
\end{displaymath}
is a 2-Delannic algebra in $\cC$ by \cite[Proposition~7.18]{delmap} and \cite[Proposition~7.24]{delmap}. By the universal property for $\cA$, we thus have a faithful tensor functor
\begin{displaymath}
F_{m,n} \colon \cA \to \cC, \qquad F_{m,n}(\sA(\bR)) = A_{m,n}.
\end{displaymath}
In \cite[\S 8.3]{delmap}, we observed that $F_{1,0}$ and $F_{1,2}$ belong to different local envelopes; this is how we showed that $\cA$ must have at least two local envelopes. The same reasoning shows that $F_{\ell,0}$ and $F_{m,n}$ belong to different envelopes when $n>0$. It seems to be very difficult to show directly that the $F_{\ell,0}$ all belong to the same local envelope, or that the $F_{m,n}$ with $n>0$ all belong to the same local envelope; we attempted to do this using the theory of homological kernels from \cite{HomKer} directly, but made little progress. However, the above results imply these statements. Indeed, $A_{m,n}$ is bounded if and only if $n=0$, and so the $F_{m,n}$ has type $\Phi$ if $n=0$ and type $\Psi$ if $n \ne 0$.

We note one important consequence of the above examples.

\begin{corollary}
There exists an exact tensor functor $\cD \to \cC$.
\end{corollary}

\begin{proof}
Apply Corollary~\ref{cor:D-univ} with the algebra $A_{1,2}$.
\end{proof}

In fact, there are many exact tensor functors $\cD \to \cC$; each $A_{m,n}$ with $n>0$ provides a different example. The fact that $\cD$ admits an exact tensor functor to $\cC$ is not at all obvious from its definition in \S \ref{s:D}.

\end{document}